\documentclass[12pt]{amsart}
\usepackage{researchmacros}
\usepackage{ytableau}
\usepackage{shuffle}
\usepackage{comment}
\usepackage{mathabx}
\usepackage{amsxtra}
\usepackage{hyperref}
\usepackage{doc}
\usepackage{mathtools}

\makeatletter
\def\moverlay{\mathpalette\mov@rlay}
\def\mov@rlay#1#2{\leavevmode\vtop{%
   \baselineskip\z@skip \lineskiplimit-\maxdimen
   \ialign{\hfil$\m@th#1##$\hfil\cr#2\crcr}}}
\newcommand{\charfusion}[3][\mathord]{
    #1{\ifx#1\mathop\vphantom{#2}\fi
        \mathpalette\mov@rlay{#2\cr#3}
      }
    \ifx#1\mathop\expandafter\displaylimits\fi}
\makeatother

\newcommand{\cupdot}{\charfusion[\mathbin]{\cup}{\cdot}}
\newcommand{\bigcupdot}{\charfusion[\mathop]{\bigcup}{\cdot}}

\newtheorem{problem}{Problem}

\renewcommand{\top}{\mathrm{gr\,} \ideal}
\newcommand{\ideal}{I}
\newcommand{\topdeg}{\tau}

\newcommand{\la}{\lambda}
\newcommand{\x}{\mathbf{x}}
\newcommand{\y}{\mathbf{y}}
\newcommand{\rk}{\mathrm{rk}}
\newcommand{\cO}{O}
\newcommand{\ft}{\mathfrak{t}}

\newcommand{\Frob}{\mathrm{Frob}}
\newcommand{\Frobq}{\mathrm{Frob}_q}
\newcommand{\Hilbq}{\mathrm{Hilb}_q}
\newcommand{\revq}{\mathrm{rev}_q}
\newcommand{\bx}{\mathbf{x}}
\newcommand{\by}{\mathbf{y}}
\newcommand{\bz}{\mathbf{z}}
\newcommand{\sort}{\mathrm{sort}}
\newcommand{\coinv}{\mathrm{coinv}}
\newcommand{\finv}{\mathrm{inv}}

\newcommand{\OP}{\mathcal{OP}}

\newcommand{\trunc}{\mathrm{trunc}}

\newcommand{\dominatedby}{\leq_{\text{dom}}}
\newcommand{\inv}{\mathrm{inv}}
\newcommand{\rw}{\mathrm{rw}}
\newcommand{\irw}{\mathrm{irw}}

\newcommand{\CI}{\mathrm{SCI}}
\newcommand{\Stir}{\mathrm{Stir}}
\newcommand{\Fl}{\mathrm{Fl}}
\newcommand{\Par}{\mathrm{Par}}
\newcommand{\Comp}{\mathrm{Comp}}

\newcommand{\dinv}{\mathrm{dinv}}
\newcommand{\ci}{\mathrm{CI}}
\renewcommand{\dinv}{\mathrm{dinv}}
\newcommand{\fdinv}{\mathrm{dinv}}
\newcommand{\iDes}{\mathrm{iDes}}

\newcommand{\FCI}{\mathrm{SECI}}
\newcommand{\fci}{\mathrm{ECI}}

\newcommand{\float}{\varphi}
\newcommand{\standard}{\mathrm{std}}
\newcommand{\invcode}{\mathrm{invcode}}
\newcommand{\dinvcode}{\mathrm{dinvcode}}
\newcommand{\rev}{\mathrm{rev}}
\newcommand{\SSYT}{\mathrm{SSYT}}
\newcommand{\dgprime}{\mathrm{dg}'}
\newcommand{\fgl}{\mathfrak{gl}}

\newcommand{\basement}{\mathrm{B}}
\newcommand{\diagram}{\mathrm{D}}

\newcommand{\hide}[1]{\iffalse #1 \fi}

\newcommand{\st}{\,:\,}

\usepackage{mathrsfs}

\newcommand{\multibinom}[2]{
  \left[\!\genfrac{}{}{0pt}{}{#1}{#2}\!\right]
}

\newcommand{\qbinom}{\genfrac{[}{]}{0pt}{}}

\begin{document}

\title{Ordered set partitions, Garsia-Procesi modules, \\ and rank varieties}
\author{Sean T.~Griffin}
\thanks{Partially supported by NSF Grant DMS-1764012.
 }
\address{Department of Mathematics, University of Washington, Seattle, WA 98195, USA}
\email{\href{mailto:stgriff@uw.edu}{stgriff@uw.edu}}
\date{\today}

\maketitle

\begin{abstract}
We introduce a family of ideals $I_{n,\la, s}$ in $\bQ[x_1,\dots, x_n]$ for $\lambda$ a partition of $k\leq n$ and an integer $s \geq \ell(\la)$. This family contains both the Tanisaki ideals $I_\lambda$ and the ideals $I_{n,k}$ of Haglund-Rhoades-Shimozono as special cases. We study the corresponding quotient rings $R_{n,\lambda, s}$ as symmetric group modules. When $n=k$ and $s$ is arbitrary, we recover the Garsia-Procesi modules, and when $\lambda=(1^k)$ and $s=k$, we recover the generalized coinvariant algebras of Haglund-Rhoades-Shimozono. 

We give a monomial basis for $R_{n,\lambda, s}$ in terms of $(n,\lambda, s)$-staircases, unifying the monomial bases studied by Garsia-Procesi and Haglund-Rhoades-Shimozono. We realize the $S_n$-module structure of $R_{n,\lambda, s}$ in terms of an action on $(n,\la, s)$-ordered set partitions. We find a formula for the Hilbert series of $R_{n,\la, s}$ in terms of inversion and diagonal inversion statistics on a set of fillings in bijection with $(n,\la, s)$-ordered set partitions. Furthermore, we prove an expansion of the graded Frobenius characteristic of our rings into Gessel's fundamental quasisymmetric basis.

We connect our work with Eisenbud-Saltman rank varieties using results of Weyman. As an application of our results on $R_{n,\la, s}$, we give a monomial basis, Hilbert series formula, and graded Frobenius characteristic formula for the coordinate ring of the scheme-theoretic intersection of a rank variety with diagonal matrices.
\end{abstract}

\section{Introduction}\label{sec:Intro}

The goal of this paper is to unify the representation theory and combinatorics of the \emph{generalized coinvariant algebras} $R_{n,k}$ introduced by Haglund, Rhoades, and Shimozono \cite{HRS1}, and the singular cohomology rings $R_\la$ of the \emph{Springer fibers}  introduced by T.~A.~Springer \cite{Springer-TrigSum,Springer-WeylGrpReps}. On the one hand, the generalized coinvariant algebras are graded modules of the symmetric group whose combinatorics are controlled by ordered set partitions. On the other hand, the cohomology rings of Springer fibers are graded modules of the symmetric group whose combinatorics are controlled by tabloids. We introduce a family of rings $R_{n,\la, s}$ which are graded modules of the symmetric group whose combinatorics are controlled by $(n,\la, s)$-ordered set partitions. We recover the rings $R_{n,k}$ and $R_{\la}$ as special cases of our rings. Furthermore, we show that the rings $R_{n,\la, s}$ have connections to the geometry of rank varieties defined by Eisenbud and Saltman \cite{Eisenbud-Saltman}. These rank varieties are not to be confused with the rank varieties of Billey and Coskun~\cite{Billey-Coskun}. In particular, we obtain a formula for the Hilbert series and graded Frobenius characteristic of $\mathbb{Q}[\overline{O}_{n,\la}\cap \mathfrak{t}]$, the coordinate ring of the scheme-theoretic intersection of a rank variety with diagonal matrices.

Let us recall the generalized coinvariant algebras $R_{n,k}$. Fix positive integers $k\leq n$, and let $\bx_n = \{x_1,\dots, x_n\}$ be a set of $n$ commuting variables. Let $\bQ[\bx_n]$ be the polynomial ring on the variables $\bx_n$ with rational coefficients, and let $S_n$ be the symmetric group of permutations of $1,2,\dots, n$. We consider $\bQ[\bx_n]$ as a $S_n$-module, where $S_n$ acts by permuting the variables. For $1\leq d\leq n$, let $e_d(\bx_n)$ be the 
\emph{elementary symmetric polynomial of degree $d$} in the variables $\bx_n$, defined by $e_d(\bx_n) = \sum_{1\leq i_1< \cdots < i_d\leq n} x_{i_1}x_{i_2}\cdots x_{i_d}$. The ideal $I_{n,k}$ is defined to be
\begin{align}
I_{n,k} \coloneqq \langle x_1^{k},x^{k}_2,\dots, x^{k}_n, e_n(\bx_n), e_{n-1}(\bx_n),\dots, e_{n-k+1}(\bx_n)\rangle\subseteq \bQ[\bx_n].
\end{align}
Haglund, Rhoades and Shimozono defined the \emph{generalized coinvariant algebra} $R_{n,k}$ to be the quotient ring $R_{n,k} \coloneqq \bQ[\bx_n]/I_{n,k}$. Since $I_{n,k}$ is homogeneous and stable under the action of $S_n$, the quotient ring $R_{n,k}$ has the structure of a graded $S_n$-module.  When $k=n$, then it can be shown that (see \cite[Section 1]{HRS1})
\begin{align}
I_{n,n} = \langle e_1(\bx_n),\dots,e_n(\bx_n)\rangle = \langle \bQ[\bx_n]^{S_n}_+\rangle,
\end{align}
which is the ideal generated by the homogeneous positive degree invariants of $\bQ[\bx_n]$. Hence, $R_{n,n}$ is the well-known \emph{coinvariant algebra}.

We also recall some standard terminology in order to state our main results. A \emph{weak ordered set partition of $[n]$} is a partitioning of the set $[n]$ into an ordered list of subsets $B_1,\dots, B_k$, where we allow $B_i$ to be empty in general. We denote such an ordered set partition by $(B_1|B_2|\cdots | B_k)$.  Let $\OP_{n,k}$ be the collection of ordered set partitions of $[n]$ into $k$ nonempty blocks. The size of $\OP_{n,k}$ is easy to compute in terms of Stirling numbers of the second kind,
\begin{align}
|\OP_{n,k}| = k!\cdot \Stir(n,k).
\end{align}
The group $S_n$ acts on $\OP_{n,k}$ by permuting the letters $1,2,\dots, n$. Define the usual $q$-analogues of numbers, factorials, and multinomial coefficients,
\begin{align}
[n]_q \coloneqq 1 + q + \cdots + q^{n-1}, & & [n]!_q \coloneqq [n]_q[n-1]_q\cdots [1]_q,\\
\multibinom{n}{a_1,\dots,a_r}_q \coloneqq \frac{[n]!_q}{[a_1]!_q \cdots [a_r]!_q}, & & \qbinom{n}{a}_q \coloneqq \frac{[n]!_q}{[a]!_q [n-a]!_q}.
\end{align}
Let $\bx = (x_1,x_2,\dots)$ be an infinite set of variables, and let $\bQ[[\bx]]$ be the formal power series ring over the rational numbers in the variables $\bx$. Given $f\in \bQ[[\bx]][q]$, let $f = a_0 + a_1q + \cdots + a_n q^n$ be its expansion as a polynomial in $q$ with coefficients in $\bQ[[\bx]]$. Define $\revq(f) = a_n + a_{n-1}q + \cdots + a_0q^n$.

Given two sequences of nonnegative integers $(a_1,\dots, a_r)$ and $(b_1,\dots, b_s)$, a \emph{shuffle} of these two sequences is an interleaving $(c_1,\dots, c_{r+s})$ of the two sequences  such that the $a_i$ appear in order from left to right and the $b_i$ appear in order from left to right. An \emph{$(n,k)$-staircase} is a shuffle of the sequence $(0,1,\dots, k-1)$ and the sequence $((k-1)^{n-k})$ consisting of $k-1$ repeated $n-k$ many times~\cite{HRS1}.

Haglund, Rhoades, and Shimozono proved that $R_{n,k}$ has the following properties which generalize the well-known properties of the coinvariant algebra \cite{HRS1}.
\begin{itemize}[itemsep=4pt]
\item The dimension of $R_{n,k}$ is given by $\dim_\bQ(R_{n,k}) = |\OP_{n,k}| =  k!\cdot \Stir(n,k)$. The Hilbert series is 
\begin{align}
\Hilbq(R_{n,k}) &= \revq([k]!_q\cdot \Stir_q(n,k)) = \sum_{\sigma\in \OP_{n,k}} q^{\mathrm{coinv}(\sigma)},
\end{align} 
where $\Stir_q(n,k)$ is a well-known $q$-analogue of the Stirling number of the second kind, and where $\mathrm{coinv}$ is the coinversion statistic on ordered set partitions, respectively. See \cite{Wilson} for more details on ordered set partition statistics. 
\item The set of monomials
\begin{align}
\cA_{n,k} = \{x_1^{a_1}\cdots x_n^{a_n} \st (a_1,\dots,a_n) \text{ is component-wise }\leq \text{ some }(n,k)\text{-staircase}\}
\end{align}
represents a basis of $R_{n,k}$, generalizing the Artin basis of the coinvariant algebra. As a consequence, we have $|\cA_{n,k}| = |\OP_{n,k}|$.
\item As $S_n$-modules,
\begin{align}
R_{n,k} \cong_{S_n} \bQ \OP_{n,k},
\end{align}
where $\bQ \OP_{n,k}$ is the vector space over $\bQ$ whose basis is indexed by $\OP_{n,k}$ and whose $S_n$-module structure is induced from the natural action of $S_n$ on $\OP_{n,k}$ permuting the letters $1,\dots, n$.
\item The graded $S_n$-module structure of $R_{n,k}$ can be expressed in terms of the dual Hall-Littlewood functions $Q_\mu'(\bx;q)$ as follows,
\begin{align}\label{eq:HRSGrFrob}
\Frobq(R_{n,k}) = \revq\left[ \sum_{\mu} q^{\sum_{i= 1}^k (i-1)(\mu_i-1)} \multibinom{k}{m_1(\mu),\dots,m_n(\mu)}_q Q_\mu'(\bx;q)\right],
\end{align}
where the sum is over partitions $\mu$ of $n$ into $k$ parts.
See Section~\ref{sec:Background} for all undefined terminology.
\item The $S_n$-module $R_{n,k}$ is related to the \emph{Delta Conjecture} of Haglund, Remmel, and Wilson \cite{HRW}. Precisely, Haglund, Rhoades, and Shimozono~\cite{HRS1} proved that
\begin{align}
\Frobq(R_{n,k}) = (\revq\circ \omega) C_{n,k}(\bx;q),
\end{align}
where $C_{n,k}(\bx;q)$ is the expression in the Delta Conjecture at $t=0$.
\item More recently, Pawlowski and Rhoades~\cite{Pawlowski-Rhoades} proved that $R_{n,k}$ is isomorphic to the rational cohomology ring for the space of spanning line arrangements $X_{n,k}$. Their result also holds with integral coefficients.
\end{itemize}

Next, we describe certain quotient rings coming from the geometry of \emph{Springer fibers}. Let $\la \vdash n$, and let the conjugate partition be $\lambda' = (\lambda_1'\geq \lambda_{2}'\geq\dots \geq \lambda_n'\geq 0)$. Here, we are padding the conjugate partition by $0$s to make it length $n$. Let $p^n_m(\lambda) \coloneqq \lambda_n'+\lambda_{n-1}'+\dots + \lambda_{n-m+1}'$ for $1\leq m\leq n$. 
Given a subset of variables $S\subseteq \bx_n$ and a positive integer $d$, define $e_d(S)$ to be the sum over all squarefree monomials of degree $d$ in the set of variables $S$. For example, we have $e_2(\{x_1,x_3,x_5\}) = x_1x_3 + x_1x_5 + x_3x_5$. The \emph{Tanisaki ideal} $I_\la$ is defined by
\begin{align}
I_\la \coloneqq \langle e_d(S) \st S\subseteq\bx_n,\, d > |S| - p^n_{|S|}(\la)\rangle,
\end{align}
and the ring $R_\la$ is defined by
\begin{align}
R_\la \coloneqq \bQ[\bx_n]/I_\la.
\end{align}

By work of De Concini and Procesi \cite{dCP}, the ring $R_{\lambda}$ is isomorphic to the singular cohomology ring with rational coefficients of the \emph{Springer fiber} corresponding to a nilpotent matrix with Jordan type $\lambda$. This particular presentation for the cohomology ring in terms of partial elementary symmetric polynomials is due to Tanisaki \cite{Tanisaki}. When $\la = (1^n)$, the Springer fiber corresponding to $(1^n)$ is the \emph{complete flag variety} whose cohomology ring is the coinvariant algebra. Indeed, we have $R_{(1^n)} = \bQ[\bx_n]/\langle e_1(\bx_n),\dots, e_n(\bx_n)\rangle$, which is the coinvariant algebra. See~\cite{CG} for more background on Springer fibers and~\cite{Tymoczko} for more details on their combinatorial properties.

We refer to the graded $S_n$-module $R_{\la}$ as the \emph{Garsia-Procesi module} based on their seminal work in~\cite{Garsia-Procesi} on the $S_n$-module structure of $R_\la$. The ring $R_\la$ has the following properties.
\begin{itemize}[itemsep=4pt]
\item The dimension of $R_\la$ is the multinomial coefficient
\begin{align}
\dim_\bQ(R_\la) = \binom{n}{\la_1,\dots,\la_\ell},
\end{align}
where $\la = (\la_1\geq \la_2\geq \cdots \geq \la_\ell>0)$. The Hilbert series of $R_\la$ is given by the generating function for the \emph{cocharge} statistic on a certain set of words, see \cite[Remark 1.2]{Garsia-Procesi}. Alternatively, we have the following characterization of the Hilbert series which follows from work of Haglund, Haiman, and Loehr~\cite{HHL},
\begin{align}
\Hilbq(R_\la) = \sum_{\sigma} q^{\inv(\sigma)},
\end{align}
where the sum is over standard fillings of the Young diagram of $\la'$ which increase down each column, and $\inv$ is the number of attacking pairs which form an inversion of $\sigma$.
\item There is a monomial basis $\cA_\la$ of $R_\la$ which specializes to the Artin basis of the coinvariant algebra when $\la = (1^n)$. In \cite{Garsia-Procesi}, this basis is denoted by $\mathscr{B}(\la)$.
\item As $S_n$-modules, we have
\begin{align}
R_\la \cong_{S_n} \bQ (S_n / S_{\la_1}\times\cdots \times S_{\la_k}),
\end{align}
where $S_{\la_1}\times \cdots \times S_{\la_k}$ is the Young subgroup of $S_n$ permuting $1,\dots,\la_1$ among themselves, $\la_1+1,\dots,\la_1+\la_2$ among themselves, and so on. Equivalently, $R_\la$ is isomorphic to the $S_n$-module given by the action of $S_n$ on tabloids of shape $\la$.
\item The graded $S_n$-module structure of $R_\la$ is given by the reversal of the dual Hall-Littlewood function,
\begin{align}
\Frobq(R_\la) = \revq(Q_\la'(\bx;q)).
\end{align}
\item If $\la,\mu\vdash n$ such that $\la\dominatedby \mu$, we have the monotonicity property
\begin{align}
[s_\nu]\Frob_q(R_\la) \geq [s_\nu] \Frob_q(R_\mu),
\end{align}
for all $\nu\vdash n$, where $[s_\nu]f$ stands for the coefficient of $s_\nu$ in the Schur function expansion of $f$, and the inequality is a coefficient-wise comparison of two polynomials in $q$.
\end{itemize}

Fix positive integers $k\leq n$, a partition $\la$ of $k$, and an integer $s\geq \ell(\la)$, where $\ell(\la)$ is the length of $\la$. Let the conjugate of $\la$ be $\la' = (\la_1' \geq \la_{2}'\geq \cdots \geq \la_n'\geq 0)$, where we pad the conjugate partition by $0$s to make it length $n$, and define $p^n_m(\la) \coloneqq \la_n' + \la_{n-1}'+\cdots +\la_{n-m+1}'$ for $1\leq m\leq n$. We introduce the ring $R_{n,\la, s}$, defined as follows.
\begin{definition}\label{def:TheIdeal}
Define the ideal $I_{n,\la, s}$ and quotient ring $R_{n,\la, s}$ by
\begin{align}
I_{n,\la, s} &\coloneqq \langle x_i^{s} \st 1\leq i \leq n\rangle + \langle e_d(S) \st \, S\subseteq \bx_n,\, d > |S|-p^n_{|S|}(\lambda)\rangle,\\
 R_{n,\lambda, s} &\coloneqq \bQ[\bx_n]/I_{n,\lambda, s}. 
 \end{align}
\end{definition}
Since the ideal $I_{n,\la, s}$ is generated by homogeneous polynomials, it is a homogeneous ideal. Furthermore, since the generating set is closed under the action of $S_n$, the ideal $I_{n,\la, s}$ is closed under the action of $S_n$. Therefore, the quotient ring $R_{n,\la, s}$ inherits the structure of a graded $S_n$-module.
\noindent For example, the ideal $I_{6,(3,2), 3}$ is generated by the set of homogeneous polynomials
\begin{gather}
\{x_1^3,\dots,x_6^3\}\cup\{ e_2(\bx_{6}),e_3(\bx_6),e_4(\bx_6),e_5(\bx_6),e_6(\bx_6)\} \cup\{e_3(S)\,|\, S\subseteq \bx_6, |S|=5\} \\[1em]
\cup\{e_4(S)\,|\, S\subseteq \bx_6, |S|=5\}
\cup\{e_5(S)\,|\, S\subseteq \bx_6, |S|=5\}
\cup\{e_4(S)\,|\, S\subseteq \bx_6, |S|=4\},
\end{gather}
which is closed under the action of $S_6$.

The generalized coinvariant algebras $R_{n,k}$ and the rings $R_\la$ are special cases of the rings $R_{n,\la, s}$. We have
\begin{align}
R_{n,k} & = R_{n,\,(1^k),\, k}&&\text{ for }k\leq n \label{eq:RnkSpecial},  \\
R_{\la} &= R_{n,\,\la,\,\ell(\la)} &&\text{ for }\la\vdash n\label{eq:RlaSpecial},
\end{align}
where~\eqref{eq:RnkSpecial} follows from Definition~\ref{def:TheIdeal}. See Remark~\ref{rmk:Specialization} for the justification of~\eqref{eq:RlaSpecial}. As a bonus, we also have $R_{n, k, s} = R_{n,(1^s),k}$, where $R_{n, k, s}$ is the ring defined in \cite[Section 6]{HRS1}.

Let a \emph{$(n,\la, s)$-ordered set partition} be a weak ordered set partition $(B_1|\cdots|B_s)$ of $[n]$ into $s$ blocks such that $|B_i|\geq \lambda_i$ for $i\leq \ell(\la)$.  Here, we allow $B_i$ to be empty for $\ell(\la) < i\leq s$. Let $\mathcal{OP}_{n,\lambda, s}$ be the set of $(n,\la, s)$-ordered set partitions. The group $S_n$ acts on $\OP_{n,\la, s}$ by permuting the letters $1,2,\dots, n$.

We prove the following properties of $R_{n,\la, s}$, generalizing many of the properties of $R_{n,k}$ and $R_\la$. All terminology not defined here is defined in Section~\ref{sec:Background} and in the section corresponding to each theorem. 
\begin{itemize}[itemsep=4pt]
\item The dimension of $R_{n,\la, s}$ is given by $\dim_\bQ(R_{n,\la, s}) = |\OP_{n,\la, s}|$ (see Theorem~\ref{thm:UngradedFrobChar}). We give a formula for the Hilbert polynomial (see Corollary~\ref{cor:FakeInvFormulaHilb}),
\begin{align}\label{eq:HilbEquation}
\Hilbq(R_{n,\la, s})  =  \sum_{\float \in \FCI_{n,\la, s}} q^{\fdinv(\float)}.
\end{align}
The indexing set $\FCI_{n,\la, s}$ in the sum on the right-hand side of~\eqref{eq:HilbEquation} is a set of \emph{standard extended column-increasing fillings}, which is in bijection with $\OP_{n,\la, s}$. See Subsection~\ref{subsec:Fcoinv} for the definition of a standard extended column-increasing filling.
\item During the preparation of this article, an alternative formula for the Hilbert series was obtained by Rhoades, Yu, and Zhao on ordered set partitions~\cite[Corollary 4.9]{Rhoades-Yu-Zhao} involving a statistic they call $\mathrm{coinv}$. They also characterize the harmonic spaces corresponding to the ideals $I_{n,\la, s}$ introduced in this paper. Motivated by their work, we define a statistic $\inv$ on $\FCI_{n,\la, s}$ which is similar to, but not the same as, their $\mathrm{coinv}$ statistic. We prove that the $\mathrm{inv}$ statistic also gives a formula for the Hilbert series (see Corollary~\ref{cor:FakeInvFormulaHilb}),
\begin{align}\label{eq:HilbEquation2}
\Hilbq(R_{n,\la, s})  =  \sum_{\float \in \FCI_{n,\la, s}} q^{\finv(\float)}.
\end{align}
The formula~\eqref{eq:HilbEquation2} follows naturally as a corollary of our graded Frobenius characteristic formula. Since~\cite{Rhoades-Yu-Zhao} uses our work to prove their results, we have been careful to give independent proofs.

\item In Section~\ref{sec:genset}, we define an \emph{$(n,\lambda, s)$-staircase} as a shuffle of a certain set of compositions depending on $n$, $\lambda$, and $s$. We have that
\begin{align}
\cA_{n,\la, s} = \{x_1^{a_1}\cdots x_n^{a_n} \st (a_1,\dots,a_n) \text{ is component-wise }\leq \text{ some }(n,\la, s)\text{-staircase} \}
\end{align}
represents a monomial basis of $R_{n,\la, s}$ (see Theorem~\ref{thm:MonomialBasis}).
\item We have the following chain of equalities and $S_n$-module isomorphisms (see Theorem~\ref{thm:TopIdealEquality}, Corollary~\ref{cor:SymmModuleIso}, and Theorem~\ref{thm:UngradedFrobChar})
\begin{align}
R_{n,\la, s} = \frac{\bQ[\bx_n]}{\top(X_{n,\la, s})} \cong_{S_n} \bQ X_{n,\la, s} \cong_{S_n} \bQ \OP_{n,\la, s},
\end{align}
where $X_{n,\la, s}$ is a finite set of points in $\bQ^n$ which is stable under the $S_n$-action permuting coordinates.
\item The graded Frobenius characteristic can be expressed as a sum of monomials (see Theorem~\ref{thm:MonomialTheorem}), or equivalently in terms of Gessel's fundamental quasisymmetric functions (see Corollary~\ref{cor:FBasisExpansion}),
\begin{align}\label{eq:MonomialExpansion}
\Frobq(R_{n,\la, s}) &= \sum_{\float\in \fci_{n,\la, s}} q^{\finv(\float)} \bx^{\float} = \sum_{\float\in \FCI_{n,\la, s}} q^{\finv(\float)} F_{n, \iDes(\rw(\float))}(\bx),\\
& = \sum_{\float\in \fci_{n,\la, s}} q^{\fdinv(\float)} \bx^{\float} = \sum_{\float \in \FCI_{n,\la, s}} q^{\fdinv(\float)} F_{n,\iDes(\rw(\float))}(\bx), 
\end{align}
where $\bx^\float$ is the monomial whose powers record the number of occurrences of each label in $\float$.

\item In forthcoming work, we prove $\Frobq(R_{n,\la, s})$ can be expressed in terms of the dual Hall-Littlewood functions,
\begin{align}\label{eq:IntroGradedFrobFormula}
\Frobq(R_{n,\lambda, s}) = \revq\left[ \sum_{\substack{\mu\in \Par(n, s),\\ \mu\supseteq\lambda}} \, q^{n(\mu,\lambda)} \, \prod_{i\geq 0}\multibinom{\mu_i' - \la_{i+1}'}{\mu_i'-\mu_{i+1}'}_q \, Q_{\mu}'(\bx;q)\right],
\end{align}
where $\mu_0' \coloneqq s$, and $n(\mu,\la) \coloneqq \sum_{i\geq 1} \binom{\mu_i' - \la_i'}{2}$. The proof of this formula has been moved to a forthcoming article, since the techniques for proving it are different than the techniques used to prove~\eqref{eq:MonomialExpansion}.
\item Let $h\leq k\leq n$ be positive integers, let $\la\in \Par(h, s)$, and let $\mu\in \Par(k, s)$ such that either $h=k$ and $\la\dominatedby\mu$ or $h<k$ and $\la\subseteq \mu$. We have the monotonicity property (see Theorem~\ref{thm:Monotonicity})
\begin{align}
[s_{\nu}] \Frobq(R_{n,\la, s})\geq [s_\nu] \Frobq(R_{n,\mu, s}),
\end{align}
for all $\nu\vdash n$. 
\end{itemize}

One of our main tools for proving these results is Theorem~\ref{thm:TopIdealEquality} which identifies the ring $R_{n,\la, s}$ as the associated graded ring of the coordinate ring of a finite set of points in $\bQ^n$. This identification extends results of Garsia and Procesi \cite[Proposition 3.1, Remark 3.1]{Garsia-Procesi} and Haglund, Rhoades, and Shimozono \cite[Equation 4.28]{HRS1}. See also the work of Kraft~\cite[Proof of Proposition 4]{Kraft} from 1981 for a proof in the case of $R_\la$ using associated cones.

Haglund, Rhoades, and Shimozono use Gr\"obner bases to prove their results. In particular, they find Gr\"obner bases of the ideals $I_{n,k}$ in terms of \emph{Demazure characters}. To the author's knowledge, such explicit Gr\"obner bases for the ideals $I_\la$ are not known. Therefore, different techniques are required to prove our results. Indeed, we prove the above results without the use of Gr\"obner bases using techniques similar to those of Garsia and Procesi. In particular, we use a straightening algorithm which expresses any element of $R_{n,\la, s}$ in terms of our monomial basis $\cA_{n,\la, s}$. It is an open problem to find explicit Gr\"obner bases for the ideals $I_{n,\la, s}$.

Let $\la\vdash n$, and let $\cO_\la$ be the conjugacy class of nilpotent $n\times n$ matrices over $\bQ$ whose Jordan blocks are of sizes recorded by $\la$. Let $\overline{\cO}_{\la}$ be its closure in the space of $n\times n$ matrices. Let $\mathfrak{t}$ be the set of diagonal matrices. De Concini and Procesi~\cite{dCP}, extending work of Kostant~\cite{Kostant} on the coinvariant algebras, proved that $R_\la$ is isomorphic to the coordinate ring of the scheme-theoretic intersection $\overline{\cO}_{\la'}\cap \mathfrak{t}$.

We connect the rings $R_{n,\la, s}$ to a generalization of these scheme-theoretic intersections as follows. Given $k\leq n$ and $\la\vdash k$, define $I_{n,\la}$ to be the ideal
\begin{align}
I_{n,\la} \coloneqq \langle e_d(S) \st S\subseteq \bx_n,\, d > |S|-p^n_{|S|}(\lambda)\rangle.
\end{align}
 Define the quotient ring $R_{n,\la} \coloneqq \bQ[\bx_n]/I_{n,\la}$. When $k<n$, the ring $R_{n,\la}$ has positive Krull dimension, and hence it is infinite-dimensional as a $\bQ$-vector space. Observe that for fixed $n$, $\la$, and $d$, the $d$th degree components of $R_{n,\la, s}$ stabilize to the $d$th degree component of $R_{n,\la}$ as $s\to \infty$.
 
In Section~\ref{sec:Geometry}, using work of Weyman \cite{Weyman} we prove that $R_{n,\la}$ is isomorphic to the coordinate ring of the scheme-theoretic intersection $\overline{\cO}_{n,\la'}\cap \mathfrak{t}$, where $\overline{\cO}_{n,\la}$ is the rank variety of Eisenbud and Saltman~\cite{Eisenbud-Saltman} (see Corollary~\ref{cor:WeymanCorollary}),
\begin{align}
R_{n,\la} \cong \bQ[\overline{\cO}_{n,\la'}\cap \mathfrak{t}].
\end{align}
We use our results on the finite-dimensional rings $R_{n,\la, s}$ to find monomial bases of these coordinate rings by allowing $s$ to approach infinity in our combinatorial formulas. We also prove the following formula for the graded Frobenius characteristic of these coordinate rings (see Theorem~\ref{thm:RankVarFrob} and Corollary~\ref{cor:RankVarFrobCor}),
\begin{align}
\Frobq(\bQ[\overline{\cO}_{n,\la'}\cap \mathfrak{t}]) =  \sum_{\float\in \fci_{n,\la}} q^{\finv(\float)} \bx^\float = \sum_{\float \in \FCI_{n,\la}} q^{\finv(\float)} F_{n,\iDes(\rw(\float))}(\bx).
\end{align}

In \cite[Question 5.3.1]{Church-Ellenberg-Farb}, Church, Ellenberg, and Farb ask for the dimensions of the graded pieces of the ring of polynomial functions on a certain scheme supported on a rank variety. We solve a related question by giving a formula for the Hilbert series of $\bQ[\overline{O}_{n,\la}\cap \ft]$ (see Corollary~\ref{cor:RankHilb}),
\begin{align}
\Hilbq(\bQ[\overline{\cO}_{n,\la'}\cap \mathfrak{t}]) =  \sum_{\float\in \FCI_{n,\la}} q^{\finv(\float)}.
\end{align}

This article is the full version of the extended abstract~\cite{Griffin-GPMod-FPSAC}.

The rest of the paper is structured as follows. In Section~\ref{sec:Background}, we review relevant definitions and results from the literature. In Section~\ref{sec:genset}, we construct a monomial basis of $R_{n,\la, s}$ and prove that $R_{n,\la, s}$ is isomorphic to $\bQ \OP_{n,\la, s}$ as an $S_n$-module. In Section~\ref{sec:Skew}, we give algebraic tools for analyzing the graded Frobenius characteristic of $R_{n,\la, s}$, including a skewing formula. We also prove that the rings $R_{n,\la, s}$ fit into certain exact sequences. In Section~\ref{sec:Statistics}, we define the inversion and diagonal inversion statistics and use them to provide formulas for the graded Frobenius characteristic of $R_{n,\la, s}$. In Section~\ref{sec:Geometry}, we relate the rings $R_{n,\la}$ to the geometry of the rank varieties of Eisenbud and Saltman. We then prove our formula for the graded Frobenius characteristic of the scheme-theoretic intersection of a rank variety with diagonal matrices. 
In Section~\ref{sec:FinalRemarks}, we give final remarks and state some open problems related to the rings $R_{n,\la, s}$ and $R_{n,\la}$.

\section{Background}\label{sec:Background}

\subsection{Partitions, compositions, and permutations}

A \emph{partition of $n$ into $\ell$ parts} is a weakly decreasing sequence of positive integers $\la = (\la_1\geq \la_2\geq \cdots \geq \la_\ell > 0)$ such that $\sum_{i=1}^\ell \la_i = n$. We sometimes write $\la\vdash n$ or $|\la| = n$ to denote that $\la$ has size $n$. Let $\ell(\la)\coloneqq \ell$ be the \emph{length} of $\la$. Let $\Par(n, s)$ be the set of partitions of $n$ into at most $s$ parts. For example, we have
\begin{align}
\Par(4,2) = \{(4),(3,1),(2,2)\}.
\end{align}
Given integers $a$ and $b$ with $b\geq 0$, we denote by $(a^b)$ the partition $(a, a,\dots, a)$ where $a$ appears $b$ many times. When $b=0$, then $(a^b)$ is the empty partition.

Given two partitions $\la$ and $\mu$ of $n$, we say that $\la$ \emph{is dominated by} $\mu$, denoted by $\la\dominatedby \mu$, if and only if for all $i\leq \ell(\la)$, we have that $\la_1+\dots +\la_i\leq \mu_1+\dots+\mu_i$. The partial ordering $\dominatedby$ on the partitions of $n$ is called \emph{dominance order}.
Let $n(\la)$ be the statistic
\begin{align}\label{eq:nOfLa}
n(\la) \coloneqq \sum_{i\geq 1} (i-1)\la_i = \sum_{i=1}^{\la_1} \binom{\la_i'}{2}.
\end{align}

For $\la\vdash n$, we draw its \emph{Young diagram} as rows of boxes, called \emph{cells}, with $\la_i$ cells in row $i$. We follow the French convention where the rows are numbered from bottom to top. We also number the columns of the diagram from left to right. Let $\la'$ be the \emph{conjugate partition} of $\la$, which is the partition of $n$ whose $i$th entry records the number of cells in the $i$th column of the Young diagram of $\la$. See Figure~\ref{fig:Partition} for the Young diagrams of $\la = (3,2,1,1,1)$ and its conjugate $\la' = (5,2,1)$.

\begin{figure}[t]
\centering
\includegraphics[scale=1.5]{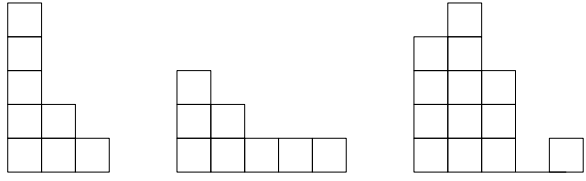}
\caption{From left to right, the Young diagram of $\la = (3,2,1,1,1)$, the Young diagram of $\la' = (5,2,1)$, and the conjugate diagram $\dgprime(4,5,3,0,1)$. \label{fig:Partition}}
\end{figure}

Given $\la\vdash n$ and an index $0\leq i < \ell(\la)$, let $\la^{(i)}$ be the partition of $k-1$ obtained by sorting the parts of
\begin{align}\label{eq:ReductionDef}
(\la_1, \dots,\la_i, \la_{i+1}-1, \la_{i+2},\dots, \la_{\ell(\la)})
\end{align}
into decreasing order. We say that $\la^{(i)}$ is the \emph{$i$th reduction} of $\la$. Note that our $\la^{(i)}$ is the same as $\la^{(i+1)}$ defined in \cite{Garsia-Procesi}.
For example, if $\la = (6,6,4,4,1)$, then $\la^{(2)} = (6,6,4,3,1)$.
Observe that if $j$ is maximal such that $i < \la_j'$, then
\begin{align}
\la^{(i)} = (\la_1,\dots,\la_{\la_j'}-1,\dots,\la_\ell).
\end{align}

Let $\Comp(n, s)$ be the set of weak compositions of $n$ into $s$ nonnegative parts. 
Given $\alpha\in \Comp(n, s)$, the \emph{conjugate diagram}, denoted by $\dgprime(\alpha)$, is the diagram consisting of $\alpha_i$ cells in the $i$th column from the left.
For $\la\in \Par(n, s)$, then $\dgprime(\la)$ is the Young diagram of $\la'$.
If $\alpha_i=0$ for some $i$, we simply draw a horizontal bar at the bottom of column $i$ to signify an empty column of cells. Index the cells $(i,j)$ of the diagram of $\alpha$ in Cartesian coordinates starting with $(1,1)$ in the lower left corner, so that $(i,j)$ is the cell in the $i$th column and $j$th row.

Let $\sort(\alpha)\in \Par(n, s)$ be the partition obtained by sorting the entries of $\alpha$ from greatest to smallest and then deleting trailing $0$s. We denote by $\trunc(\alpha) \coloneqq (\alpha_1,\dots,\alpha_{n-1})$ the \emph{truncation} of $\alpha$. Let $\rev(\alpha) \coloneqq (\alpha_s, \alpha_{s-1}, \dots, \alpha_1)$ be the reversed composition.
Given $\alpha \in \Comp(n, s)$ and $\beta\in \Comp(m, s')$, we denote by $\alpha\cdot \beta \coloneqq (\alpha_1,\dots,\alpha_s,\beta_1,\dots,\beta_{s'})$ the \emph{concatenation} of $\alpha$ and $\beta$. 

If $\alpha\in \Comp(n, s)$ and $\beta\in \Comp(m, s)$, we say that $\alpha$ \emph{is contained in} $\beta$, denoted by $\alpha\subseteq \beta$, if $\alpha_i\leq \beta_i$ for all $i\leq s$. If $\alpha\subseteq \beta$, let $\dgprime(\beta)/\dgprime(\alpha)$ be the set difference of the conjugate diagrams of $\beta$ and $\alpha$. When $\alpha = \la\subseteq \beta = \mu$ are partitions, then this is the \emph{skew diagram} $\mu'/\la'$.

Denote the symmetric group of permutations of $1,\dots, n$ by $S_n$. For $1\leq i < n$, let $s_i\in S_n$ be the adjacent transposition $s_i = (i,i+1)$. For a weak composition of length $n$, let $s_i(\alpha)$ be the composition obtained by swapping the $i$th and $(i+1)$th entries of $\alpha$. For any word on positive integers $w\in \mathbb{N}^n$, let $\standard(w)$ be the \emph{standardization} of $w$, which is the unique permutation in $S_n$ obtained by replacing the entries of $w$ with $1, 2,\dots, n$, keeping the same relative order, where repeated letters are considered as increasing from left to right. Given $\pi\in S_n$, the \emph{descent set} of $\pi$ is $\Des(\pi) \coloneqq \{i\st 1\leq i\leq n-1,\, \pi_i > \pi_{i+1}\}$. The \emph{inverse descent set} of $\pi$ is defined by $\iDes(\pi) \coloneqq \Des(\pi^{-1})$.

\subsection{Symmetric functions}\label{subsec:SymmFunctions}

In this subsection, we introduce our notation for symmetric functions. See \cite{Macdonald} for a comprehensive treatment of symmetric functions.

Let $\bx = \{x_1,x_2,\dots\}$ be an infinite set of variables. Let $\Lambda_{\bQ(q)}$ be the ring of symmetric functions with coefficients in the field $\bQ(q)$ of rational functions in $q$. For $\la\vdash n$, let $m_\la(\bx)$, $h_\la(\bx)$, $e_\la(\bx)$, and $s_\la(\bx)$ denote the \emph{monomial, complete homogeneous, elementary}, and \emph{Schur symmetric functions}, respectively. Let $\omega$ be the involutory automorphism of $\Lambda_{\bQ(q)}$ such that
\begin{align}
\omega(e_\la) = h_\la, \quad \omega(h_\la) = e_\la,\quad \omega(s_\la) = s_{\la'}.
\end{align}
Let $\langle \cdot,\cdot \rangle$ be the \emph{Hall inner product} on symmetric functions with the property that
\begin{align}
\langle s_\la(\bx),s_\mu(\bx)\rangle = \delta_{\la,\mu},
\end{align}
where $\delta_{\la,\mu}$ takes the value $1$ if $\la=\mu$, and $0$ otherwise. Given a symmetric function $F(\bx)$, let $F(\bx)^\perp$ be the linear operator on $\Lambda$ which is adjoint to multiplication by $F(\bx)$ with respect to the Hall inner product. Precisely, given a symmetric function $G(\bx)$, then $F(\bx)^\perp G(\bx)$ is the unique symmetric function such that
\begin{align}
\langle F(\bx)^\perp G(\bx),H(\bx)\rangle = \langle G(\bx), F(\bx) H(\bx)\rangle
\end{align}
for all symmetric functions $H(\bx)$. We primarily work with the operators $e_j(\bx)^\perp$, which we refer to as the \emph{$j$th skewing operator}.

Symmetric functions have a close connection to representations of the symmetric group via the Frobenius characteristic map. Given $\la\vdash n$, let $S^\la$ be the irreducible $S_n$-module indexed by $\la$. Given a finite-dimensional vector space $V$ over $\bQ$ which has the structure of a $S_n$-module, it decomposes as a direct sum as
\begin{align}
V\cong \bigoplus_{\la\vdash n}(S^\la)^{c_\la}
\end{align}
for some nonnegative integers $c_\la$. The \emph{Frobenius characteristic} of $V$ is defined to be the symmetric function
\begin{align}
\Frob(V) = \sum_{\la\vdash n} c_\la s_\la(\bx).
\end{align}
Given a graded $S_n$-module $V = \bigoplus_{i = 0}^m V_i$ with finite-dimensional direct summands $V_i$, the \emph{graded Frobenius characteristic} of $V$ is defined to be
\begin{align}
\Frobq(V) = \sum_{i=0}^m \Frob(V_i)q^i \in \Lambda_{\bQ(q)}.
\end{align}
The \emph{Hilbert series} of $V$ is defined to be
\begin{align}
\Hilbq(V) = \sum_{i=0}^m \dim_\bQ(V_i)q^i.
\end{align}
The graded Frobenius characteristic and the Hilbert series of $V$ are related by
\begin{align}
\Hilbq(V) = \langle h_{(1^n)},\Frobq(V)\rangle = [x_1\cdots x_n] \Frobq(V),\label{eq:HilbFromFrobq}
\end{align}
where if $F(\bx;q)\in \Lambda_{\bQ(q)}$, then $[x_1\cdots x_n] F(\bx;q)$ denotes the coefficient of the monomial $x_1\cdots x_n$ in $F(\bx;q)$.


We sometimes expand symmetric functions in terms of \emph{Gessel's fundamental quasisymmetric functions}, defined as follows. Given a subset $D\subseteq [n-1]$, $F_{n,D}(\bx)$ is defined by
\begin{align}\label{eq:DefOfF}
F_{n,D}(\bx) \coloneqq \sum x_{a_1}x_{a_2}\cdots x_{a_n},
\end{align}
where the sum is over all sequences of positive integers $1\leq a_1\leq a_2\leq \cdots \leq a_n$ such that $a_i < a_{i+1}$ for all $i\in D$. Alternatively, Gessel's fundamental quasisymmetric function \cite{GesselQSym} can be written as
\begin{align}
F_{n,D}(\bx) = \sum_{\substack{w\in \mathbb{N}^n,\\ \standard(w) = \pi}} \prod_{i\geq 1} x_i^{\# i\text{'s in }w},\label{eq:GesselAlt}
\end{align}
where $\pi\in S_n$ is a fixed permutation such that $\iDes(\pi) = D$. We also have the \emph{monomial quasisymmetric function},
\begin{align}\label{eq:DefOfM}
M_{n,D}(\bx) \coloneqq \sum x_{a_1}x_{a_2}\cdots x_{a_n},
\end{align}
where the sum is over all sequences of positive integers $1\leq a_1\leq a_2\leq \cdots \leq a_n$ such that $a_i < a_{i+1}$ for all $i\in D$ and $a_i = a_{i+1}$ for all $i\in [n-1]\setminus D$.

For $1\leq j\leq n$, let $S_{n-j} \times S_{j}\subseteq S_n$ be the subgroup of $S_n$ of permutations which permute $1,\dots, n-j$ among themselves and permute $n-j+1,\dots,n$ among themselves. Let $\bQ S_n$ be the group algebra of $S_n$, and let $\epsilon_j\in \bQ S_n$ be the idempotent element
\begin{align}
\epsilon_j = \frac{1}{j!}\sum_{\pi \in S_{\{n-j+1,\dots, n\}}} \mathrm{sgn}(\pi) \pi.
\end{align}
If $V$ is an $S_n$-module, then $\epsilon_j V$ is an $S_{n-j}$-module. It is well known (see e.g.\ \cite[Equation 6.20]{HRS1}) that
\begin{align}
\Frobq(\epsilon_j V) = e_j(\bx)^{\perp} \Frobq(V).\label{eq:AntisymmFrob}
\end{align}
Furthermore, the symmetric function $\Frobq(V)$ is uniquely defined by its images under the $e_j(\bx)^\perp$ operators for $j>1$.
\begin{lemma}[\cite{Garsia-Procesi},\cite{HRS1} Lemma 3.6]\label{lem:AntisymmIdentification}
Let $F(\bx)$ and $G(\bx)$ be symmetric functions with equal constant terms. We have that $F(\bx)=G(\bx)$ if and only if $e_j(\bx)^{\perp}F(\bx) = e_j(\bx)^{\perp} G(\bx)$ for all $j\geq 1$.
\end{lemma}
We need the following basic fact about the $\epsilon_j$ idempotent elements. We include a proof for the sake of completeness.
\begin{lemma}\label{lem:EjLemma}
Given a finite set $Y$ with an $S_n$-action, let $y_1,\dots, y_m$ be a distinct set of representatives of the $S_{\{n-j+1,\dots,n\}}$-orbits of $Y$. Then $\dim_\bQ(\epsilon_j \bQ Y)$ is equal to the number of $y_i$ which have trivial $S_{\{n-j+1,\dots,n\}}$-stabilizer.
\end{lemma}

\begin{proof}
 We have a direct sum decomposition
\begin{align}
\epsilon_j\bQ Y = \bigoplus_{i=1}^m \epsilon_j\bQ S_{\{n-j+1,\dots,n\}} y_i.
\end{align}
Observe that if $y$ and $y'$ are in the same $S_{\{n-j+1,\dots, n\}}$-orbit, then $\epsilon_j y = \pm \epsilon_j y'$. Furthermore, if $y$ has nontrivial stabilizer under the $S_{\{n-j+1,\dots, n\}}$-action, then $\epsilon_j y = 0$. Hence, $\epsilon_j\bQ S_{\{n-j+1,\dots,n\}} y_i$ is either one-dimensional if $y_i$ has trivial $S_{\{n-j+1,\dots,n\}}$-stabilizer or zero-dimensional otherwise. Hence, $\dim_\bQ(\epsilon_j\bQ Y)$ is equal to the number of $y_i$ which have trivial $S_{\{n-j+1,\dots,n\}}$-stabilizer.
\end{proof}

\subsection{Hall-Littlewood symmetric functions}\label{subsec:HLFunctions}

The algebra $\Lambda_{\bQ(q)}$ of symmetric functions has a basis given by the \emph{Hall-Littlewood symmetric functions} $P_\la(\bx;q)$ which have the property that
\begin{align}
s_\la(\bx) = \sum_{\mu\vdash n}K_{\la,\mu}(q) P_{\mu}(\bx;q),
\end{align}
where $K_{\la,\mu}(q)$ is the \emph{Kostka-Foulkes polynomial}, see~\cite{Macdonald}. The \emph{dual Hall-Littlewood symmetric functions} $Q_\la'(\bx;q)$ are given by
\begin{align}
Q_\mu'(\bx;q) = \sum_{\la\vdash n} K_{\la,\mu}(q) s_\la(\bx).
\end{align}
The degree of $Q_\la'(\bx;q)$ is given by $\deg(Q_\la'(\bx;q)) = n(\la)$, defined in~\eqref{eq:nOfLa}. The reversal of these symmetric functions are sometimes denoted by $\widetilde{H}_\la(\bx;q) \coloneqq \revq(Q_\la'(\bx;q)) = q^{n(\la)} Q_\la'(\bx; 1/q)$.
\begin{theorem}[\cite{Garsia-Procesi,Springer-TrigSum}]\label{thm:GPTheorem}
The Frobenius characteristic of the $S_n$-module $R_\la$ is the reversal of the dual Hall-Littlewood symmetric function,
\begin{align}
\Frobq(R_\la) = \widetilde{H}_{\la}(\bx;q) = \revq(Q_\la'(\bx;q)).
\end{align}
\end{theorem}

Next, we recall the $t=0$ specialization of the Haglund-Haiman-Loehr formula for Macdonald polynomials, which gives an expansion for $\widetilde H_\la(\bx;q) = \revq(Q'_\la(\bx;q))$ in terms of inversions in labelings of the Young diagram of $\la'$. Given $\alpha\in \Comp(n, s)$, a pair of cells $((i,j),(i',j'))$ in $\dgprime(\alpha)$ are said to be an \emph{attacking pair} if either $j=j'$ and $i < i'$, or if $j = j'+1$ and $i>i'$. See Figure~\ref{fig:Attacking} for two examples of attacking cells in $\dgprime(4,5,3,0,1)$, where the attacking cells are indicated by dots.

\begin{figure}[t]
\centering
\includegraphics[scale=0.75]{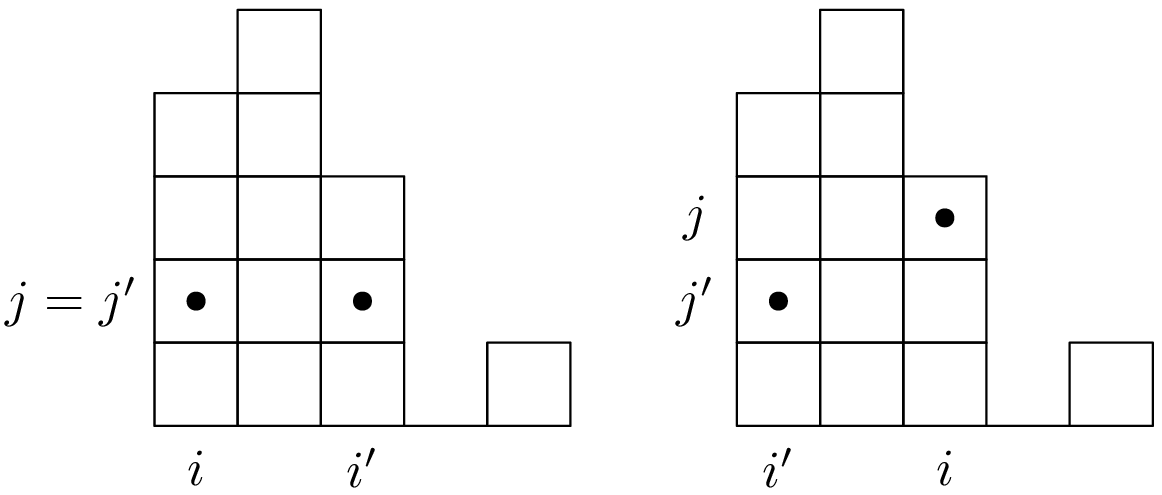}
\caption{Two examples of attacking pairs in $\dgprime(4,5,2,0,3)$.\label{fig:Attacking}}
\end{figure}

Let $\sigma$ be a filling of $\dgprime(\alpha)$. Given a cell $(i,j)$ of the diagram, let $\sigma_{i,j}$ be the label of the cell $(i,j)$ in $\sigma$. The \emph{reading word} of $\sigma$, denoted $\mathrm{rw(\sigma)}$, is the word obtained by reading the filling across each row from left to right, starting with the top-most row and ending with row $1$. We say that the corresponding ordering of the cells of $\dgprime(\alpha)$ is the \emph{reading order} of the diagram.  See Figure~\ref{fig:ReadingOrder} for an example of a filling $\sigma$ of $\dgprime(4,5,3,0,1)$ with reading word $\rw(\sigma) = 1\, 4\, 2\, 5\, 8\, 3\, 6\, 9\, 10\, 7\, 11\, 12\, 13$.

\begin{figure}[t]
\centering
\includegraphics[scale=0.75]{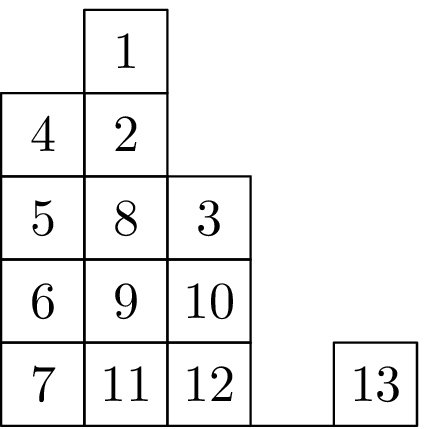}
\caption{A column-increasing filling $\sigma\in \CI_{10,\alpha,5}$ where $\alpha = (4,5,3,0,1)$.\label{fig:ReadingOrder}}
\end{figure}

A filling $\sigma$ with positive integers is said to be \emph{column-increasing} if the labeling weakly increases down each column. Let a \emph{standard} filling be a filling with the integers $1,\dots, n$, such that each integer $1\leq i\leq n$ is used exactly once. Given $\la\in \Par(n, s)$, let $\ci_{n,\la, s}$ be the set of column-increasing fillings of $\la'$, and let $\CI_{n,\la, s}$ be the subset of $\ci_{n,\la, s}$ of standard column-increasing fillings. Similarly, given a composition $\alpha\in \Comp(n, s)$, let $\ci_{n,\alpha, s}$ be the set of column-increasing fillings of $\dgprime(\alpha)$, and let $\CI_{n,\alpha, s}$ be the subset of $\ci_{n,\alpha, s}$ of standard column-increasing fillings. See Figure~\ref{fig:SCIFillings} for all fillings in $\CI_{5,\la, 2}$ where $\la = (3,2)$.

\begin{figure}[t]
\centering
\includegraphics[scale=0.75]{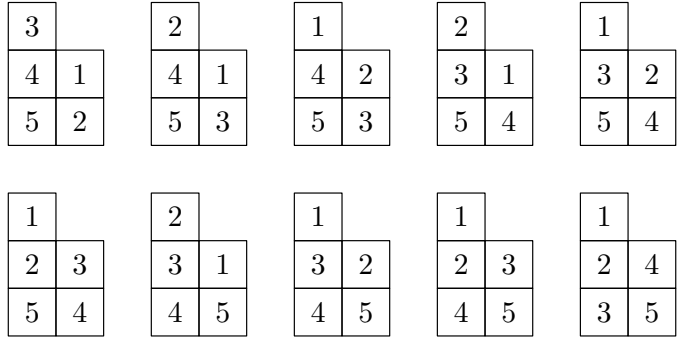}
\caption{All standard column-increasing fillings in $\CI_{5,\la, 2}$ where $\la = (3,2)$.\label{fig:SCIFillings}}
\end{figure}

 Define a \emph{diagonal inversion} of $\sigma$ to be an attacking pair $((i,j),(i',j'))$ of cells of $\dgprime(\alpha)$ with $(i,j)$ appearing earlier in the reading order, such that $\sigma_{i,j} > \sigma_{i',j'}$. Denote by $\dinv(\sigma)$ the number of diagonal inversions of $\sigma$. Letting $\sigma$ be the filling in Figure~\ref{fig:ReadingOrder}, then the cells $(1,3)$ and $(3,3)$ form a diagonal inversion in $\sigma$, and the cells $(2,3)$ and $(1,2)$ also form a diagonal inversion in $\sigma$. The reader can check that $\dinv(\sigma) = 6$. 
 
The diagonal inversion statistic defined above appears in the following corollary of the beautiful fundamental quasisymmetric function expansion for Macdonald symmetric functions proven by Haglund, Haiman, and Loehr \cite{HHL}. Setting $t=0$ in \cite[Equation 36]{HHL}, we have the following expansion for the reversal of the dual Hall-Littlewood function in our notation.
\begin{corollary}
For $\la\vdash n$, we have
\begin{align}\label{eq:HHLFormula}
\Frobq(R_{\la}) = \revq(Q_\la'(\mathbf{x};q)) = \sum_{\sigma \in \CI_{n,\la,\ell(\la)}} q^{\dinv(\sigma)} F_{n,\mathrm{iDes}(\rw(\sigma))}(\bx).
\end{align}
\end{corollary}

A family of symmetric functions generalizing the dual Hall-Littlewood functions are the \emph{LLT polynomials} introduced by Lascoux, Leclerc, and Thibon~\cite{LLT}. We use here a variant introduced in \cite{HHLplus}. Let $\la\subseteq \mu$, and let $\nu = \mu/\la$ be the skew diagram. Let $c = (i,j)$ be the cell of $\nu$ in column $i$ and row $j$. The \emph{content} of $u$ is $c(u) \coloneqq j-i$. A \emph{semistandard Young tableau} of $\nu$ is a labeling of the cells of the skew diagram with positive integers which weakly increases left to right across each row and strictly increases up each column. Let $\SSYT(\nu)$ be the set of such fillings. Given $T\in \SSYT(\nu)$, let $\bx^T$ be the product of the variables corresponding to the labels appearing in $T$, counting multiplicity. 

Given a tuple of skew diagrams $\boldsymbol{\nu} = (\nu^{1},\dots, \nu^{m})$, let $\SSYT(\boldsymbol{\nu}) \coloneqq \SSYT(\nu^{1})\times \SSYT(\nu^{2})\times\cdots\times \SSYT(\nu^{m})$. Given $\mathbf{T} = (T^{1},\dots, T^{m})\in \SSYT(\boldsymbol{\nu})$, let $\bx^{\mathbf{T}} = \bx^{T^{1}}\cdots \bx^{T^{m}}$. Given $u = (i,j)$ a cell in $T^{a}$ and $v = (i',j')$ a cell of $T^{b}$, we say $(u, v)$ form an \emph{inversion} if we have the inequality of entries $T^{a}_{u} > T^{b}_{v}$, and either $a < b$ and $c(u) = c(v)$, or $a > b$ and $c(u) = c(v) + 1$. Denote by $\inv(\mathbf{T})$ the number of inversions of $\mathbf{T}$. The \emph{LLT polynomial indexed by $\boldsymbol{\nu}$} is
\begin{align}
G_{\boldsymbol{\nu}}(\bx;q) \coloneqq \sum_{\mathbf{T}\in \SSYT(\boldsymbol{\nu})} q^{\inv(\mathbf{T})} \bx^{\mathbf{T}}.
\end{align}
We need the following theorem, which has several algebraic and combinatorial proofs.
\begin{theorem}[\cite{HHL,HHLplus,LLT}]\label{thm:LLTSymmetry}
The LLT polynomial $G_{\boldsymbol{\nu}}(\bx;q)$ is symmetric in the variables $\bx$.
\end{theorem}

\section{Frobenius characteristic of $R_{n,\la, s}$}\label{sec:genset}

In this section, we identify $R_{n,\lambda, s}$ as a symmetric group module. Our main strategy, used by Garsia-Procesi~\cite{Garsia-Procesi} and formalized by Haglund-Rhoades-Shimozono~\cite[Section 4.1]{HRS1}, is to show that $R_{n,\la, s}$ is the associated graded ring of the coordinate ring of a finite set of points in $\bQ^n$. We then prove Theorem~\ref{thm:MonomialBasis}, which identifies a monomial basis of $R_{n,\la, s}$.

	\subsection{Associated graded rings and point orbits}\label{subsec:AssociatedGraded}

Throughout this section, we fix positive integers $k\leq n$ and $s$, and a partition $\lambda \in \Par(k, s)$. Fix $s$ distinct rational numbers $\alpha_{1},\dots, \alpha_{s}\in \bQ$. Let $X_{n,\la, s}$ be the set of points $p = (p_{1},\dots, p_{n})\in \bQ^{n}$ such that for each $1\leq i\leq n$, $p_i = \alpha_j$ for some $j$, and for each $1\leq i\leq s$, $\alpha_i$ appears as a coordinate in $p$ at least $\lambda_i$ many times. The \emph{defining ideal} of $X_{n,\la, s}$ is
\begin{align}
\ideal(X_{n,\la, s}) \coloneqq \{f\in \bQ[\bx_n] \st f(p) = 0 \text{ for }p\in X_{n,\la, s}\} \subseteq \bQ[\bx_n].
\end{align}
The quotient ring $\bQ[\bx_n]/\ideal(X_{n,\la, s})$ is the \emph{coordinate ring} of the set $X_{n,\la, s}$. It is isomorphic to the ring of polynomial functions $X_{n,\la, s}\to \bQ$. See \cite{CoxLittleOshea} for more background on the defining ideal and the coordinate ring of a variety.

For a degree $d$ polynomial $f = f_d + f_{d-1}+\dots +f_0\in \bQ[\bx_n]$ where $f_i$ is the degree $i$ homogenous summand of $f$, define $\topdeg(f) = f_d$ to be the top homogenous component of $f$. For example, if $f=x_1^2 x_2 + 2x_1^2 x_3 + x_2 x_3 + x_1 + 3$, then $\topdeg(f) = x_1^2x_2 + 2x_1^2 x_3$. 

The \emph{associated graded ideal} of $\ideal(X_{n,\la, s})$ with respect to the filtration by degree is
\begin{align}
\top(X_{n,\la, s}) = \langle\topdeg(f) \st f\in \ideal(X_{n,\la, s})\rangle,
\end{align}
which is a homogenous ideal since each of the generators is homogeneous. See~\cite{Eisenbud} for more details.

It is well known that the corresponding quotient ring $\bQ[\bx_n]/\top(X_{n,\la, s})$ is isomorphic to the associated graded ring of $\bQ[\bx_n]/\ideal(X_{n,\la, s})$ with respect to the filtration by degree. This is true more generally for the associated graded ideal of any ideal in $\bQ[\bx_n]$. See, e.g.\ \cite[Remark 3.1]{Garsia-Procesi}, for a proof of this fact in the case of the ideal $\ideal(X_{n,\la, s})$ when $k=n$. However, the proof easily extends to any ideal.

Since $X_{n,\la, s}$ is a finite set, we have
\begin{align}
|X_{n,\la, s}| = \dim_\bQ \frac{\bQ[\bx_n]}{\ideal(X_{n,\la, s})} = \dim_\bQ\frac{\bQ[\bx_n]}{\top(X_{n,\la, s})},\label{eq:DimEqualsX}
\end{align}
where all dimensions are as $\bQ$-vector spaces.
Since $X_{n,\la, s}$ is stable under the action of $S_n$ given by permuting coordinates, we have an $S_n$-action on the rings $\bQ[\bx_n]/\ideal(X_{n,\la, s})$ and $\bQ[\bx_n]/\top(X_{n,\la, s})$ given by permuting the variables $\bx_n$. As $S_n$-modules,
\begin{align}\label{eq:UngradedModuleIso}
\bQ X_{n,\la, s}\cong_{S_n} \frac{\bQ[\bx_n]}{\ideal(X_{n,\la, s})} \cong_{S_n} \frac{\bQ[\bx_n]}{\top(X_{n,\la, s})},
\end{align}
where $\bQ X_{n,\la, s}$ is the $S_n$-module of formal $\bQ$-linear combinations of points in $X_{n,\la, s}$. See~\cite[Section 4.1]{HRS1} for more details.

\begin{theorem}\label{thm:TopIdealEquality}
We have $I_{n,\la, s} = \top(X_{n,\la, s})$. Hence, we have the equality of rings
\begin{align}
R_{n,\la, s} = \frac{\bQ[\bx_n]}{\top(X_{n,\la, s})}.
\end{align} 
\end{theorem}

\begin{corollary}\label{cor:SymmModuleIso}
As $S_n$-modules, $R_{n,\la, s} \cong_{S_n} \bQ X_{n,\la, s}$.
\end{corollary}
\begin{proof}
Combine Theorem~\ref{thm:TopIdealEquality} with \eqref{eq:UngradedModuleIso}.
\end{proof}

We prove Theorem~\ref{thm:TopIdealEquality} in three parts. Recall the definition of $I_{n,\la, s}$ in Definition~\ref{def:TheIdeal}. First, we show that all of the generators of $I_{n,\la, s}$ are in $\top(X_{n,\la, s})$ in Lemma~\ref{lem:PolysInTop} by adapting the proof of Garsia-Procesi \cite[Proposition 3.1]{Garsia-Procesi}. Second, we find a monomial spanning set of the quotient $R_{n,\la, s} = \bQ[\bx_n]/I_{n,\la, s}$ of size $|X_{n,\la, s}|$. Finally, we finish the proof of Theorem~\ref{thm:TopIdealEquality} using a dimension counting argument. As a consequence, we see that our monomial spanning set is a monomial basis of $R_{n,\la, s}$.

\begin{lemma}\label{lem:PolysInTop}
We have the containment of ideals
\[
I_{n,\la, s}\subseteq \top(X_{n,\la, s}).
\]
\end{lemma}

\begin{proof}
First we show that $x_i^s\in \top(X_{n,\la, s})$ for all $i$.
For any $p\in X_{n,\la, s}$, the coordinates of $p$ are in the set $\{\alpha_1,\dots,\alpha_s\}$. Therefore, for each $i$, the polynomial function 
\begin{align}\label{eq:ProductFunction}
(x_i-\alpha_1)(x_i-\alpha_2)\cdots (x_i-\alpha_s) 
\end{align}
is in $\ideal(X_{n,\la, s})$. Since the top degree component of \eqref{eq:ProductFunction} is $x_i^s$, we have $x_i^s\in \top(X_{n,\la, s})$.

Second, we show that for any $d$ and $S\subseteq \bx_n$ such that $|S|\geq d > |S| - p_{|S|}^n(\la)$, we have $e_d(S)\in \top(X_{n,\la, s})$. For $m\leq n$, let $\bx_m\coloneqq \{x_1,\dots, x_m\}$. Since the ideal $\top(X_{n,\la, s})$ is closed under the action of $S_n$, it suffices to prove that $e_d(\bx_m)\in \top(X_{n,\la, s})$ for $d$ and $m$ such that $m\geq d > m-p^n_m(\la)$. 

Observe that $p_{m}^n(\la)$ is the number of cells of $\la$ weakly to the right of column $n-m+1$. For each $i\leq s$, let $c_{i, m}$ be the number of cells of the Young diagram of $\la$ which are in the $i$th row and are weakly to the right of column $n-m+1$. Observe that for any $p\in X_{n,\la, s}$, at least $c_{i, m}$ many $\alpha_i$'s must appear among the coordinates $p_1,\dots, p_m$. 
 Therefore, $\prod_{i=1}^m (t+p_i)$ is divisible by $\prod_{i=1}^s (t+\alpha_i)^{c_{i, m}}$. Hence, there exists a polynomial $z(t)$ such that 
\begin{align}\label{eq:EasyQuotient}
\prod_{i=1}^m (t+p_i) = z(t) \prod_{i=1}^s (t+\alpha_i)^{c_{i, m}}.
\end{align}

Let $\by_s = \{y_1,\dots, y_s\}$ be a second set of indeterminates. To show that $e_{d}(\bx_m) \in \top(X_{n,\la, s})$ for $d>m-p_m^n(\la)$, we consider the more general division problem of dividing $\prod_{i=1}^m (t+x_i)$ by $\prod_{i=1}^s (t+y_i)^{c_{i,m}}$ as polynomials in $t$ with coefficients in $\bQ[\bx_m,\by_s]$. Since $\sum_i c_{i, m} = p^n_m(\la)$, the remainder upon dividing $\prod_{i=1}^m (t+x_i)$ by $\prod_{i=1}^s (t+y_i)^{c_{i, m}}$ as polynomials in $t$ will be degree at most $p^n_m(\la)-1$ in $t$. Therefore, there exist polynomials $q(t,\bx_m,\by_s)$ and $r_d(\x_m,\y_s)$ for $0\leq d\leq p^n_m(\la)-1$ such that
\begin{align}\label{eq:quotientm}
\prod_{i=1}^m (t+x_i) = q(t,\x_m,\y_s) \prod_{i=1}^s (t+y_i)^{c_{i,m}} + \sum_{d=0}^{p^n_m(\la)-1} r_d(\x_m,\y_s) t^d.
\end{align}
Observe $r_d(\bx_m,\by_s)$ is homogeneous as a polynomial in $\bQ[\bx_m,\by_s]$, so $r_d(\bx_m,0^s)$ is the top degree component of $r_d(\bx_m,\by_s)$ as a polynomial in $\bx_m$ with coefficients in $\bQ[\by_s]$. Hence, $r_d(\bx_m,0^s) = \topdeg(r_d(\bx_n,\alpha_1,\dots,\alpha_s))$.
Plugging $y_i=0$ into \eqref{eq:quotientm} for $1\leq i\leq s$, we have
\begin{align}
\prod_{i=1}^m (t+x_i) = \sum_{i=0}^m e_{m-i}(\x_m)t^i = q(t,\x_m,0^s) t^{p^n_m(\la)} + \sum_{d=0}^{p^n_m(\la)-1} r_d(\x_m,0^s)t^d.
\end{align}
Hence, for $0\leq d\leq p^n_m(\la)-1$, we have $e_{m-d}(\x_m) = r_d(\x_m,0^s)$. 

By~\eqref{eq:EasyQuotient}, we have $r_d(p_1,\dots,p_m,\alpha_1,\dots,\alpha_s) = 0$ for all $p\in X_{n,\la, s}$ and $0\leq d\leq p^n_m(\la)-1$, and hence $r_d(\bx_m,\alpha_1,\alpha_2,\dots,\alpha_s)\in I(X_{n,\la, s})$. Hence, we have
\begin{align}
e_{m-d}(\bx_m) = r_d(\bx_m,0^s) = \topdeg(r_d(\bx_m,\alpha_1,\alpha_2,\dots,\alpha_s)) \in \top(X_{n,\la, s})
\end{align}
for all $0\leq d \leq p^n_m(\la) - 1$.
Replacing $d$ with $m-d$ yields $e_d(\x_m)\in \top(X_{n,\la, s})$ for all $m\geq d\geq m-p^n_m(\la)+1$. Hence, all of the generators of $I_{n,\la, s}$ are in $\top(X_{n,\la, s})$, so $I_{n,\la, s}\subseteq \top(X_{n,\la, s})$.
\end{proof}

\begin{remark}\label{rmk:Specialization}
When $k=n$, by \cite[Remark 3.1]{Garsia-Procesi} we have $I_\la = \top(X_{n,\la, \ell(\la)})$, which contains $x_i^{\ell(\la)}$ by the proof of Lemma~\ref{lem:PolysInTop}. Hence, we have $R_{n,\la, s} = R_\la$ for all $s\geq \ell(\la)$ in this case. 
\end{remark}

	\subsection{Shuffles and $(n,\la, s)$-staircases}\label{subsec:Shuffles}
	
	Before we proceed, we prove a couple of combinatorial lemmata concerning shuffles of compositions. We use these lemmata to construct a monomial basis of $R_{n,\la, s}$ in the next subsection.
	
\begin{lemma}\label{lem:ShuffleLemma}
Let $1\leq a\leq b$, and let $n=a+b$. Let $\gamma$ be a shuffle of the compositions $(0,1,\dots, a-1)$ and $(0,1,\dots, b-1)$. Then there exists a shuffle $\delta$ of $(0,1,\dots, a-1)$ and $(0,1,\dots, b-1)$ such that $\delta_n=a-1$ and $\trunc(\gamma)\subseteq \trunc(\delta)$.
\end{lemma}
\begin{proof}
If $\gamma_n=a-1$, then we may take $\delta=\gamma$, and we are done. Therefore, we assume $\gamma_n=b-1$ for the remainder of the proof. 

Let $P$ be the labeled lattice path whose corresponding shuffle is $\gamma$ as follows.
Suppose we have a lattice path in the plane starting at $(0,0)$ and ending at $(a,b)$ and taking only east steps $E=(1,0)$ and north steps $N=(0,1)$. Label the $i$th east step with $i-1$ for $i\leq a$, and label the $i$th north step of the path with $i-1$ for all $i\leq b$. To such a labeled lattice path, we associate the shuffle of $(0,1,\dots, a-1)$ and $(0,1,\dots, b-1)$ obtained by reading off the labels of the steps from left to right.

\begin{figure}[t]
\centering
\hfill
	\begin{minipage}{0.45\textwidth}
	\centering
	\includegraphics[scale=0.75]{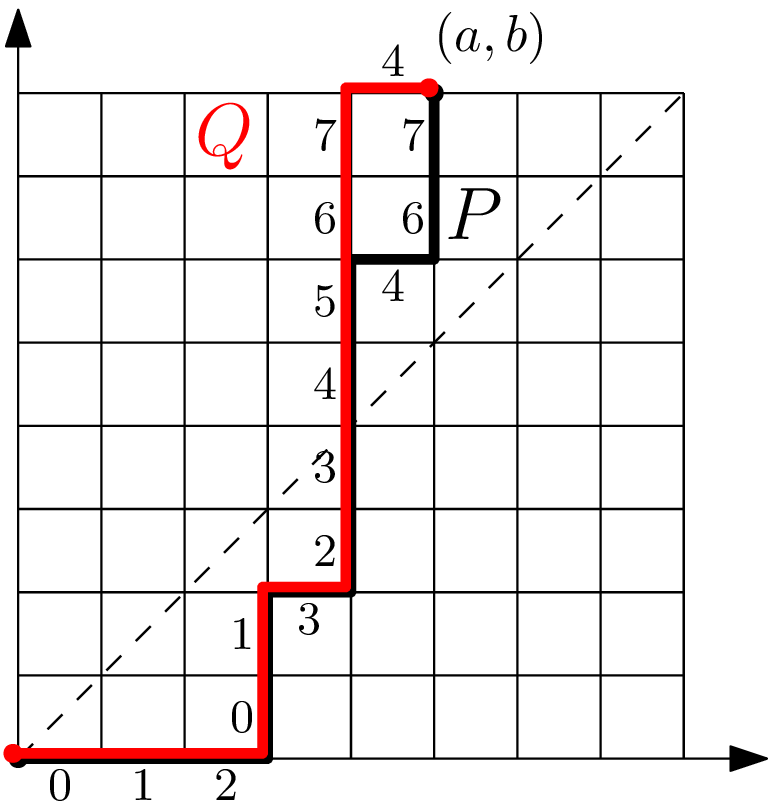}
	\end{minipage}
\hfill
	\begin{minipage}{0.45\textwidth}
	\centering
	\includegraphics[scale=0.75]{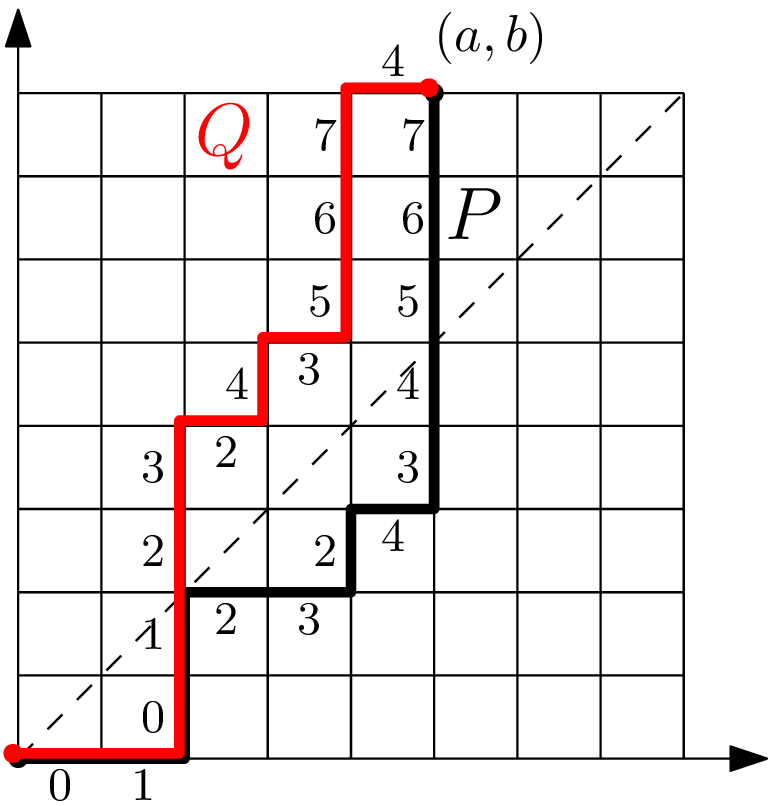}
	\end{minipage}
\hfill
\caption{Two examples of the path $Q$ constructed from the path $P$ in the proof of Lemma~\ref{lem:ShuffleLemma} for the case where $a=5$ and $b=8$. On the left, we have $\gamma= (0,1,2,0,1,3,2,3,4,5,4,6,7)$. On the right, $\gamma= (0,1,0,1,2,3,2,4,3,4,5,6,7)$.\label{fig:PathShuffles}}
\end{figure}

We construct a lattice path $Q$ starting at $(0,0)$ and ending at $(a,b)$ which stays weakly above $P$ as follows. 
See Figure~\ref{fig:PathShuffles} for examples of the path $Q$ we construct from the path $P$.
First, suppose that the last east step of $P$ lies weakly above the diagonal $y=x$, and suppose it is the $m$th step of $P$. Then the starting point of the last east step is at $(a-1,m-a)$, and we must have $m-a>a-1$. Define the first $m-1$ steps of $Q$ to be identical to the first $m-1$ steps of $P$. Define the rest of the path $Q$ to be $N^{n-m}E$. Since the $m$th step in $P$ is the last east step of $P$, then $Q$ stays weakly above $P$ and ends at $(a,b)$. 
Furthermore, the label of the $i$th step of $Q$ is greater than or equal to the $i$th step of $P$ for $i<n$. Letting $\delta$ be the composition read from the labels of $Q$ from left to right, then we have $\trunc(\gamma)\subseteq \trunc(\delta)$. Furthermore, since $Q$ ends in an east step we have $\delta_n = a-1$.

Second, suppose the last east step of $P$ lies below the diagonal $y=x$. In this case, the path $P$ must cross the point $(a,a)$ and then end with $b-a$ many north steps. In addition, $P$ must touch the diagonal $y=x$ in at least one other point. Suppose the point $(c,c)$ is the second to last time $P$ touches the diagonal. We define $Q$ in three segments, as follows. Define the first segment of $Q$ to be identical to the subpath of $P$ between the points $(0,0)$ and $(c,c)$. Let $P'$ be the subpath of $P$ from the point $(c,c)$ to the point $(a,a-1)$. Define the second segment of $Q$ to be the reflection of $P'$ across the diagonal $y=x$, which is a path from $(c,c)$ to $(a-1,a)$. Define the third segment of $Q$ to be $N^{b-a}E$, so that $Q$ ends at $(a,b)$. By construction, $Q$ lies weakly above $P$. Observe that the label of the $i$th step of $P$ equals the label of the $i$th step of $Q$ for all $i\leq 2a-1$. Furthermore, the label of the $i$th step of $Q$ is greater than or equal to the $i$th step of $P$ for $i<n$. Letting $\delta$ be the composition read from the labels of $Q$ from left to right, we have $\delta_n = a-1$ and $\trunc(\gamma)\subseteq \trunc(\delta)$. 
\end{proof}

\begin{definition}\label{def:Staircase}
For $1\leq j \leq \la_1$,  let $\beta^j(\la) = (0,1,\dots,\la_j'-1)$. An \emph{$(n,\la, s)$-staircase} is a shuffle $\gamma = (\gamma_1,\dots,\gamma_n)$ of the compositions $\beta^1(\la),\beta^2(\la),\dots,\beta^{\la_1}(\la)$, and $((s-1)^{n-k})$. Let 
\begin{align}
\cC_{n, \la, s} \coloneqq \{\alpha = (\alpha_1,\dots, \alpha_n) \st \alpha\subseteq \gamma\text{ for some }(n,\la, s)\text{-staircase }\gamma\}.
\end{align}
\end{definition}

Observe that when $\la = (1^k)\in \Par(k, k)$ and $s=k$, then $\beta^{1}(1^k) = (0,1,\dots, k-1)$. Hence, an $(n,(1^k),k)$-staircase is an $(n,k)$-staircase, as defined in \cite{HRS1} and in the Introduction.

\begin{lemma}\label{lem:TruncLemma}
Let $\alpha\in \cC_{n, \la, s}$, and suppose $\alpha_n < \ell(\la)$. Then $\trunc(\alpha)\in \cC_{n-1, \la^{(\alpha_n)}, s}$.
\end{lemma}

\begin{proof}
Since $\alpha\in \cC_{n,\la, s}$, then $\alpha\subseteq \gamma$, for some $(n,\la, s)$-staircase $\gamma$. Let $j$ be maximal such that $\alpha_n < \la_j'$.
It suffices to prove that $\trunc(\alpha)$ is contained in a shuffle of the compositions 
\begin{align}\label{eq:RecursiveShuffles}
\beta^{1}(\la),\dots,\trunc(\beta^{j}(\la)),\dots,\beta^{\la_1}(\la),\text{ and }((s-1)^{n-k}).
\end{align}
We have two cases: either $\gamma_n = s-1$, or $\gamma_n = \la'_h-1$ for some $h$. 

In the first case when $\gamma_n=s-1$, let $q$ be the index such that $\gamma_q = \la_j'-1$, corresponding to the last part of the composition $\beta^j(\la)$ in the shuffle $\gamma$. Let $\overline{\gamma}$ be the composition obtained by swapping the $q$th and $n$th entries of $\gamma$. Then $\overline{\gamma}$ is still an $(n,\la, s)$-staircase. Furthermore, since $\alpha_q \leq \gamma_q= \la_j'-1\leq s-1$, we have $\trunc(\alpha)\subseteq \trunc(\overline{\gamma})$, where $\trunc(\overline{\gamma})$ is a shuffle of the compositions listed in \eqref{eq:RecursiveShuffles}, hence $\trunc(\alpha)\in \cC_{n-1,\la^{(\alpha_n)}, s}$.

In the second case, we have $\gamma_n = \la'_h-1$ for some $h$. Since $\alpha_n \leq \gamma_n = \la'_h-1$, then $\alpha_n < \la'_h$, so we must have that $h\leq j$ by maximality of $j$. Let $\delta$ be the restriction of the composition $\gamma$ to the parts corresponding to $\beta^h(\la)$ and $\beta^j(\la)$ in the shuffle, so that $\delta_{\la'_h+\la'_j} = \la'_h-1$. By Lemma~\ref{lem:ShuffleLemma} with $a=\la'_j$ and $b=\la'_h$, there exists a shuffle $\epsilon$ of $\beta^h(\la)$ and $\beta^j(\la)$ such that $\epsilon_{\la'_h+\la'_j} = \la'_j-1$ and $\trunc(\delta)\subseteq \trunc(\epsilon)$. Let $\overline\gamma$ be the composition obtained by replacing the parts of $\gamma$ corresponding to $\delta$ with the parts of the composition $\epsilon$, in order from left to right. Then $\trunc(\overline{\gamma})$ is a shuffle of the compositions in \eqref{eq:RecursiveShuffles}. Furthermore, since $\trunc(\alpha)\subseteq \trunc(\gamma)$ as compositions of length $n$ and $\trunc(\delta)\subseteq\trunc(\epsilon)$ as compositions of length $\la'_h+\la'_j$, we have $\trunc(\alpha)\subseteq \trunc(\gamma)\subseteq \trunc(\overline\gamma)$, hence $\trunc(\alpha)\in \cC_{n-1, \la^{(\alpha_n)}, s}$.
\end{proof}

For a collection $\cC$ of compositions and an integer $i$, let us denote by $\cC\cdot (i)$ the collection of compositions $\alpha\cdot (i)$ for $\alpha\in \cC$. The symbol $\cupdot$ denotes a union of sets which are pairwise disjoint.
\begin{lemma}\label{lem:ShuffleRecursion}
We have the following decomposition of the set $\cC_{n, \la, s}$,
\begin{align}\label{eq:ShuffleRecursion}
\cC_{n,\la, s} = \bigcupdot_{i=0}^{\ell(\la)-1} \cC_{n-1, \la^{(i)}, s}\cdot (i) \,\cupdot\, \bigcupdot_{i=\ell(\la)}^{s-1} \cC_{n-1,\la, s}\cdot (i),
\end{align}
where on the right-hand side we interpret $\cC_{n-1,\la, s} = \emptyset$ if $n=|\la|$.
\end{lemma}

\begin{proof}
Observe that the right-hand side of \eqref{eq:ShuffleRecursion} is indeed a disjoint union of sets, since each set contains compositions with a distinct last coordinate. 
Given $\alpha\in \cC_{n-1,\la^{(i)}, s}$ for $i < \ell(\la)$ with $\alpha = (\alpha_1,\dots,\alpha_{n-1})$, then $\alpha\subseteq \beta$ for some $(n-1,\la^{(i)}, s)$-staircase $\beta$ by definition. Let $j$ be maximal such that $i < \la_j'$. Then we have $\alpha\cdot (i)\subseteq \beta\cdot (\la'_j-1)$, where $\beta\cdot (\la'_j-1)$ is an $(n,\la, s)$-staircase. Hence, we have $\alpha \cdot (i)\in \cC_{n,\la, s}$. Given $\alpha\in \cC_{n-1,\la, s}$ and $\ell(\la) \leq i \leq s -1$, then $\alpha\subseteq \beta$ for some $(n-1,\la, s)$-staircase $\beta$ by definition. Then we have $\alpha\cdot(i) \subseteq \beta\cdot (s-1)$, where $\beta\cdot(s-1)$ is an $(n,\la, s)$-staircase. Hence, we have $\alpha\cdot (i)\in \cC_{n,\la, s}$. Therefore, the disjoint union on the right-hand side of~\eqref{eq:ShuffleRecursion} is contained in the left-hand side as sets.

 Let $\alpha\in \cC_{n,\la, s}$. By definition of $\cC_{n,\la, s}$, there is a shuffle $\beta = (\beta_1,\dots,\beta_n)$ of $\beta^{1}(\la),\dots,\beta^{\la_1}(\la)$, and $((s-1)^{n-k})$ such that $\alpha\subseteq \beta$. We have $\trunc(\alpha)\subseteq \trunc(\beta)$ and $\alpha_n \leq \beta_n$.

If $\alpha_n \geq \ell(\la)$, since each part of the composition $\beta^{j}(\la)$ is at most $\ell(\la)-1$, it must be that $\beta_n=s-1$. Therefore, $\trunc(\beta)$ is an $(n-1,\la, s)$-staircase. Since $\trunc(\alpha)\subseteq \trunc(\beta)$, then $\trunc(\alpha)\in \cC_{n-1,\la, s}$, so that $\alpha = \trunc(\alpha)\cdot (\alpha_n) \in \cC_{n-1,\la, s}\cdot(\alpha_n)$. 

If $\alpha_n < \ell(\la)$, then by Lemma~\ref{lem:TruncLemma}, we have $\trunc(\alpha)\in \cC_{n-1,\la^{(\alpha_n)}, s}$. Therefore, we have $\alpha = \trunc(\alpha)\cdot (\alpha_n)\in \cC_{n-1,\la^{(i)}, s}\cdot(i)$ for $i=\alpha_n$. Hence, we have the equality of sets~\eqref{eq:ShuffleRecursion}.
\end{proof}

\begin{lemma}\label{lem:StaircaseContainment}
Let $h\leq k\leq n$ and $s$ be positive integers. Let $\la\in \Par(h, s)$, and let $\mu\in \Par(k, s)$. If $h=k$ and $\la\dominatedby\mu$, or if $h<k$ and $\la\subseteq \mu$, then $\cC_{n,\mu, s}\subseteq \cC_{n, \la, s}$.
\end{lemma}

\begin{proof}
It suffices to show that every $(n,\mu, s)$-staircase is contained in some $(n,\la, s)$-staircase.
First, suppose $h=k$ and $\la\dominatedby \mu$. It suffices to consider the case when $\mu$ covers $\la$ in dominance order. In this case, the Young diagrams of $\mu$ and $\la$ differ in only two columns. Suppose these two columns of $\mu$ are lengths $b$ and $a$ from left to right, so that $a\leq b$ and the two columns of $\la$ are lengths $b+1$ and $a-1$ from left to right.
Given $\beta$ a shuffle of the compositions $(0,1,\dots, b-1)$ and $(0,1,\dots, a-1)$, then by Lemma~\ref{lem:ShuffleLemma} there exists a shuffle $\delta$ of $(0,1,\dots, b-1)$ and $(0,1,\dots, a-1)$ such that $\delta_{a+b} = a-1$ and $\trunc(\beta)\subseteq \trunc(\delta)$. Let $\gamma = \trunc(\delta)\cdot (b)$. Then $\gamma$ is a shuffle of $(0,1,\dots, b)$ and $(0,1,\dots, a-2)$ such that $\beta\subseteq \gamma$.
Therefore, any $(n,\mu, s)$-staircase is contained in a $(n,\la, s)$-staircase, hence $\cC_{n,\mu, s}\subseteq \cC_{n,\la, s}$.

Second, suppose $h<k$ and $\la\subseteq \mu$. It suffices to consider the case when $k=h+1$, hence when the Young diagrams of $\mu$ and $\la$ only differ by one box. In this case, we have $\mu_j' = \la_j' + 1$ for some $j$. Given any $(n,\mu, s)$-staircase $\beta$,  replace the copy of $\mu_j'-1$ in $\beta$ corresponding to the last entry of $\beta^j(\mu)$ with an $s-1$. Then the resulting composition is an $(n,\la, s)$-staircase containing $\beta$, hence $\cC_{n,\mu, s}\subseteq \cC_{n,\la, s}$.
\end{proof}

	\subsection{Monomial basis of $R_{n,\la, s}$}\label{subsec:MonomialBasis}

To each weak composition $\alpha$ of length $n$, we associate a monomial
\begin{align}
\bx_n^\alpha \coloneqq \prod_{i=1}^n x_i^{\alpha_i}.
\end{align}
Let $\cA_{n,\la, s}$ be the following set of monomials in $\bQ[\bx_n]$,
\begin{align}\label{eq:DefOfBasis}
\cA_{n,\la, s} \coloneqq \{\bx_n^\alpha \st \alpha\in \cC_{n,\la, s}\}.
\end{align}

\begin{remark}\label{rmk:PartitionOfN}
Observe that if $n=k$, then $\cA_{n,\la, s} = \cA_{n,\la,\ell(\la)}$. This is consistent with the fact that $I_{n,\la, s}=I_{n,\la, \ell(\la)} = I_\la$ in this case.
\end{remark}

Given a monomial $\bx_n^\alpha$ and a set of monomials $\cA$, we denote by $\bx_n^\alpha \cA$ the set $\{\bx_n^\alpha \bx_n^\beta \st x^\beta\in \cA\}$.  We have the following recursion for the sets of monomials $\cA_{n,\la, s}$, which is an immediate corollary of Lemma~\ref{lem:ShuffleRecursion}.

\begin{corollary}\label{cor:BasisRecursion}
We have the following decomposition of the set $\cA_{n,\la, s}$,
\begin{align}\label{eq:BasisRecursion}
\cA_{n,\la, s} &= \bigcupdot_{i=0}^{\ell(\la)-1} x_n^{i} \cA_{n-1,\la^{(i)}, s} \cupdot \bigcupdot_{i=\ell(\la)}^{s-1} x_n^{i} \cA_{n-1,\la, s}.
\end{align}
\end{corollary}

\begin{example}\label{ex:TreeOfPartitions}
Let $\la\in \Par(k, s)$. We can obtain the set $\cA_{n,\la, s}$ by iteratively applying the recursion in Corollary~\ref{cor:BasisRecursion}.  
We have that
\begin{equation}\label{eq:ExampleOfTree}
\begin{aligned}
\cA_{4,(2,1), 3} = &\left\{ 1, x_1, x_1^2, x_2, x_2^2, x_3, x_1x_3, x_1^2x_3, x_2x_3, x_2^2x_3, x_3^2,\right. \\
&\left. x_2x_3^2, x_4, x_1x_4, x_1^2x_4, x_2x_4, x_2^2x_4, x_3x_4, x_3^2x_4, x_4^2, x_2x_4^2, x_3x_4^2 \right\}.
\end{aligned}
\end{equation}
\end{example}

\begin{lemma}\label{lem:TechnicalLemma}
Let $d$ and $m$ be positive integers. Let $i$ be a nonnegative integer, and let $S\subseteq \bx_{n-1}$ with $m=|S|$. We have that
\begin{align}
 x_n^{i} e_d(S)\in x_n^{i+1}\bQ[\bx_n]+I_{n,\la, s}\label{eq:TechnicalCondition}
 \end{align}
 in the following cases,
 \begin{enumerate}[(a)]
 \item  $d>m+1-p_{m+1}^n(\la)$,
 \item $d=m+1-p_{m+1}^n(\la)$ and $d+i > m - p^n_{m}(\la)$,
 \item $i < \ell(\la)$ and $e_d(S)$ is in the generating set of $I_{n-1,\la^{(i)}, s}$, 
 \item $\ell(\la) \leq i \leq s-1$, $k < n$, and $e_d(S)$ is in the generating set of $I_{n-1,\la, s}$.
 \end{enumerate}
\end{lemma}
\begin{proof}
In case (a), by our hypothesis and Definition~\ref{def:TheIdeal}, we have $e_d(S\cup \{x_n\})\in I_{n,\la, s}$. Hence, we have
\begin{align}
e_d(S) &= -x_n e_{d-1}(S) + e_d(S\cup\{x_n\}) \equiv -x_n e_{d-1}(S) \mod I_{n,\la, s},\label{eq:myeq1}\\
x_n^{i} e_d(S) &\equiv -x_n^{i+1}e_{d-1}(S) \mod I_{n,\la, s}\label{eq:myeq2},
\end{align}
so \eqref{eq:TechnicalCondition} holds. 

In case (b), we assume $d=m+1-p_{m+1}^n(\la)$. For $u>1$, we have $d+u>m+1-p^n_{m+1}(\la)$, so $e_{d+u}(S\cup\{x_n\})\in I_{n,\la, s}$. Furthermore, since $d+i > m-p_m^n(\la)$ by assumption, we have $e_{d+i}(S)\in I_{n,\la, s}$. Consider the identity
\begin{align}\label{eq:IdentityTrick}
(1-x_n^{i}t^{i})\prod_{x_j\in S} (1 - x_j t) = (1 + x_n t + \cdots + x_n^{i-1} t^{i-1}) \prod_{x_j\in S\cup \{x_n\}} (1 - x_j t).
\end{align}
The coefficient of $t^{d + i}$ on the left-hand side of \eqref{eq:IdentityTrick} is 
\begin{align}
(-1)^{d+i}e_{d+i}(S) + (-1)^{d+1} x_n^{i} e_d(S) \equiv (-1)^{d+1} x_n^{i}e_d(S) \mod I_{n,\la, s},
\end{align}
while the coefficient of $t^{d+i}$ on the right-hand side of \eqref{eq:IdentityTrick} is in $I_{n,\la, s}$ by the fact that $e_{d+u}(S\cup \{x_n\})\in I_{n,\la, s}$ for all $u>0$. Therefore, we have $x_n^{i} e_d(S)\in I_{n,\la, s}$, so \eqref{eq:TechnicalCondition} holds.

In case (c), we have $i < \ell(\la)$ and $d> m - p^{n-1}_m(\la^{(i)})$. Let $j$ be maximal such that $i < \la_j'$, so the Young diagram of $\la^{(i)}$ is obtained from the Young diagram of $\la$ by deleting a cell from the $j$th column from the left.
 If $n-m\leq j$, then $p^{n-1}_m(\la^{(i)}) = p^{n}_{m+1}(\la)-1$. Combining this with the inequality $d>m-p^{n-1}_m(\la^{(i)})$, we have $d>m + 1-p^n_{m+1}(\la)$, and we are done by case (a).

If on the other hand we have $n-m>j$, then $p^{n-1}_m(\la^{(i)}) = p^n_{m+1}(\la)$. Since $d>m-p^{n-1}_m(\la^{(i)})$, we have $d\geq m+1-p^{n-1}_m(\la^{(i)})$. If $d>m+1-p^{n-1}_m(\la^{(i)}) = m+1-p^n_{m+1}(\la)$, then we are again done by case (a). If we have $d=m+1-p^{n-1}_m(\la^{(i)}) = m+1-p^n_{m+1}(\la)$.
Furthermore, 
\begin{align}
p^n_{m+1}(\la) &= p^n_m(\la) + \la_{n-m}' < p^n_m(\la)+ i + 1 \label{eq:myeq3},
\end{align}
which follows from the maximality of $j$ and our assumption that $n-m > j$. Combining our assumption that $d=m+1-p^n_{m+1}(\la)$ with \eqref{eq:myeq3}, we have
\begin{align}
d+i > m-p^n_m(\la).
\end{align}
Hence, we are done by case (b). 

In case (d), we have $d> m - p^{n-1}_m(\la)$.
Since $k<n$ by assumption, we have that $\la'_n = 0$, so $p^n_{m+1}(\la) = p^{n-1}_m(\la)$. Hence, we have $d> m-p^n_{m+1}(\la)$. 
If $d>m+1-p^n_{m+1}(\la)$, then we are done by case (a). Otherwise, we have $d=m+1-p^n_{m+1}(\la)$. Furthermore, since $\la'_{n-m}\leq \ell(\la) < i + 1$, then \eqref{eq:myeq3} continues to hold, and combining it with the equality $d=m+1-p^n_{m+1}(\la)$, we obtain $d+i > m-p^n_m(\la)$. Hence, we are done by case (b).
\end{proof}

\begin{lemma}\label{lem:SpanningSet}
The set $\cA_{n,\la, s}$ represents a $\bQ$-spanning set of $R_{n,\la, s}$.
\end{lemma}
\begin{proof}
We proceed by induction on $n$. When $n=1$, either $\la = (1)$ or $\la  = \emptyset$. In the case when $\la = (1)$, then $R_{1,\la, s} = \bQ[\bx_1]/\langle x_1\rangle$ and $\cA_{1,\la, s} = \{1\}$. In the case when $\la = \emptyset$, then $R_{1,\la, s} = \bQ[x_1]/(x_1^s)$ and $\cA_{1,\la, s} = \{1,x_1,\dots, x_1^{s-1}\}$. Therefore, the statement holds for $n=1$.

 Assume $n>1$. Suppose by way of induction that $\cA_{m,\mu, s}$ is a $\bQ$-spanning set of $R_{m,\mu, s}$ for all $\mu$ and $m<n$. Let $\la\in \Par(k, s)$ for some $k\leq n$. Since $x_n^{s}\in I_{n,\la, s}$, we have an isomorphism of $\bQ$-vector spaces,
\begin{align}
R_{n,\la, s} \cong \bigoplus_{i=0}^{s-1} x_n^{i} R_{n,\la, s}/ x_n^{i+1} R_{n,\la, s}.
\end{align}
Therefore, it suffices to show that every polynomial of the form $x_n^{i} p(x_1,\dots, x_{n-1})$ is congruent to a polynomial in $\vecspan_{\bQ}(\cA_{n,\la, s})$ modulo $x_n^{i+1} \bQ[\bx_n] + I_{n,\la, s}$. We have two cases: either $0\leq i < \ell(\la)$ or $\ell(\la) \leq i \leq s-1$.
 
 In the first case when $0\leq i < \ell(\la)$, by our inductive hypothesis we have that $\cA_{n-1,\la^{(i)}, s}$ is a $\bQ$-spanning set of $R_{n-1,\la^{(i)}, s}$, so
 \begin{align}
 p(x_1,\dots, x_{n-1}) = \sum_{\bx_{n-1}^\alpha \in \cA_{n-1,\la^{(i)}, s}} c_\alpha \,\bx_{n-1}^\alpha  + \sum_{e_d(S)\in I_{n-1,\la^{(i)}, s}} A(d,S) e_d(S) + \sum_{j=1}^{n-1} A_j x_j^s,
 \end{align}
 for some constants $c_\alpha \in \bQ$ and some polynomials $A(d,S)$ and $A_j$ in $\bQ[\bx_{n-1}]$, where the second sum is over all generators of $I_{n-1,\la^{(i)}, s}$ of the form $e_d(S)$. Hence,
 \begin{align}
x_n^{i} p(x_1,\dots, x_{n-1}) = \sum_{\bx_{n-1}^\alpha \in \cA_{n-1,\la^{(i)}, s}} c_\alpha \,x_n^{i}\bx_{n-1}^\alpha  + \sum_{e_d(S)\in I_{n-1,\la^{(i)}, s}} A(d,S) x_n^{i}e_d(S) + \sum_{j=1}^{n-1} A_j x_n^{i}x_j^s.
 \end{align}
 Observe that $x_n^{i}x_j^s\in I_{n,\la, s}$ and that $x_n^{i}e_d(S)\in x_n^{i+1}\bQ[\bx_n] + I_{n,\la, s}$ by Lemma~\ref{lem:TechnicalLemma}(c). Furthermore, $x_n^{i}\bx_{n-1}^\alpha \in \cA_{n,\la, s}$ for all $\bx_{n-1}^\alpha \in \cA_{n-1,\la^{(i)}, s}$ by Corollary~\ref{cor:BasisRecursion}. Therefore, $x_n^{i}p$ is congruent to a polynomial in $\vecspan_\bQ(\cA_{n,\la, s})$ modulo $x_n^{i+1}\bQ[\bx_n]+I_{n,\la, s}$, and we are done.
 
In the second case, we have $\ell(\la) \leq i \leq s-1$. If $n=k$, then $x_n^i\in I_\la = I_{n,\la, s}$ by Remark~\ref{rmk:Specialization}, so $x_n^i p \in I_{n,\la, s}$. Otherwise, we have $k<n$, and by our inductive hypothesis we have that $\cA_{n-1,\la, s}$ is a $\bQ$-spanning set of $R_{n-1,\la, s}$. A similar application of Lemma~\ref{lem:TechnicalLemma}(d) completes the induction. Hence, $\cA_{n,\la, s}$ is a $\bQ$-spanning set of $R_{n,\la, s}$.
\end{proof}

\begin{example}
To illustrate the proof of Lemma~\ref{lem:SpanningSet}, take $n=4$, $\la = (2,1)$, and $s=3$. Let us show that $x_1x_4^2$ is in the span of $\cA_{4,(2,1),3}$ modulo the ideal
\begin{align}
I_{4,(2,1),3} =& \langle x_1^3,x_2^3,x_3^3,x_4^3,e_2(x_1,x_2,x_3,x_4), e_3(x_1,x_2,x_3,x_4),e_4(x_1,x_2,x_3,x_4),\\ &e_3(x_1,x_2,x_3),e_3(x_1,x_2,x_4),e_3(x_1,x_3,x_4),e_3(x_2,x_3,x_4)\rangle.
\end{align}
In this case, the power of $x_4$ is $i=2= \ell(\la)$ and $p(x_1,x_2,x_3) = x_1$, so we are in the second case of the proof of Lemma~\ref{lem:SpanningSet}. By induction, we know that $p = x_1$ is in the span of $\cA_{3,(2,1),3} = \{1,x_2,x_3\}$ modulo $I_{3,(2,1),3}$. Indeed,
\begin{align}
x_1 = &-x_2 -x_3 + e_1(x_1,x_2,x_3),
\end{align}
where $x_2,x_3\in \cA_{3,(2,1),3}$ and $e_1(x_1,x_2,x_3)\in I_{3,(2,1),3}$.
Multiplying through by $x_4^2$,
\begin{align}
x_1x_4^2 = -x_2x_4^2 - x_3x_4^2 + x_4^2e_1(x_1,x_2,x_3).
\end{align}

From~\eqref{eq:ExampleOfTree}, we see that $x_2x_4^2,x_3x_4^2 \in \cA_{4,(2,1),3}$, so it suffices to show that $x_4^2 e_1(x_1,x_2,x_3)\in I_{4,(2,1),3}$. Indeed, as in the proof of Lemma~\ref{lem:TechnicalLemma}(b) for $d=1$, $n=4$, $i=2$, $S = \{x_1,x_2,x_3\}$, we have the identity
\begin{align}
-e_3(x_1,x_2,x_3) + x_4^2e_1(x_1,x_2,x_3) = -e_3(x_1,x_2,x_3,x_4) + x_4 e_2(x_1,x_2,x_3,x_4).
\end{align}
Since $e_3(x_1,x_2,x_3)$, $e_3(x_1,x_2,x_3,x_4)$, and $e_2(x_1,x_2,x_3,x_4)$ are in $I_{4,(2,1),3}$, then $x_4^2 e_1(x_1,x_2,x_3)\in I_{4,(2,1),3}$, and we are done.
\end{example}

\begin{lemma}\label{lem:BasisBijection}
We have the equality of cardinalities $|X_{n,\la, s}| = |\cA_{n,\la, s}|$.
\end{lemma}
\begin{proof}
Recall that $X_{n,\la, s}$ is the set of points $p = (p_{1},\dots, p_{n})\in \bQ^{n}$ such that for each $1\leq i\leq n$, $p_i = \alpha_j$ for some $j$, and for each $1\leq i\leq s$, $\alpha_i$ appears as a coordinate in $p$ at least $\lambda_i$ many times. Observe that the size of $X_{n,\la, s}$ does not depend on our choice of the distinct rational numbers $\alpha_i$. The statement of the lemma holds when $n=1$. Indeed, if $\la = (1)$ then $X_{1,\la, s} = \{(\alpha_1)\}$ and $\cA_{1,\la, s} = \{1\}$. Otherwise, we have that $\la = \emptyset$, in which case $X_{1,\la, s} = \{(\alpha_1),(\alpha_2),\dots,(\alpha_s)\}$ and $\cA_{1,\la, s} = \{1,x_1,\dots, x_1^{s-1}\}$, which are equinumerous. In light of Corollary~\ref{cor:BasisRecursion}, it suffices to prove that the cardinalities $|X_{n,\la, s}|$ satisfy the same recursion as $|\cA_{n,\la, s}|$, namely that
\begin{align}
|X_{n,\la, s}| = \sum_{i=0}^{\ell(\la)-1} |X_{n-1,\la^{(i)}, s}| + \sum_{i=\ell(\la)}^{s-1} |X_{n-1,\la, s}|.\label{eq:XRecursion}
\end{align}
The identity \eqref{eq:XRecursion} follows by observing that for $\ell(\la)\leq i\leq s-1$, the set $X_{n-1,\la, s}$ is in bijection with the set
\begin{align}
\{p\in X_{n,\la, s} \st (p_1,\dots, p_{n-1})\in X_{n-1,\la, s}, \,p_n = \alpha_{i+1}\}
\end{align}
and for $0\leq i < \ell(\la)$, the set $X_{n-1,\la^{(i)}, s}$ is in bijection with the set
\begin{align}
\{p\in X_{n,\la, s}\st (p_1,\dots, p_{n-1})\in X_{n-1,\la^{(i)}, s},\, p_n=\alpha_{i+1}\},
\end{align}
which completes the proof.
\end{proof}
\begin{remark}
The proof of Lemma~\ref{lem:BasisBijection} naturally leads to a recursively constructed bijection between the sets $X_{n,\la, s}$ and $\cA_{n,\la, s}$, though we will have no use for this bijection.
\end{remark}

\begin{proof}[Proof of Theorem~\ref{thm:TopIdealEquality}]
Recall that by Lemma~\ref{lem:PolysInTop}, we have the containment of ideals $I_{n,\la, s}\subseteq \top(X_{n,\la, s})$. Combining this with \eqref{eq:DimEqualsX}, we have
\begin{align}
\dim_\bQ (R_{n,\la, s}) \geq \dim_\bQ\frac{\bQ[\bx_n]}{\top(X_{n,\la, s})} = |X_{n,\la, s}|.\label{eq:DimEq1}
\end{align}
Furthermore, by Lemma~\ref{lem:SpanningSet} and Lemma~\ref{lem:BasisBijection}, we have
\begin{align}
|X_{n,\la, s}| = |\cA_{n,\la, s}| \geq \dim_\bQ (R_{n,\la, s}).\label{eq:DimEq2}
\end{align}
Stringing together \eqref{eq:DimEq1} and \eqref{eq:DimEq2}, we see that all inequalities must be equalities. In particular, we have the equality
\begin{align}
\dim_\bQ (R_{n,\la, s}) = \dim_\bQ \frac{\bQ[\bx_n]}{\top(X_{n,\la, s})}.
\end{align}
As a consequence, we also have $I_{n,\la, s} = \top(X_{n,\la, s})$, hence $R_{n,\la, s} = \bQ[\bx_n]/\top(X_{n,\la, s})$.
\end{proof}

\begin{theorem}\label{thm:MonomialBasis}
The set of monomials $\cA_{n,\la, s}$ represents a basis of $R_{n,\lambda, s}$.
\end{theorem}
\begin{proof}
By the proof of Theorem~\ref{thm:TopIdealEquality}, we have that $|\cA_{n,\la, s}| = \dim_\bQ (R_{n,\la, s})$. By Lemma~\ref{lem:SpanningSet}, $\cA_{n,\la, s}$ is a $\bQ$-spanning set of $R_{n,\la, s}$. Hence, $\cA_{n,\la, s}$ is a basis of $R_{n,\la, s}$.
\end{proof}

\begin{lemma}\label{lem:Containment}
Let $h\leq k\leq n$ be positive integers, let $\la\in \Par(h, s)$, and let $\mu\in \Par(k, s)$. If $h=k$ and $\la\dominatedby\mu$, or if $h<k$ and $\la\subseteq \mu$, then $\cA_{n,\mu, s}\subseteq \cA_{n,\la, s}$ and $I_{n,\la, s}\subseteq I_{n,\mu, s}$.
\end{lemma}
\begin{proof}
If $h=k$ and $\la\dominatedby \mu$ or if $h<k$ and $\la\subseteq\mu$, then $p^n_m(\la)\leq p^n_m(\mu)$ for all $m$. Therefore, the generating set of $I_{n,\la, s}$ is contained in the generating set of $I_{n,\mu, s}$, so $I_{n,\la, s}\subseteq I_{n,\mu, s}$. By Lemma~\ref{lem:StaircaseContainment}, we have $\cC_{n,\mu, s}\subseteq \cC_{n, \la, s}$, so the containment $\cA_{n,\mu, s}\subseteq \cA_{n,\la, s}$ follows.
\end{proof}

As a consequence of the containment of ideals in Lemma~\ref{lem:Containment}, we have a monotonicity property of the multiplicities of irreducible representations contained in the ring $R_{n,\la, s}$. We state it next in terms of graded Frobenius characteristics.
\begin{theorem}\label{thm:Monotonicity}
Let $h\leq k\leq n$ be positive integers, let $\la\in \Par(h, s)$, and let $\mu\in \Par(k, s)$ such that either $h=k$ and $\la\dominatedby\mu$, or $h<k$ and $\la\subseteq \mu$. For each partition $\nu\vdash n$, we have the inequality
\begin{align}
[s_{\nu}] \Frobq(R_{n,\la, s})\geq [s_\nu] \Frobq(R_{n,\mu, s}),
\end{align}
where $[s_{\nu}]f$ stands for the coefficient of $s_\nu$ in the Schur function expansion of $f$, and the inequality is a coefficient-wise comparison of two polynomials in $q$.
\end{theorem}

	\subsection{Ordered set partitions}
In this subsection, we relate the $S_n$-module $R_{n,\la, s}$ to an action on $(n,\la, s)$-ordered set partitions. We then find the Frobenius characteristic of $R_{n,\la, s}$. 

\begin{theorem}\label{thm:UngradedFrobChar}
We have that $\dim_\bQ(R_{n,\la, s}) = |\OP_{n,\la, s}|$. Furthermore, we have the isomorphism of $S_n$-modules
\begin{align}
R_{n,\la, s} \cong_{S_n}\mathbb{Q}\mathcal{OP}_{n,\la, s}.
\end{align}
\end{theorem}
\begin{proof}
By Corollary~\ref{cor:SymmModuleIso}, we have that $R_{n,\la, s}\cong_{S_n} \bQ X_{n,\la, s}$ as $S_n$-modules. Therefore, it suffices to show there is an $S_n$-equivariant bijection between $X_{n,\la, s}$ and $\OP_{n,\la, s}$. Define a map $\varphi: X_{n,\la, s}\to \OP_{n,\la, s}$ as follows. Given $p\in X_{n,\la, s}$, define $\varphi(p) = (B_1|\cdots|B_s)$ where $B_i$ is the set of indices $1\leq j\leq n$ such that $p_j = \alpha_i$. By the definition of $X_{n,\la, s}$, we have that $\alpha_i$ appears at least $\la_i$ many times as a coordinate in $p$. Therefore, we have $|B_i|\geq \la_i$ for all $i\leq s$, so $\varphi(p)\in \OP_{n,\la, s}$. The map $\varphi$ is clearly the desired $S_n$-equivariant bijection, which completes the proof.
\end{proof}

\begin{corollary} The Frobenius characteristic of $R_{n,\la, s}$ is
\begin{align}\label{eq:FrobFormula}
\Frob(R_{n,\la, s}) =  \sum_{\substack{\mu\in \Par(n, s),\\ \mu\supseteq \la}} \, h_{\mu} \prod_{i\geq 0} \binom{\mu_i' - \la_{i+1}'}{\mu_i'-\mu_{i+1}'}  ,
\end{align}
where $\mu'_0\coloneqq s$.
\end{corollary}
\begin{proof}
By Theorem~\ref{thm:UngradedFrobChar}, we have $R_{n,\la, s} \cong_{S_n} \bQ\OP_{n,\la, s}$. We can partition the set $\OP_{n,\la, s}$ into $S_n$-orbits, where an orbit is determined by the tuple of block sizes $\alpha = (|B_1|,\dots,|B_s|)$. This correspondence sets up a bijection between the set of $S_n$-orbits of $\OP_{n,\la, s}$ and the set of all weak compositions $\alpha = (\alpha_1,\dots,\alpha_s)$ of $n$ such that $\alpha\supseteq \la$. Given such a composition $\alpha$, let $\mathscr{O}$ be the corresponding $S_n$ orbit of $\OP_{n,\la, s}$. Letting $\mu = \sort(\alpha)$, then $\mathscr{O}$ is isomorphic as an $S_n$-module to the set of tabloids with $\mu_i$ boxes in the $i$th row. Hence, the Frobenius characteristic of the action of $S_n$ on the submodule $\bQ \mathscr{O}$ of $\bQ \OP_{n,\la, s}$ is equal to $h_\mu$. 

Since $\bQ\OP_{n,\la, s}$ is the direct sum over all $\bQ\mathscr{O}$ where $\mathscr{O}$ is an $S_n$-orbit, we have
\begin{align}
\Frob(R_{n,\la, s}) = \sum_{\substack{\mu\in \Par(n, s),\\ \mu\supseteq \la}} \, a_{\la,\mu}^{(s)} h_{\mu},
\end{align}
where $a_{\la,\mu}^{(s)}$ is the number of $\alpha\in \Comp(n, s)$ such that $\alpha\supseteq \la$ and $\sort(\alpha) = \mu$. It is then an easy exercise to verify $a_{\la,\mu}^{(s)}$ is equal to the coefficient of $h_\mu$ in the right-hand side of \eqref{eq:FrobFormula}.
\end{proof}

\section{Skewing formulas and exact sequences for $R_{n,\la, s}$}\label{sec:Skew}
In this section, we develop algebraic tools for analyzing the graded Frobenius characteristic of $\Frobq(R_{n,\la, s})$. Our main tool is a recursive formula for the image of $\Frobq(R_{n,\la, s})$ under the skewing operator $e_j(\bx)^\perp$ from Subsection~\ref{subsec:SymmFunctions}. We also show that the rings $R_{n,\la, s}$ fit into certain exact sequences.

	\subsection{Skewing formula}

Fix $j$, $k$, $s$, and $n$ positive integers with $k\leq n$, and fix $\la\in \Par(k, s)$ throughout the subsection. 
In order to simplify notation, let $z_i \coloneqq x_{n-j+i}$ for $1\leq i \leq j$ and $\bz_j \coloneqq \{z_1,\dots, z_j\} = \{x_{n-j+1},\dots, x_n\}$. Given a polynomial $f(\bx_{n-j},\bz_j)$, then $\sigma\in S_{n-j}$ acts on the $\bx_{n-j}$ variables, and $\epsilon_j$ acts on the $\bz_j$ variables. 
We have the following definitions, which we need for our formulas for $e_j^\perp \Frobq(R_{n,\la, s})$.
\begin{definition}
Let $I = (i_1,\dots, i_j) \in [0,s-1]^j$. Construct a partition $\la^{(I)}$ recursively using the reduction operations defined in \eqref{eq:ReductionDef} as follows. Let $\la_{(j)} \coloneqq \la$, and for $1\leq h \leq j$, let 
\[
\la_{(h-1)}  \coloneqq 
\begin{cases} 
 \la_{(h)}^{(i_h)} & \text{if }i_h < \ell(\la_{(h)}),\\
 \la_{(h)} & \text{if } i_h \geq \ell(\la_{(h)}).
 \end{cases}
\]
Define $\la^{(I)} \coloneqq \la_{(0)}$.
\end{definition}

Let $\cI_s^j$ be the set of increasing sequences $(i_1 < \cdots < i_j)$ of nonnegative integers with $i_j < s$. For $I \in \cI_s^j$, let $m$ be maximal such that $i_m < \ell(\la)$. In this case, observe that $\la^{(I)}\in \Par(k-m, s)$ is the partition obtained from $\la$ by deleting one box from the end of the rows $i_1+1,\dots, i_m+1$ and then sorting the rows. Further observe that for $I=(i)$, if $0\leq i < \ell(\la)$, then $\la^{(I)} = \la^{(i)}$, and if $ \ell(\la) \leq i < s$, then $\la^{(I)} = \la$.

Given a sequence $I = (i_1,\dots, i_j)$ of distinct nonnegative integers which is not necessarily increasing, let $\bz_j^{I} \coloneqq z_1^{i_1}\cdots z_j^{i_j}$. Furthermore, let $\Sigma(I) \coloneqq i_1 + \cdots + i_j$, and let $\sort_{\leq}(I)$ be the increasing sequence obtained by sorting the entries of $I$.

\begin{theorem}\label{thm:AntisymmFrobqRecursion}
We have
\begin{align}
e_j^{\perp} \Frobq(R_{n,\la, s}) = \Frobq(\epsilon_j R_{n,\la, s})  = \sum_{I \in \cI_s^j} q^{\Sigma(I)} \,\Frobq(R_{n-j,\la^{(I)}, s}).
\end{align}
\end{theorem}
\noindent The first equality in Theorem~\ref{thm:AntisymmFrobqRecursion} follows immediately by~\eqref{eq:AntisymmFrob}, so it suffices to prove the second equality. 

Define
\begin{align}
V_r &\coloneqq \bigoplus_{\substack{I\in \cI_s^j,\\\Sigma(I) = r}} R_{n-j,\la^{(I)}, s} \otimes \bQ\{\epsilon_j \bz_j^{I}\}.\label{eq:VDef}
\end{align}
The vector space $V_r$ has the structure of a $S_{n-j}$-module, where $S_{n-j}$ acts on the first factor of each tensor product. It is also graded, where the degree of a nonzero simple tensor $f\otimes \epsilon_j \bz_j^I$ is $\deg(f) + \Sigma(I)$.
Observe that if we could directly show that $\epsilon_j R_{n,\la, s} \cong \bigoplus_{r\geq 0} V_r$ as graded $S_{n-j}$-modules, then the second equality in Theorem~\ref{thm:AntisymmFrobqRecursion} would follow. A natural choice for this isomorphism would be the map from $\bigoplus_{r\geq 0}V_r$ to $\epsilon_j R_{n,\la, s}$ induced by multiplication, if it is well-defined. Unfortunately, this map is not well-defined in general, as the next example illustrates. 

\begin{example}
Let $n = 4$, $\la = (1^3)$, $s=3$, and $j=2$. In order to have a well-defined map $V_r\to  \epsilon_j R_{n,\la, s}$ induced by multiplication, then in particular we would need a well-defined map $R_{n-j,\la^{(I)}, s}\otimes \bQ\{\epsilon_j \bz_j^I\}\to \epsilon_j R_{n,\la, s}$ for each $I\in \cI_s^j$ induced by multiplication. Letting $I = (0,1)$, then we would need a map
\begin{align}
R_{2,(1), 3} \otimes \bQ\{x_3-x_4\}\to \epsilon_2 R_{4,(1^3), 3}.
\end{align}
However, in order for this map to be induced from multiplication, we would need the containment \mbox{$I_{2,(1), 3}\cdot (x_3-x_4)\subseteq \epsilon_2 I_{4,(1^3), 3}$}. In particular, we would need 
\begin{align}
e_2(x_1,x_2)(x_3-x_4) = x_1x_2(x_3-x_4)\in \epsilon_2 I_{4,(1^3), 3},
\end{align} 
which is not true. However, observe that we do have 
\begin{align}
x_1x_2 (x_3-x_4) + x_1(x_3^2-x_4^2) + x_2(x_3^2-x_4^2) + x_3x_4(x_3-x_4)\in \epsilon_2 I_{4,(1^3), 3},
\end{align}
which contains $x_1x_2(x_3-x_4)$ as a term. Further observe that all other terms in this element have higher degree in the variables $z_1=x_3$ and $z_2 = x_4$. This suggests that we may be able to define a map from $V_r$ induced by multiplication if we first filter the codomain $\epsilon_j R_{n,\la, s}$ by total degree in the $\bz_j$ variables.
\end{example}

Indeed, we prove the second equality in Theorem~\ref{thm:AntisymmFrobqRecursion} by filtering $\epsilon_j R_{n,\la, s}$ by total degree in the $\bz_j = \{x_{n-j+1},\dots, x_n\}$ variables and then constructing an explicit isomorphism of $S_{n-j}$-modules which corresponds to the equality of symmetric functions in Theorem~\ref{thm:AntisymmFrobqRecursion}. Before giving the proof, we need a few lemmata.

\begin{lemma}\label{lem:DimOfAntisymm} We have
\begin{align}\label{eq:DimOfAntisymm}
\dim_\bQ(\epsilon_j R_{n,\la, s}) = \dim_\bQ(\epsilon_j\bQ \OP_{n,\la, s}) = \sum_{I \in \cI_s^j} |\OP_{n-j,\la^{(I)}, s}|.
\end{align}
\end{lemma}
\begin{proof}
The first equality in \eqref{eq:DimOfAntisymm} follows by Theorem~\ref{thm:UngradedFrobChar}, so it suffices to prove the second equality. Given $\sigma\in \OP_{n,\la, s}$, we may find a unique representative of its $S_{\{n-j+1,\dots, n\}}$-orbit by sorting the letters $n-j+1,\dots, n$ so that they appear in order from left to right in $\sigma = (B_1|B_2|\cdots |B_s)$. By Lemma~\ref{lem:EjLemma} applied to $Y= \OP_{n,\la, s}$, we have that $\dim_\bQ(\epsilon_j\bQ\OP_{n,\la, s})$ is the number of such representatives with trivial $S_{\{n-j+1,\dots, n\}}$-stabilizer. Since $\sigma\in \OP_{n,\la, s}$ has trivial $S_{\{n-j+1,\dots, n\}}$-stabilizer if and only if $n-j+1,\dots, n$ are in distinct blocks of $\sigma$, then $\dim_\bQ(\epsilon_j\bQ\OP_{n,\la, s})$ is the number of $\sigma\in \OP_{n,\la, s}$ such that $n-j+1,\dots, n$ are in distinct blocks and appear in order from left to right. 

Define a bijection between the set of $\sigma\in \OP_{n,\la, s}$ such that $n-j+1,\dots, n$ are in distinct blocks and appear in order from left to right and the formal disjoint union of sets
\begin{align}\label{eq:FormalDisjUnion}
\bigsqcup_{I \in \cI_s^j} \OP_{n-j,\la^{(I)}, s}
\end{align}
as follows. Given such a $\sigma$, let $I = (i_1<\dots<i_j)\in \cI_s^j$ be such that $n-j+1, \dots, n$ are in $B_{i_1+1},\dots, B_{i_j+1}$, respectively. Map $\sigma$ to the ordered set partition obtained by removing $n-j+1,\dots, n$ from $\sigma$, considered as an element of $\OP_{n-j,\la^{(I)}, s}$ in the formal disjoint union~\eqref{eq:FormalDisjUnion}. This defines a bijection. Indeed, given $\sigma' = (B'_1|\cdots |B'_s) \in \OP_{n-j,\la^{(I)}, s}$ for $I = (i_1<\cdots< i_j)\in \cI_s^j$, then the corresponding $\sigma$ may be recovered by adding $n-j+h$ to block $B'_{i_h+1}$ for $1\leq h\leq j$. Hence, the second equality in \eqref{eq:DimOfAntisymm} follows.
\end{proof}

\begin{lemma}\label{lem:RemovingBoxes}
Let $H \in [0,s-1]^j$ with distinct entries, and let $I = \sort_{\leq}(H)$. We have $\la^{(I)}\dominatedby \la^{(H)}$.
\end{lemma}

\begin{proof}
Recall that we draw Young diagrams according to the French convention. If $s = \ell(\la)$, then there is a bijection between the cells of the skew Young diagram $\la/\la^{(H)}$ and the cells of $\la/\la^{(I)}$ such that each cell of $\la/\la^{(H)}$ is weakly above and to the left of its corresponding cell of $\la/\la^{(I)}$. It follows that $\la^{(I)}\dominatedby \la^{(H)}$. See \cite[pp. 128-130]{Garsia-Procesi} for the full proof of the case when $s=\ell(\la)$. In the case when $s>\ell(\la)$, let $m$ be maximal such that $i_m < \ell(\la)$, and let $H'$ be the subsequence of $H$ consisting of elements less than $\ell(\la)$. Then we have 
\begin{align}
\la^{(I)} = \la^{(i_1,\dots, i_m)} \dominatedby \la^{(H')} = \la^{(H)},
\end{align}
where the inequality in the middle follows from the first case, since $H'\in [0,\ell(\la)-1]^m$ and $(i_1,\dots, i_m) = \sort_{\leq}(H')$.
\end{proof}

\begin{remark}\label{rmk:Vandermonde}
The alternating polynomial $\epsilon_j \bz_j^{I}$ is equal to a scalar multiple of a generalized Vandermonde determinant, denoted by  $\Delta_{i_1,i_2,\dots,i_j} = j! \det(z_p^{i_q})_{p, q=1,\dots, j}$ in \cite{Garsia-Procesi}.
\end{remark}

\begin{lemma}\label{lem:AltBasis}
The collection of polynomials $\cA^{j\text{-}\mathrm{alt}}_{n,\la, s}$ defined by
\begin{align}
\cA^{j\text{-}\mathrm{alt}}_{n,\la, s} \coloneqq \bigcupdot_{I \in \cI_s^j} \cA_{n-j,\la^{(I)}, s}\cdot \epsilon_j\bz_j^{I}\label{eq:AntisymmBasis}
\end{align}
represents a basis of $\epsilon_j R_{n,\la, s}$, where $\cA_{n,\la, s}$ is as defined in \eqref{eq:DefOfBasis}.
\end{lemma}

\begin{proof}
By Theorem~\ref{thm:MonomialBasis}, we have $|\cA_{n,\la, s}| = |\OP_{n,\la, s}|$.  By Lemma~\ref{lem:DimOfAntisymm}, we have
\begin{align}\label{eq:SkewDimEq}
|\cA^{j\text{-alt}}_{n,\la, s}| = \sum_{I\in \cI_s^j} |\cA_{n-j,\la^{(I)}, s}| = \sum_{I\in \cI_s^j} |\OP_{n-j,\la^{(I)}, s}| = \dim_\bQ(\epsilon_j\bQ \OP_{n,\la, s}) = \dim_\bQ(\epsilon_jR_{n,\la, s}).
\end{align}
Hence, it suffices to prove $\cA^{j\text{-alt}}_{n,\la, s}$ spans $\epsilon_j R_{n,\la, s}$.

Recall $R_{n,\la, s}$ is spanned by $\cA_{n,\la, s}$, so $\epsilon_j R_{n,\la, s}$ is spanned by $\epsilon_j \cA_{n,\la, s}$. 
Applying the recursion in Corollary~\ref{cor:BasisRecursion} for $\cA_{n,\la, s}$ exactly $j$ many times shows $\cA_{n,\la, s}$ can be partitioned as
\begin{align}
\cA_{n,\la, s} = \bigcupdot_{H \in [0,s-1]^j} \cA_{n-j,\la^{(H)}, s}\cdot  \bz_j^H.\label{eq:DisjUnionTree}
\end{align}
Applying $\epsilon_j$ to both sides of \eqref{eq:DisjUnionTree}, we see that $\epsilon_j R_{n,\la, s}$ is spanned by the union
\begin{align}\label{eq:AzUnion}
\bigcup_{H \in [0,s-1]^j} \cA_{n-j,\la^{(H)}, s}\cdot  \epsilon_j \bz_j^H.
\end{align}
Given $H\in [0,s-1]^j$, if any two entries of $H$ are equal, then we have $\epsilon_j \bz_j^H = 0$ by Remark~\ref{rmk:Vandermonde}. Therefore, the union~\eqref{eq:AzUnion} is equal to the same union restricted to those $H$ with distinct entries.
In order to prove $\cA^{j\text{-alt}}_{n,\la, s}$ spans $\epsilon_j R_{n,\la, s}$, it suffices to prove that for any $H\in [0,s-1]^j$ with distinct entries and any $\bx_{n-j}^\alpha \in \cA_{n-j,\la^{(H)}, s}$, we have that $\bx_{n-j}^\alpha \cdot \epsilon_j \bz_j^H$ is in the span of $\cA^{j\text{-alt}}_{n,\la, s}$.

Indeed, let $I = \sort_{\leq}(H)\in \cI_s^j$. 
By Lemma~\ref{lem:RemovingBoxes}, we have $\la^{(I)}\dominatedby \la^{(H)}$, so by Lemma~\ref{lem:Containment} we have $\cA_{n-j,\la^{(H)}, s}\subseteq \cA_{n-j,\la^{(I)}, s}$. Hence, we have $\bx_{n-j}^\alpha\in \cA_{n-j,\la^{(I)}, s}$. Furthermore, since $H$ is a permutation of $I$, we have $\epsilon_j \bz_j^H = \pm \epsilon_j\bz_j^{I}$, so $\bx_{n-j}^\alpha\cdot \epsilon_j \bz_j^H$ is in the span of $\cA^{j\text{-alt}}_{n,\la, s}$, and the proof is complete.
\end{proof}

Given a list of distinct variables $W = (w_1,\dots, w_v)\subseteq \bz_j$ and any subset $U = \{u_1<\dots< u_{|U|}\}\subset [v]$, let $W_U \coloneqq (w_{u_1},\dots, w_{u_{|U|}})$. Given a tuple $I \in \cI_s^v$, let $I_U \coloneqq (i_{u_1} <\dots < i_{u_{|U|}}) \in \cI_s^{|U|}$.
For any tuple $H = (h_1 < \dots < h_{|U|}) \in \cI_{s}^{|U|}$, define the monomial $W_U^{H} \coloneqq w_{u_1}^{h_1}\cdots w_{u_{|U|}}^{h_{|U|}}$. 
Recall the following lemma on elementary symmetric polynomials.
\begin{lemma}[\cite{Garsia-Procesi}, Lemma 6.1]\label{lem:GPLemma}
Let $C \subseteq \bx_n$ and $W = (w_1,\dots, w_v) \subseteq \bz_j$ be disjoint sets of variables, where $W$ comes with a total ordering. Let $I = (i_1< \dots < i_v)\in \cI_s^v$. For any $d\geq 1$, we have
\begin{align}
e_d(C)\cdot (-w_1)^{i_1}\cdots(-w_v)^{i_v} = \sum_{U\subseteq [v]} (-1)^{|U|} \sum_{H} (-1)^{\Sigma(H)}  e_{d+\Sigma(I)-\Sigma(H)}(C\cup W_U)\cdot W_U^{H},
\end{align}
where the inner sum on the right-hand side is over all tuples $H = (h_1,\dots, h_{|U|})$ of nonnegative integers, such that $h_p < i_{u_p}$ for $1\leq p\leq |U|$.
\end{lemma}

\begin{lemma}\label{lem:ExpressionInIdeal}
Let $I\in \cI_s^j$, and let $e_d(S)$ be a generator of $I_{n-j,\la^{(I)}, s}$ for some $d$ and $S\subseteq \bx_{n-j}$. We have $e_d(S\cup \bz_{a})\cdot \bz_j^I \in I_{n,\la, s}$, where $a =  p^n_{|S|+j}(\la) - p^{n-j}_{|S|}(\la^{(I)})$ and $\bz_a = \{z_1,\dots, z_a\}$.
\end{lemma}
\begin{proof}
Let $m = |S|$. Then $a = p^n_{m+j}(\la) - p^{n-j}_m(\la^{(I)})$, which is the number of boxes weakly to the right of column $n-(m+j)+1$ of the skew diagram $\la/\la^{(I)}$. In order to show $e_d(S\cup \bz_a)\cdot \bz_j^I\in I_{n,\la, s}$, it suffices to prove that $e_d(S\cup \bz_{a})\cdot z_{a+1}^{i_{a+1}}\cdots z_j^{i_j}\in I_{n,\la, s}$, where the product $z_{a+1}^{i_{a+1}}\cdots z_j^{i_j}$ is taken to be $1$ in the case $a=j$. By our assumption that $e_d(S)$ is a generator of $I_{n-j,\la^{(I)}, s}$, we have 
\begin{align}\label{eq:GenAssumption}
d> m-p^{n-j}_m(\la^{(I)}).
\end{align}

In the case $a=j$, then by combining \eqref{eq:GenAssumption} and $p^{n-j}_m(\la^{(I)}) = p^n_{m+j}(\la) - a$, we have $d>m+a-p^n_{m+a}(\la)$. Hence, $e_d(S\cup\bz_{a})\in I_{n,\la, s}$ so $e_d(S\cup \bz_a)\bz_j^I\in I_{n,\la, s}$, and we are done. 

In the case $a<j$, then applying Lemma~\ref{lem:GPLemma} with $C = S\cup \bz_{a}$ and $W = (z_{a+1},\dots, z_j)$, we have that
\begin{align}
e_d(S\cup \bz_{a}) z_{{a}+1}^{i_{{a}+1}}\cdots z_j^{i_j}
\end{align}
is an alternating sum of terms of the form
\begin{align}
e_{d+\Sigma(I)-\Sigma(H)}(S\cup \bz_{a} \cup \bz_U) \cdot \bz_U^{H},\label{eq:AlternatingTerms}
\end{align}
where $U\subseteq \{a+1,\dots, j\}$ and $H = (h_1,\dots, h_{|U|}) \in \cI_s^{|U|}$ such that $h_p < i_{u_p}$ for $1\leq p\leq |U|$. 
To complete the proof, we show that in each case $e_{d+\Sigma(I)-\Sigma(H)}(S\cup \bz_{a}\cup \bz_U)\in I_{n,\la, s}$, or equivalently that 
\begin{align}
d+\Sigma(I)-\Sigma(H) > m+{a}+|U| - p^n_{m+{a}+|U|}(\la)\label{eq:SufficientInequality}.
\end{align}

We claim that $i_{a+1} \geq \la'_{n-(m+j)+1}$. If $|\la/\la^{(I)}| = a$, then by the definition of $\la^{(I)}$, we see that $|\la/\la^{(I)}|$ is the number of elements of $I$ which are strictly less than $\ell(\la)$. Under the assumption that $|\la/\la^{(I)}|=a$, we have $i_{a+1}\geq \ell(\la)\geq \la'_{n-(m+j)+1}$. For example, let $\la = (4,3,2,2)$, $j= 4$, $I = (0,2,3,4)$, and $\la^{(I)} = (3,3,1,1)$ as shown in Figure~\ref{fig:CountingBoxes}. Taking $n = 12$, $s=5$, and $m = 7$, then $a=3$, which is the number of boxes of $\la/\la^{(I)}$ weakly to the right of column $n-(m+j)+1 = 2$. In this case, $4=i_{a+1}\geq \la_2' = 4$.

Otherwise, we have $|\la/\la^{(I)}| \geq {a}+1$, so by the definition of $a$, the $({a}+1)$th box of $\la/\la^{(I)}$ from the right must be in a column strictly to the left of column $n-(m+j)+1$ of $\la$. Therefore, the inequality $i_{{a}+1} \geq \la'_{n-(m+j)+1}$ continues to hold. For example, if we take the same $n,\la,s,j$, and $I$ as in our previous example but with $m=6$ and $a=1$, then $3 = |\la/\la^{(I)}| \geq a+1 = 2$ and $2 = i_{a+1} \geq \la_3' = 2$.

We have the string of inequalities
\begin{align}
i_j > \cdots > i_{{a}+2} > i_{{a}+1} \geq \la_{n-(m+j)+1}' \geq \la_{n-(m+j-1)+1}' \geq \cdots.
\end{align}
Therefore, we have
\begin{align}
\sum_{t\in [{a}+1,j] \setminus U} i_t \geq  \sum^{m+j}_{t = m + {a} + |U| + 1} \la_{n-t+1}' = p^n_{m+j}(\la) - p^n_{m+{a}+|U|}(\la),\label{eq:SumIndicesIneq}
\end{align}
since both sides of the inequality sum over $j-a - |U|$ many terms.
By our assumption that $i_{u_p} > h_p$ for $1\leq p\leq |U|$, we have
\begin{align}
\sum_{t\in U} i_t  \geq \Sigma(H) + |U|.\label{eq:AssumptionInequality}
\end{align}
Recalling \eqref{eq:GenAssumption}, we have
\begin{align}\label{eq:DgreaterThanStuff}
d > m- p^{n-j}_m (\la^{(I)}) = m + {a}- p^n_{m+j}(\la).
\end{align}
The inequality \eqref{eq:SufficientInequality} then follows by combining \eqref{eq:SumIndicesIneq}, \eqref{eq:AssumptionInequality}, and \eqref{eq:DgreaterThanStuff} with 
\begin{align}
\Sigma(I) \geq \sum_{t\in [{a}+1,j]\setminus U} i_t + \sum_{t\in U} i_t,
\end{align}
which completes the proof.
\end{proof}

\begin{figure}
\centering
\includegraphics[scale=0.4]{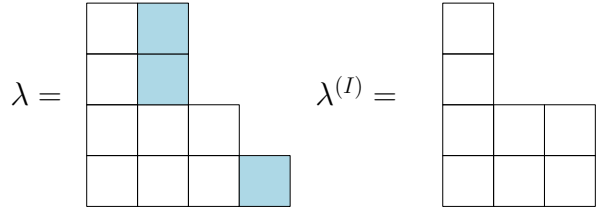}
\caption{The Young diagrams of $\la = (4,3,2,2)$ and $\la^{(I)} = (3,3,1,1)$, where $I = (0,2,3,4)$. The cells of $\la/\la^{(I)}$ are shaded in the diagram of $\la$. \label{fig:CountingBoxes}}
\end{figure}

\begin{lemma}\label{lem:MultWellDefined}
With the same hypotheses as Lemma~\ref{lem:ExpressionInIdeal}, we have
\begin{align}
e_d(S)\cdot \epsilon_j \bz_j^{I}\in \vecspan_\bQ\left( \bigcup_{\substack{H\in \cI_s^j,\\ \Sigma(H)\geq \Sigma(I)+1}} \bQ[\bx_{n-j}]\cdot \epsilon_j \bz_j^{H}\right) + \epsilon_j I_{n,\la, s}.\label{eq:MultWellDefined}
\end{align}
\end{lemma}
\begin{proof}
We have that
\begin{align}
e_d(S) = e_d(S\cup\bz_a) - \sum_{t = 1}^d e_{d-t}(S)\cdot e_t(\bz_a),
\end{align}
so multiplying both sides by $\bz_j^{I}$ and applying $\epsilon_j$, we have
\begin{align}
e_d(S)\cdot \epsilon_j \bz_j^{I} = \epsilon_j (e_d(S\cup\bz_a) \cdot \bz_j^{I}) - \sum_{t=1}^d e_{d-t}(S)\cdot \epsilon_j (e_t(\bz_a)\cdot \bz_j^{I}).\label{eq:ETimesZIdentity}
\end{align}
By Lemma~\ref{lem:ExpressionInIdeal}, we have $e_d(S\cup\bz_a)\cdot \bz_j^{I}\in I_{n,\la, s}$. Therefore, by \eqref{eq:ETimesZIdentity}, it suffices to prove that each polynomial of the form $e_{d-t}(S)\cdot \epsilon_j (e_t(\bz_a)\cdot \bz_j^{I})$ is in the set on the right-hand side of \eqref{eq:MultWellDefined}.

For each $t\geq 1$, consider the expansion of $e_t(\bz_a)\cdot \bz_j^{I}$ on the right-hand side of \eqref{eq:ETimesZIdentity} into monomials. Each term in the expansion is of one of the following types: (1) a monomial in $\bz_j$ such that two of the exponents agree, (2) a monomial in $\bz_j$ whose largest exponent is $s$, or (3) a monomial of the form $\pi\cdot \bz_j^{H}$ for some $\pi\in S_{\{n-j+1,\dots, n\}}$ and some $H\in \cI_s^j$ such that $\Sigma(H)\geq \Sigma(I)+1$. Monomials of the first type are sent to 0 by the operator $\epsilon_j$, monomials of the second type are elements of $I_{n,\la, s}$, and for a monomial of the third type, we have $\epsilon_j (\pi \cdot \bz_j^H) = \pm \epsilon_j \bz_j^H$. Therefore, each term in the sum on the right-hand side of \eqref{eq:ETimesZIdentity} is in the set on the right-hand side of \eqref{eq:MultWellDefined}. Hence, $e_d(S)\cdot \epsilon_j\bz_j^I$ is in the set as well.
\end{proof}

\begin{example}
To illustrate the proof of Lemma~\ref{lem:MultWellDefined}, let $n = 12$, $\la = (4,3,2,2)$, $s = 5$, $d = 5$, $S = \{x_1,\dots, x_6\}$, and $I = (0,2,3,4)$. Then $j = 4$, $m = 6$, and $a = 1$. We have the identity
\begin{align}
e_5(S) = e_5(x_1,\dots, x_6,x_9) - e_4(S)\cdot e_1(x_9).
\end{align}
Multiplying through by $\bz_j^I = x_{10}^2 x_{11}^3 x_{12}^4$ and applying $\epsilon_4$,
\begin{align}
e_5(S) \cdot \epsilon_4(x_{10}^2 x_{11}^3 x_{12}^4) = \epsilon_4(e_5(x_1,\dots, x_6,x_9)\cdot x_{10}^2 x_{11}^3 x_{12}^4) - e_4(S)\cdot \epsilon_4(x_9 x_{10}^2 x_{11}^3 x_{12}^4),
\end{align}
where $\epsilon_4(e_5(x_1,\dots, x_6,x_9)\cdot x_{10}^2 x_{11}^3 x_{12}^4) \in \epsilon_4 I_{n,\la,s}$. As in the proof Lemma~\ref{lem:MultWellDefined}, the monomial $x_9 x_{10}^2 x_{11}^3 x_{12}^4$ is of type (3), where $\pi$ is the identity permutation and $H = (1,2,3,4)$, hence $e_4(S)\cdot \epsilon_4(x_9 x_{10}^2 x_{11}^3 x_{12}^4)\in \bQ[\bx_{n-j}]\cdot \epsilon_j \bz_j^H$.
\end{example}

\begin{lemma}\label{lem:AntisymmIso}
We have an isomorphism of graded $S_{n-j}$-modules
\begin{align}
\epsilon_j R_{n,\la, s} \cong \bigoplus_{r\geq 0} V_r,
\end{align}
where $V_r$ is defined in~\eqref{eq:VDef}.
\end{lemma}
\begin{proof}
Given $r\geq 0$, let
\begin{align}
U_r &\coloneqq \bigoplus_{\substack{I\in \cI_s^j,\\\Sigma(I) = r}} \bQ[\bx_{n-j}]\otimes \bQ\{\epsilon_j\bz_j^{I}\},\label{eq:UDef}\\
W_{\geq r} &\coloneqq \vecspan_{\bQ} \left(\bigcup_{\substack{I\in \cI_s^j,\\\Sigma(I) \geq r}} \bQ[\bx_{n-j}]\cdot \epsilon_j\bz_j^{I}\right)\subseteq \epsilon_j R_{n,\la, s},\label{eq:WDef}\\
W_r &\coloneqq W_{\geq r}/W_{\geq r+1}.
\end{align}
By Lemma~\ref{lem:AltBasis}, we have $W_{\geq 0} = \epsilon_j R_{n,\la, s}$, so the subspaces $W_{\geq r}$ of $\epsilon_j R_{n,\la, s}$ form a descending filtration of the space $\epsilon_j R_{n,\la, s}$. Observe that the spaces $U_r$, $V_r$, $W_{\geq r}$, and $W_r$ are $\bQ[\bx_{n-j}]$-modules. Furthermore, each of them has the structure of a graded $S_{n-j}$-module, where $S_{n-j}$ acts on the variables $\bx_{n-j}$ and the grading is by total degree in the variables $\bx_n = \bx_{n-j}\cup \bz_j$.  
Since $W_{\geq r}$ is a filtration of $\epsilon_j R_{n,\la, s}$ which respects the graded $S_{n-j}$-module structure, we have that $\epsilon_j R_{n,\la, s} \cong \bigoplus_{r\geq 0} W_r$ as graded $S_{n-j}$-modules. 
Therefore, it suffices to prove that $V_r\cong W_r$ as graded $S_{n-j}$-modules.

For $r\geq 0$, let
\begin{align}
\widetilde\mu_r : U_r \to W_r
\end{align}
be the map induced by sending $f(\bx_{n-j})\otimes \epsilon_j \bz_j^{I}$ to the product $f(\bx_{n-j})\cdot \epsilon_j \bz_j^{I}$. Then $\widetilde \mu_r$ is a homomorphism of $\bQ[\bx_{n-j}]$-modules, as well as a homomorphism of graded $S_{n-j}$-modules. For each $I\in \cI_s^j$ such that $\Sigma(I) = r$ and each generator of $I_{n-j,\la^{(I)}, s}$ of the form $e_d(S)$, we have that $\widetilde\mu_r(e_d(S)\otimes \epsilon_j \bz_j^{I})=0$ by Lemma~\ref{lem:MultWellDefined}. We also have that $\widetilde\mu_r(x_i^s \otimes \epsilon_j \bz_j^{I})=0$ for all $i\leq n-j$ since $x_i^s \cdot\epsilon_j\bz_j^I = \epsilon_j(x_i^s \cdot \bz_j^I)\in \epsilon_j I_{n,\la, s}$. Since $\widetilde\mu_r$ is a homomorphism of $\bQ[\bx_{n-j}]$-modules, the map $\widetilde\mu_r$ descends to a $S_{n-j}$-module homomorphism $\mu_r : V_r \to W_r$.
By Theorem~\ref{thm:MonomialBasis}, the set $\cA_{n-j,\la^{(I)}, s}$ is a basis of $R_{n-j,\la^{(I)}, s}$, so 
\begin{align}\label{eq:BasisOfVr}
\bigcupdot_{\substack{I\in \cI_s^j,\\ \Sigma(I)=r}} \cA_{n-j,\la^{(I)}, s}\otimes \epsilon_j \bz_j^{I}
\end{align}
 represents a basis of $V_r$. Furthermore, by Lemma~\ref{lem:AltBasis} we have that $\cA^{j\text{-alt}}_{n,\la, s}$ represents a basis of $\epsilon_j R_{n,\la, s}$, hence
 \begin{align}\label{eq:BasisOfWr}
\bigcupdot_{\substack{I\in \cI_s^j,\\ \Sigma(I)=r}} \cA_{n-j,\la^{(I)}, s}\cdot \epsilon_j \bz_j^{I}
\end{align}
represents a basis of $W_r$. Since $\mu_r$ maps a basis to a basis, it is an isomorphism of graded $S_{n-j}$-modules, hence $V_r\cong W_r$ as graded $S_{n-j}$-modules. This completes the proof.
\end{proof}

\begin{proof}[Proof of Theorem~\ref{thm:AntisymmFrobqRecursion}]
The identity \eqref{eq:AntisymmFrob} implies the first equality in Theorem~\ref{thm:AntisymmFrobqRecursion}. By Lemma~\ref{lem:AntisymmIso}, we have the equality of formal power series
\begin{align}
\Frobq(\epsilon_j R_{n,\la, s}) = \sum_{I\in \cI_s^j} q^{\Sigma(I)} \Frobq(R_{n-j,\la^{(I)}, s}),
\end{align}
which is the second equality in Theorem~\ref{thm:AntisymmFrobqRecursion}.
\end{proof}

	\subsection{Exact sequences}
	
	In this subsection, we show that the rings $R_{n,\la, s}$ fit into certain exact sequences of $S_n$-modules. We use these exact sequences to obtain identities for the graded Frobenius characteristic of $R_{n,\la, s}$.

\begin{lemma}\label{lem:ExactSequence}
Let $k<n$ be distinct positive integers, and let $\la\in \Par(k, s)$ with $\ell(\la)<s$. Then there is an exact sequence of $S_n$-modules
\begin{align}
0 \to R_{n,\la, s-1} \to R_{n,\la, s} \to R_{n,\la\cdot (1), s}\to 0
\end{align}
such that the first map shifts degree by $n-k$ and the second map is degree-preserving. Equivalently, we have the identity
\begin{align}\label{eq:RemovingOneZeroFormula}
\Frobq(R_{n,\la, s}) = \Frobq(R_{n,\la\cdot (1), s}) + q^{n-k}\Frobq(R_{n,\la, s-1}).
\end{align}
\end{lemma}
\begin{proof}
It can be checked that $I_{n,\la\cdot (1), s} = I_{n,\la, s} + \langle e_{n-k}(\bx_n)\rangle$. Hence, we have the quotient map $\pi : R_{n,\la, s}\to R_{n,\la\cdot (1), s}$ with $\ker\pi = \langle e_{n-k}(\bx_n)\rangle$. 
To construct a map $R_{n,\la, s-1}\to R_{n,\la, s}$, first define the map
\begin{align}
\widetilde\phi : \bQ[\bx_n]\to R_{n,\la, s}
\end{align}
of $\bQ[\bx_n]$-modules given by multiplication by $e_{n-k}(\bx_n)$. 

 We claim that $I_{n,\la, s-1}\subseteq \ker(\widetilde\phi)$. Since $I_{n,\la, s}$ and $I_{n,\la, s-1}$ have the same elementary symmetric polynomial generators, it suffices to show that $\widetilde\phi(x_i^{s-1}) = 0$ for all $i\leq n$, or equivalently that $x_i^{s-1}e_{n-k}(\bx_n)\in I_{n,\la, s}$. By symmetry, it suffices to show that $x_n^{s-1}e_{n-k}(\bx_n)\in I_{n,\la, s}$. Observe that $x_n^{s-1} e_{n-k}(\bx_n) \equiv x_n^{s-1} e_{n-k}(\bx_{n-1}) \mod I_{n,\la, s}$. Furthermore, since $e_{n-k}(\bx_{n-1})$ is a generator of $I_{n-1,\la, s}$, then by Lemma~\ref{lem:TechnicalLemma}(d) applied to $S = \bx_{n-1}$ and $i=s$, we have $x_n^{s-1}e_{n-k}(\bx_{n-1})\in  x_n^s\bQ[\bx_n] + I_{n,\la, s} = I_{n,\la, s}$. Hence, $x_n^{s-1} e_{n-k}(\bx_n)\in I_{n,\la, s}$, so $I_{n,\la, s-1} \subseteq \ker(\widetilde\phi)$ as claimed.

Therefore, $\widetilde\phi$ descends to a map
\begin{align}
\phi : R_{n,\la, s-1}\to R_{n,\la, s}
\end{align}
whose image is exactly $\langle e_{n-k}(\bx_n)\rangle = \ker\pi$.  Therefore, the sequence
\begin{align}
R_{n,\la, s-1}\stackrel{\phi}{\to} R_{n,\la, s} \stackrel{\pi}{\to} R_{n,\la\cdot (1), s}\to 0
\end{align}
is exact, where $\phi$  shifts degree by $n-k$. Furthermore, we have
\begin{align}\label{eq:SizeEquality}
|\OP_{n,\la, s}| = |\OP_{n,\la\cdot (1), s}| + |\OP_{n,\la, s-1}|,
\end{align}
which can be seen by partitioning $\OP_{n,\la, s}$ into a disjoint union of two subsets by considering those elements $(B_1|\cdots | B_s)$ such that $B_{\ell(\la)+1} \neq \emptyset$ and those such that $B_{\ell(\la)+1} = \emptyset$. By Theorem~\ref{thm:UngradedFrobChar}, then \eqref{eq:SizeEquality} implies
\begin{align}
\dim_\bQ(R_{n,\la, s}) = \dim_\bQ(R_{n,\la\cdot (1), s}) + \dim_\bQ(R_{n,\la, s-1}).
\end{align}
Hence, $\phi$ is injective, and we have an exact sequence
\begin{align}
0\to R_{n,\la, s-1}\stackrel{\phi}{\to} R_{n,\la, s} \stackrel{\pi}{\to} R_{n,\la\cdot (1), s}\to 0.
\end{align}
To complete the proof, observe that $\phi$ and $\pi$ are $S_n$-module homomorphisms.
\end{proof}

\begin{theorem}\label{thm:RemovingZerosFrobChar}
Let $k\leq n$ be positive integers and let $\la\in \Par(k,s)$ such that $\ell(\la) < s$. We have
\begin{align}\label{eq:RemovingZerosFormula}
\Frobq(R_{n,\la, s}) = \sum_{m\geq 0} q^{(s-\ell(\la)-m)(n-k-m)} \multibinom{s-\ell(\la)}{m}_q \Frobq\left(R_{n,\la\cdot (1^m), \ell(\la) + m}\right),
\end{align}
where $\multibinom{s-\ell(\la)}{m}_q = 0$ for $m>s-\ell(\la)$ and $\Frobq\left(R_{n,\la\cdot (1^m), \ell(\la)+m}\right)=0$ for $m>n-k$.
\end{theorem}
\begin{proof}
Proceed by induction on $s-\ell(\la)$. The base case where $s -\ell(\la)=1$ holds by Lemma~\ref{lem:ExactSequence}. Fix $a>1$, and assume by way of induction that \eqref{eq:RemovingZerosFormula} holds for $s$ and $\la$ such that $0< s - \ell(\la) < a$. Let $k\leq n$ be positive integers, and let $\la\in \Par(k, s)$ such that $s - \ell(\la) = a$. By \eqref{eq:RemovingOneZeroFormula}, we have
\begin{align}
\Frobq(R_{n,\la, s}) 	= \Frobq\left(R_{n,\la\cdot (1), s}\right) + q^{n-k}\Frobq\left(R_{n,\la, s-1}\right).
\end{align}
Applying our inductive hypothesis, we have 
\begin{multline}
\Frobq(R_{n,\la, s}) = \sum_{m\geq 0} q^{(s-\ell(\la)-1-m)(n-k-1-m)}\multibinom{s-\ell(\la)-1}{m}_q \Frobq\left(R_{n,\la\cdot (1^{m+1}), \ell(\la)+m+1}\right) \\
+ q^{n-k}\left(\sum_{p\geq 0} q^{(s-1-\ell(\la)-p)(n-k-p)}\multibinom{s-\ell(\la)-1}{p}_q \Frobq\left(R_{n,\la\cdot (1^p), \ell(\la)+p}\right)\right),
\end{multline}
which we may rewrite as
\begin{multline}
\Frobq(R_{n,\la, s}) = q^{(s-\ell(\la))(n-k)} \Frobq\left(R_{n,\la, \ell(\la)}\right)\\
 + \sum_{m\geq 1} \left(q^{(s-\ell(\la)-m)(n-k-m)}\multibinom{s-\ell(\la)-1}{m-1}_q \right.\\
 \left.+ q^{(s-1-\ell(\la)-m)(n-k-m) + (n-k)}\multibinom{s-\ell(\la)-1}{m}_q\right)\Frobq\left(R_{n,\la\cdot (1^m), \ell(\la)+m}\right).\label{eq:RemovingZerosAlmostThere}
\end{multline}
The identity $\qbinom{a}{b}_q = \qbinom{a-1}{b-1}_q + q^b\qbinom{a-1}{b}_q$ with $a = s-\ell(\la)$ and $b=m$ shows that the coefficient of $\Frobq\left(R_{n,\la \cdot(1^m), \ell(\la)+m}\right)$ on the right-hand side of \eqref{eq:RemovingZerosAlmostThere} is equal to
\begin{align}
q^{(s-\ell(\la)-m)(n-k-m)}\multibinom{s-\ell(\la)}{m}_q.
\end{align}
Hence, the right-hand side of \eqref{eq:RemovingZerosAlmostThere} is equal to the right-hand side of \eqref{eq:RemovingZerosFormula}, which completes the induction.
\end{proof}

	\section{Inversions and diagonal inversions}\label{sec:Statistics}
	
	In this section, we define inversion and diagonal inversion statistics on labeled objects which are in bijection with $\OP_{n,\la, s}$ and use them to give formulas for the Hilbert series and graded Frobenius characteristic of $R_{n,\la, s}$. In Subsection~\ref{subsec:Fcoinv}, we define extended column-increasing fillings and the statistics $\finv$ and $\fdinv$. In Subsection~\ref{subsec:Equidistribution}, we prove that the statistics $\finv$ and $\fdinv$ are equidistributed and their monomial generating functions are equal.  In Subsection~\ref{subsec:ProofOfCoinvQSym}, we then prove our Hilbert series and graded Frobenius characteristic formulas for $R_{n,\la, s}$ in terms of these statistics.

	\subsection{The statistics $\finv$ and $\fdinv$}\label{subsec:Fcoinv}

Given $(B_1|\cdots |B_s)\in \OP_{n,\la, s}$, define a \emph{standard extended column-increasing filling of $\la'$} as follows. First, define a column-increasing filling of the Young diagram of $\la'$ by labeling the $i$th column with the $\la_i$ smallest elements of $B_i$. Then, for each $i$, place the rest of the elements of $B_i$ in their own cells vertically below the $i$th column of the diagram in increasing order from top to bottom. We call the labels below the diagram the \emph{basement labels} and the cells containing them the \emph{basement cells}.

Let $\FCI_{n, \la, s}$ be the set of standard extended column-increasing fillings of $\la'$. See Figure~\ref{fig:FloatingDiagram} for the standard extended column-increasing filling associated to the ordered set partition $(1,3,5,7,8\,|\,2,9,10\,|\,4,6)\in \OP_{10,(3,2), 3}$, where $4, 6, 7, 8$, and $10$  are basement labels.

\begin{figure}
\centering
\includegraphics[scale=0.75]{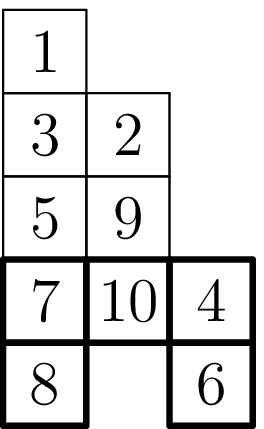}
\caption{The standard extended column-increasing filling associated to $(1,3,5,7,8\,|\,2,9,10|\,4,6)\in \OP_{10,(3,2), 3}$, where the basement cells are bolded.\label{fig:FloatingDiagram}}
\end{figure}

Given a composition $\alpha\in \Comp(k, s)$, we define an \emph{extended column-increasing filling $\float$ of $\dgprime(\alpha)$} to consist of
\begin{itemize}
\item A diagram $\diagram(\float) = \dgprime(\alpha) \cup \basement(\float)$, where $\basement(\float)$ is a possibly empty collection of \emph{basement cells} in columns $1\leq i \leq s$ and rows $j\leq 0$, such that in each column $i$ the basement cells are top justified so that the top basement cell is at coordinates $(i,0)$,
\item A labeling of the cells of $\diagram(\float)$ with positive integers which weakly increases down each column.
\end{itemize}
We denote by $\sigma(\float)$ the column-increasing filling obtained by restricting $\float$ to $\dgprime(\alpha)$. Given a cell $(i,j)\in \diagram(\float)$, we denote by $\float_{i,j}$ the label of $\float$ in the cell $(i,j)$.
Let $\fci_{n, \alpha, s}$ be the set of extended column-increasing fillings $\float$ of $\dgprime(\alpha)$ with $n$ cells. Let $\FCI_{n, \alpha, s}$ be the subset of $\fci_{n, \alpha, s}$ consisting of \emph{standard} extended column-increasing fillings which use the letters in $[n]$ without repetition. See Figure~\ref{fig:ExtendedCI} for an example of an extended column-increasing filling in $\fci_{11,(2,3,0,1),4}$.

\begin{figure}
\centering
\includegraphics[scale=0.75]{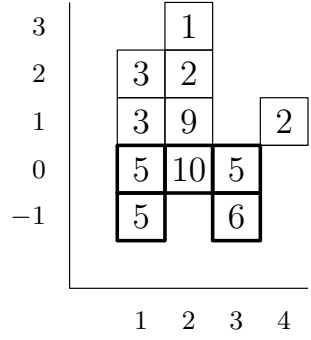}
\caption{An extended column-increasing filling in $\fci_{11,(2,3,0,1),4}$. We have labeled the rows and columns for the reader's aid.\label{fig:ExtendedCI}}
\end{figure}

 The \emph{reading word} of $\float$ is the concatenation $\rw(\float) \coloneqq \rw(\sigma(\float))\, v$, where $v$ is the word obtained by reading the labels of $\basement(\float)$ left to right across each row, from top to bottom. The \emph{inversion reading word} of $\float$ is the concatenation $\irw(\float) \coloneqq \rw(\sigma(\float))\,w$, where $w$ is the word obtained by reading the labels of $\basement(\float)$ down each column, starting with column $1$ and ending with column $s$. The ordering of the cells of $\float$ corresponding to $\rw(\float)$ and $\irw(\float)$ are the \emph{reading order} and the \emph{inversion reading order} of $\float$, respectively. For the standard extended column-increasing filling $\float$ in Figure~\ref{fig:FloatingDiagram}, we have $\rw(\float) = 1\,3\,2\,5\,9\,7\,10\,4\,8\,6$ and $\irw(\float) = 1\, 3\, 2\, 5\, 9\, 7\, 8\, 10\, 4\, 6$.

\begin{definition}\label{def:DefOfInv}
Given $\float\in \fci_{n, \alpha, s}$, an \emph{inversion} of $\float$ is one of the following,
\begin{itemize}
\item[(I1)] A diagonal inversion of $\sigma(\float)$, as defined in Subsection~\ref{subsec:HLFunctions},
\item[(I2)] A pair $((i,1), (i',j'))$ where $(i,1)\in \dgprime(\alpha)$ and $(i',j')\in \basement(\float)$, and such that $i > i'$ and $\float_{i,1} > \float_{i',j'}$,
\item[(I3)] A pair $(i, (i',j'))$, where $(i',j')\in \basement(\float)$ and $i$ is an integer such that ${1\leq i < i'}$.
\end{itemize}
Let $\finv(\float)$ be the number of inversions of $\float$.
\end{definition}
For $\float$ in Figure~\ref{fig:ExtendedCI}, we have the following inversions,
\begin{align*}
\text{Type (I1): } &  ((1,2),(2,2)),\, ((1,1),(4,1)), \text{ and } ((2,1),(4,1))\\
\text{Type (I2): } & ((2,1),(1,0)) \text{ and }((2,1),(1,-1))\\
\text{Type (I3): } & (1,(2,0)),\, (1,(3,0)),\, (2,(3,0)),\, (1,(3,-1)),\text{ and }(2,(3,-1)).
\end{align*}
In total, we have $\finv(\float) = 10$.

We also introduce a diagonal inversion statistic on extended column-increasing fillings.  Let an \emph{$\alpha$-attacking pair} be a pair of coordinates $((i,j), (i',j'))$ such that $1\leq i\leq s$, $1\leq i' \leq s$, $ j \leq \alpha_i$, and $j'\leq \alpha_{i'}$, and such that one of the following holds,
\begin{itemize}
\item We have $j=j'$ and $i < i'$,
\item We have $j = j'+1$ and $i>i'$. 
\end{itemize}
\begin{definition}\label{def:DefOfDinv}
A \emph{diagonal inversion} of $\float$ is an $\alpha$-attacking pair $((i,j),(i',j'))$ such that one of the following holds,
\begin{itemize}
\item[(D1)] We have $(i,j),(i',j')\in \diagram(\float)$ such that $\float_{i,j} > \float_{i',j'}$,
\item[(D2)] We have $j,j'\leq 0$, $(i,j)\notin \diagram(\float)$, and $(i',j') \in \diagram(\float)$.
\end{itemize}
Let $\fdinv(\float)$ be the number of diagonal inversions of $\float$.
\end{definition}
For $\float$ in Figure~\ref{fig:ExtendedCI}, we have the following diagonal inversions,
\begin{align*}
\text{Type (D1): } & ((1,2),(2,2)),\, ((1,1),(4,1)),\, ((2,1),(4,1)),\text{ and } ((2,1),(1,0)),\\
 &   ((2,0),(3,0))\text{ and }((2,0),(1,-1)) \\
\text{Type (D2): } & ((4,0),(1,-1)),\, ((4,0),(3,-1)),\text{ and } ((2,-1),(3,-1)).
\end{align*}
In total, we have $\fdinv(\float) = 9$. 

\begin{remark}
The notion of an extended column-increasing filling is a variation of the fillings introduced in~\cite{Rhoades-Yu-Zhao} during the preparation of this article. To translate between our conventions and theirs, simply flip our labelings across the horizontal axis and convert each basement label into a floating number. Under this identification between $\OP_{n,\la, s}$ and $\FCI_{n,\la, s}$, the inversion statistic in Definition~\ref{def:DefOfInv} is similar to, but not the same as, the $\coinv$ statistic in~\cite{Rhoades-Yu-Zhao}. In  particular, $\coinv$ uses a different definition of attacking pair, where $((i,j),(i',j'))$ forms an attacking pair in $\la'$ if $j = j'$ and $i < i'$, or $j' = j + 1$ and $i > i'$. The condition (I1) in Definition~\ref{def:DefOfInv} is then replaced with the condition that this alternate type of attacking pair contributes to the number of coinversions if $\float_{i,j} < \float_{i',j'}$.
\end{remark}

\subsection{Equidistribution of $\inv$ and $\dinv$}\label{subsec:Equidistribution}

We need the following equidistribution theorem for the statistics $\finv$ and $\fdinv$. Before proving the theorem, we prove several lemmata.
\begin{theorem}\label{thm:Equidistrib}
For $\la\in \Par(k, s)$, we have the identity of multivariate generating functions,
\begin{align}\label{eq:Equidistrib}
\sum_{\float\in \fci_{n, \la, s}} q^{\finv(\float)}\bx^\float = \sum_{\float\in \fci_{n, \la, s}} q^{\fdinv(\float)} \bx^\float,
\end{align}
where $\bx^\float\coloneqq \prod_{i\geq 1} x_i^{\# i\text{'s in }\float}$.
\end{theorem}

Our strategy for proving Theorem~\ref{thm:Equidistrib} is inspired by generalizations of the Carlitz bijection constructed by Gillespie~\cite{GillespieThesis}, Rhoades-Wilson~\cite{Rhoades-Wilson-Stirling}, Rhoades-Yu-Zhao~\cite{Rhoades-Yu-Zhao}, and Wilson~\cite{Wilson}. Since~\cite{Rhoades-Yu-Zhao} uses our theorems to prove their results, we give independent proofs in order to avoid creating a cycle in the logical flow of the proofs.

Let $\gamma = (\gamma_1,\dots, \gamma_m)$ be a composition of $n$ with $\gamma_i >0$ for all $i$. A sequence $c = (c_1,\dots, c_n)$ of nonnegative integers is \emph{$\gamma$-weakly decreasing} if every subsequence
\begin{align*}
c_{\gamma_1+\cdots + \gamma_i}, c_{\gamma_1+\cdots +\gamma_i + 1},\dots, c_{\gamma_1 + \cdots +\gamma_i+\gamma_{i+1}-1}
\end{align*}
is weakly decreasing. For example, if $\gamma = (3,2,1,1)$, then the sequence $c = (5,5,2,6,3,1,2)$ is $\gamma$-weakly decreasing.
Let 
\begin{align}
\cC_{n, \la, s}(\gamma) \coloneqq \{c \in \cC_{n,\la, s} \st c \text{ is }\gamma\text{-weakly decreasing}\}.
\end{align}
Observe that when $\gamma = (1^n)$, then $\cC_{n,\la, s} (1^n) = \cC_{n, \la, s}$. For $0\leq i \leq s-1$ and any $\la$, let
\begin{align}
\cC^{(i)}_{n-1, \la, s}(\gamma)\coloneqq \{(c_1,\dots, c_{n-1})\in \cC_{n-1, \la, s}(\gamma) \st c_{n-1}\geq i \}.
\end{align}

\begin{lemma}\label{lem:CompDecomp}
Let $\gamma = (\gamma_1,\dots, \gamma_m)$ be a composition of $n$ with $\gamma_i >0$ for all $i$. If $\gamma_m > 1$, we have
\begin{align}\label{eq:CompEq1}
\cC_{n,\la, s}(\gamma) = \bigcupdot_{i = 0}^{\ell(\la)-1}  \cC^{(i)}_{n-1, \la^{(i)}, s}(\gamma_1,\dots, \gamma_m - 1)\cdot (i) \cupdot \bigcupdot_{i=\ell(\la)}^{s-1} \cC^{(i)}_{n-1,\la, s}(\gamma_1,\dots, \gamma_m - 1)\cdot (i).
\end{align}
If $\gamma_m = 1$, we have
\begin{align}\label{eq:CompEq2}
\cC_{n, \la, s}(\gamma) = \bigcupdot_{i=0}^{\ell(\la)-1} \cC_{n-1,\la^{(i)}, s}(\gamma_1,\dots, \gamma_{m-1})\cdot (i) \cupdot \bigcupdot_{i=\ell(\la)}^{s-1} \cC_{n-1, \la, s}(\gamma_1,\dots, \gamma_{m-1}) \cdot (i).
\end{align}
\end{lemma}

\begin{proof}
In the case where $\gamma_m > 1$, then for $i < \ell(\la)$, the set $\cC^{(i)}_{n-1,\la^{(i)}, s}(\gamma_1,\dots, \gamma_m-1)\cdot (i)$ is the subset of $\cC_{n-1, \la^{(i)}, s}\cdot (i)$ of $\gamma$-weakly decreasing sequences, by definition. Similarly, for $i\geq \ell(\la)$, the set $\cC^{(i)}_{n-1,\la, s}(\gamma_1,\dots, \gamma_m-1)\cdot (i)$ is the subset of $\cC_{n-1, \la, s}\cdot (i)$ of $\gamma$-weakly decreasing sequences, by definition.
Therefore, the partition \eqref{eq:CompEq1} follows immediately from Lemma~\ref{lem:ShuffleRecursion} by restricting both sides to $\gamma$-weakly decreasing sequences. The case where $\gamma_m=1$ follows from Lemma~\ref{lem:ShuffleRecursion} by similar reasoning.
\end{proof}

Given a word $w$ on positive integers, we say that it has \emph{content} $\gamma$ if $i$ appears as a letter in $w$ exactly $\gamma_{i}$ many times for $1\leq i\leq m$. Let $\fci_{n, \alpha, s}(\gamma)$ be the set of $\float\in \fci_{n, \alpha, s}$ such that $\rw(\float)$ has content $\gamma$.

To each $\float\in \fci_{n,\alpha, s}(\gamma)$, we assign an \emph{inversion code} and a \emph{diagonal inversion code}, defined next. These codes are inspired by the generalized Carlitz codes in~\cite{GillespieThesis} and the coinversion codes in~\cite{Rhoades-Wilson-Stirling,Rhoades-Yu-Zhao}. 

Given $\float\in \fci_{n, \alpha, s}(\gamma)$, let $(i_1,j_1),(i_2,j_2),\dots, (i_n, j_n)$ be the cells of $\diagram(\float)$, listed so that ${\float_{i_1,j_1}\leq \float_{i_2,j_2}\leq \cdots \leq \float_{i_n, j_n}}$, and breaking ties in \textbf{inversion} reading order. For example, for $\float$ in Figure~\ref{fig:ExtendedCI}, the cells of $\diagram(\float)$ are listed in the order
\begin{align}\label{eq:CellOrderInversion}
(2,3),\, (2,2),\, (4,1),\, (1,2),\, (1,1),\, (1,0),\, (1,-1),\, (3,0),\, (3,-1),\, (2,1),\, (2,0).
\end{align}

Let $c_p$ be the total number of inversions of $\float$ of the following types,
\begin{itemize}
\item Type (I1) inversions of the form $((i,j),(i_p, j_p))$ for some $i$ and $j$,
\item Type (I2) inversions of the form $((i,1),(i_p, j_p))$ for some $i$,
\item Type (I3) inversions of the form $(i, (i_p, j_p))$ for some $i$.
\end{itemize} 
Define the \emph{inversion code} of $\float$ to be $\invcode^\alpha(\float) \coloneqq (c_n, c_{n-1},\dots, c_1)$. Observe that the sum of the entries of $\invcode^\alpha(\float)$ is equal to $\finv(\float)$. For $\float$ in Figure~\ref{fig:ExtendedCI}, we have 
\begin{align}
\invcode^\alpha(\float) = (0,1,2,0,0,1,1,2,2,0,1),
\end{align}
where $\alpha = (2,3,0,1)$. 

Given $\float\in \fci_{n, \alpha, s}(\gamma)$, let $(i_1,j_1), (i_2,j_2), \cdots, (i_n, j_n)$ be the cells of $\diagram(\float)$, listed so that ${\float_{i_1,j_1}\leq \float_{i_2,j_2}\leq \cdots \leq \float_{i_n, j_n}}$, and breaking ties in reading order. For example, for $\float$ in Figure~\ref{fig:ExtendedCI}, the cells of $\diagram(\float)$ are listed in the same order as in \eqref{eq:CellOrderInversion}, except with the three cells labeled by $5$ listed in the order $(1,0), (3,0), (1,-1)$. Let $d_p$ be the total number of diagonal inversions of type (D1) and (D2) of the form $((i,j),(i_p, j_p))$ for some $i$ and $j$.
Define the \emph{diagonal inversion code} to be $\dinvcode^\alpha(\float) \coloneqq (d_n, d_{n-1},\dots, d_1)$. For $\float$ in Figure~\ref{fig:ExtendedCI}, we have 
\begin{align}
\dinvcode^\alpha(\float) = (0,1,2,0,0,1,1,2,2,0,0),
\end{align}
where $\alpha = (2,3,0,1)$.

\begin{lemma}\label{lem:InvcodeWellDefined}
Sending $\float\in \fci_{n,\alpha, s}(\gamma)$ to its inversion code gives a map
\begin{align}
\invcode^\alpha : \fci_{n, \alpha, s}(\gamma) \to \cC_{n, \sort(\alpha), s}(\rev(\gamma)).
\end{align}
\end{lemma}

\begin{proof}
Let $\la = \sort(\alpha)$ for convenience. It suffices to show that if $\float\in \fci_{n, \alpha, s}(\gamma)$, then $(c_n, c_{n-1},\dots, c_1)\in \cC_{n, \la, s}(\rev(\gamma))$. We first show that $(c_n, c_{n-1},\dots, c_1)\in \cC_{n,\la, s}$. 

Let $(i_1,j_1), (i_2, j_2), \dots,(i_n, j_n)$ be the cells of $\float$, listed so that $\float_{i_1,j_1}\leq \cdots\leq \float_{i_n, j_n}$ and then breaking ties in inversion reading order.  Suppose the $\la_m'$ cells in the $m$th row of $\dgprime(\alpha)$ are the $p_1$th, $p_2$th,\dots, $p_{\la_m'}$th cells in the list, with $p_1 < p_2 < \cdots < p_{\la_m'}$. Then we have $c_{p_t} \leq \la'_m - t$ for $1\leq t\leq \la_m'$.  Therefore, $(c_{p_1},\dots, c_{p_{\la_m'}}) \subseteq \beta^m(\la) = (0,\dots, \la_m'-1)$. Furthermore, for each $p$ such that $(i_p, j_p)\in \basement(\float)$, we have $c_p\leq s-1$. Therefore, $(c_n, c_{n-1}, \dots, c_1)$ is contained in a $(n,\la, s)$-shuffle, hence it is in $\cC_{n,\la, s}$.

To complete the proof, it suffices to show that $(c_n, c_{n-1},\dots, c_1)$ is $\rev(\gamma)$-weakly decreasing. Equivalently, we show that for $1\leq p\leq n-1$ such that $\float_{i_p, j_p}=\float_{i_{p+1}, j_{p+1}}$, then we have $c_p\leq c_{p+1}$. Given $p$ such that $\float_{i_p, j_p}=\float_{i_{p+1}, j_{p+1}}$, define an injection from the set of inversions counted toward $c_p$ to the set of inversions counted toward $c_{p+1}$ as follows. 

First, suppose that both $(i_p, j_p)$ and $(i_{p+1},j_{p+1})$ are cells of $\dgprime(\alpha)$. In this case, all inversions counting toward $c_p$ and $c_{p+1}$ are among labels in $\sigma(\float)$. Let $((a, b),(i_p, j_p))$ be an inversion of $\sigma(\float)$, so $\float_{a, b} > \float_{i_p, j_p}=\float_{i_{p+1}, j_{p+1}}$. Since $(a, b)$ appears before $(i_{p+1}, j_{p+1})$ in reading order, then we have $b\geq j_{p+1}$. Therefore, we have $a\neq i_{p+1}$ since $\float$ is column-increasing. 

If $a < i_{p+1}$, then map $((a, b),(i_p, j_p))$ to $((a, j_{p+1}),(i_{p+1},j_{p+1}))$, which is an inversion since $\float_{a, j_{p+1}} \geq \float_{a, b} > \float_{i_{p+1},j_{p+1}}$ by the fact that $\float$ is column-increasing. Otherwise, if $a > i_{p+1}$, then $b\geq j_{p+1} + 1$. Map $((a, b),(i_p, j_p))$ to $((a, j_{p+1}+1),(i_{p+1}, j_{p+1}))$, which is an inversion since $\float_{a, j_{p+1}+1}\geq \float_{a, b} > \float_{i_{p+1},j_{p+1}}$. It is clear that this map is an injection, since each such inversion $((a, b),(i_p, j_p))$ has a unique $a$ value. Hence, $c_p\leq c_{p+1}$.

Second, suppose that $(i_p, j_p)\in \dgprime(\alpha)$ and $(i_{p+1},j_{p+1})\in \basement(\float)$. Let $((a, b),(i_p, j_p))$ be an inversion of $\float$. We have that $a\neq i_{p+1}$ by the fact that $\float$ is column-increasing. If $a > i_{p+1}$, then map $((a, b), (i_p, j_p))$ to the inversion $((a,1), (i_{p+1},j_{p+1}))$ of type (I2). This is indeed an inversion since $\float_{a,1} \geq \float_{a, b} > \float_{i_p, j_p} = \float_{i_{p+1},j_{p+1}}$. Otherwise, if $a < i_{p+1}$, then map $((a, b),(i_p, j_p))$ to the inversion $(a, (i_{p+1}, j_{p+1}))$ of type (I3). By the same reasoning as the first case, this map is an injection, hence $c_p\leq c_{p+1}$.

Finally, suppose that $(i_p, j_p), (i_{p+1},j_{p+1})\in \basement(\float)$. Since $(i_p, j_p)$ precedes $(i_{p+1}, j_{p+1})$ in inversion reading order, we have $i_p\leq i_{p+1}$. Let $((a,1),(i_p, j_p))$ be an inversion of type (I2), so that $\float_{a,1} > \float_{i_p, j_p} = \float_{i_{p+1}, j_{p+1}}$. By the fact that $\float$ is column-increasing, we have that $a\neq i_{p+1}$. If $a>i_{p+1}$, map $((a,1),(i_p, j_p))$ to the inversion $((a,1),(i_{p+1}, j_{p+1}))$ of type (I2). Otherwise, if $a< i_{p+1}$, then map $((a,1),(i_p, j_p))$ to the inversion $(a, (i_{p+1},j_{p+1}))$ of type (I3). If $(a, (i_p, j_p))$ is an inversion of type (I3) with $a < i_p$, then map it to the inversion $(a, (i_{p+1}, j_{p+1}))$ of type (I3), which is an inversion since $a < i_p\leq i_{p+1}$. By the same reasoning as the previous two cases, this map is an injection. Hence, we have $c_p\leq c_{p+1}$ in all cases.
\end{proof}

\begin{lemma}\label{lem:DinvcodeWellDefined}
Sending $\float\in \fci_{n,\alpha, s}(\gamma)$ to its diagonal inversion code gives a map
\begin{align}
\dinvcode^\alpha : \fci_{n, \alpha, s}(\gamma) \to \cC_{n, \sort(\alpha), s}(\rev(\gamma)).
\end{align}
\end{lemma}

\begin{proof}
 It suffices to show that if $\float\in \fci_{n, \alpha, s}(\gamma)$, then $(d_n ,d_{n-1},\dots, d_1)\in \cC_{n,\sort(\alpha), s}(\rev(\gamma))$. The fact that $(d_n,\dots, d_1)\in \cC_{n,\sort(\alpha), s}$ follows by the same reasoning as in the proof of Lemma~\ref{lem:InvcodeWellDefined}, so it suffices to prove that $(d_n, d_{n-1},\dots, d_1)$ is $\rev(\gamma)$-weakly decreasing. Let $(i_1,j_1), (i_2, j_2), \dots,(i_n, j_n)$ be the cells of $\float$, listed so that $\float_{i_1,j_1}\leq \cdots\leq \float_{i_n, j_n}$ and then breaking ties in reading order. It suffices to prove that for $1\leq p\leq n-1$ such that $\float_{i_p, j_p}=\float_{i_{p+1}, j_{p+1}}$, then we have $d_p\leq d_{p+1}$. Given $p$ such that $\float_{i_p, j_p} = \float_{i_{p+1},j_{p+1}}$, define an injection from the set of diagonal inversions counted toward $d_p$ to the set of diagonal inversions counted toward $d_{p+1}$ as follows.

First, suppose that $(i_p, j_p),(i_{p+1}, j_{p+1})\in \dgprime(\alpha)$. Then $d_p = c_p$ and $d_{p+1}=c_{p+1}$, where $(c_n,\dots, c_1) = \invcode^\alpha(\float)$, since all diagonal inversions counting toward $d_p$ and $d_{p+1}$ are among entries in $\sigma(\float)$, hence we have $d_p\leq d_{p+1}$ by Lemma~\ref{lem:InvcodeWellDefined}. 

Second, suppose $(i_p, j_p)\in \dgprime(\alpha)$ and $(i_{p+1},j_{p+1})\in \basement(\float)$. Let $((a, b),(i_p, j_p))$ be a diagonal inversion of $\float$. Just as in the first case of the proof of Lemma~\ref{lem:InvcodeWellDefined}, we have $b \geq j_p\geq j_{p+1}$ and $a\neq i_{p+1}$. If $a <i_{p+1}$, map $((a, b),(i_p, j_p))$ to the diagonal inversion $((a, j_{p+1}),(i_{p+1}, j_{p+1}))$, which is either a diagonal inversion of type (D1) or of type (D2), depending on whether $(a, j_{p+1})$ is in $\diagram(\float)$ or not, respectively. If $a>i_{p+1}$, map $((a, b),(i_p, j_p))$ to $((a, j_{p+1}+1),(i_{p+1}, j_{p+1}))$, which is similarly either a diagonal inversion of type (D1) or (D2) depending on whether $(a, j_{p+1}+1)$ is in $\diagram(\float)$ or not, respectively. The details of the fact that this is a well-defined injection are similar to the second case of the proof of Lemma~\ref{lem:InvcodeWellDefined}, hence we omit them. We conclude that $d_p\leq d_{p+1}$.

Finally, suppose that $(i_p, j_p), (i_{p+1}, j_{p+1})\in \diagram(\float)$. The map on diagonal inversions $((a, b),(i_p, j_p))$ of type (D1) is the same as in the second case above, and we omit the details. Let $((a, b),(i_p, j_p))$ be a diagonal inversion of type (D2), so that $(a, b)\notin \diagram(\float)$. Since $(i_{p+1}, j_{p+1})\in \diagram(\float)$ and $b\geq j_p\geq j_{p+1}$, we have $a\neq i_{p+1}$. Observe that since $(a, b)\notin\diagram(\float)$, then $(a, j_{p+1})\notin \diagram(\float)$. If $a < i_{p+1}$, map $((a, b),(i_p, j_p))$ to the diagonal inversion $((a, j_{p+1}),(i_{p+1}, j_{p+1}))$ of type (D2). Otherwise, if $a>i_{p+1}$, then $(a, j_{p+1}+1)\notin\diagram(\float)$, and we map $((a, b),(i_p, j_p))$ to the diagonal inversion $((a, j_{p+1}+1),(i_{p+1},j_{p+1}))$ of type (D2). Hence, we have $d_p\leq d_{p+1}$ in all cases.
\end{proof}

We define a map inverse to $\invcode^\alpha$ via an insertion algorithm based on~\cite{Rhoades-Wilson-Stirling,Rhoades-Yu-Zhao}. Given $(c_n,\dots, c_1)\in \cC_{n,\sort(\alpha), s}(\rev(\gamma))$, construct an element $\iota^\alpha(c_n,\dots, c_1)\in \fci_{n, \alpha, s}(\gamma)$ by the following procedure. 
At each step in the algorithm, label the columns of the partial labeling of $\dgprime(\alpha)$ with $0,1,2,\dots, s-1$ inductively as follows. Suppose we have already used the column labels $0,1,\dots, j-1$. We say that a column $1\leq i \leq s$ is \emph{unfilled} if there is a cell of $\dgprime(\alpha)$ in column $i$ which is unfilled. If there is an unfilled column of $\dgprime(\alpha)$ which does not have a label, scan through the diagram in reading order until an unfilled cell of the diagram is reached whose column is unlabeled, and label that column with $j$. Otherwise, label the leftmost unlabeled column with $j$.

At the $0$th step in the algorithm, start with the unfilled diagram $\dgprime(\alpha)$, and let $a_1a_2 \cdots a_n$ be the unique word with content $\gamma$ such that $a_1\leq a_2\leq \cdots \leq a_n$. At the $i$th step for $1\leq i\leq n$, assume we have already inserted $a_1,\dots, a_{i-1}$ into the filling. Label the columns of the partially filled diagram according to the procedure above. Let column $j$ be the unique column labeled $c_i$. If column $j$ is unfilled, label the highest unfilled cell in that column with $a_i$. Otherwise, add a basement cell to column $j$ and label it with $a_i$. Let $\iota^\alpha(c_n,\dots, c_1)$ be the filling obtained after step $n$ of the algorithm. Although $\iota^\alpha(c_n,\dots, c_1)$ also depends on $\gamma$, we suppress $\gamma$ from the notation for convenience. 

See Figure~\ref{fig:InvInsertion} for an example of the insertion algorithm when $n=9$, $\alpha = (3,2,0)$, $s=3$, $(c_n,\dots, c_1) = (0,0,2,1,1,0,0,1,0)$, and $\gamma = (2,2,1,2,1,1)$. Observe that the final output is a filling with inversion code $(c_n,\dots, c_1)$.

We also define a map inverse to $\dinvcode^\alpha$ via an insertion algorithm. The insertion algorithm is the same, except that at each step in the algorithm, we label the columns by a different procedure. 

Let $(d_n,\dots, d_1)\in \cC_{n, \sort(\alpha), s}(\rev(\gamma))$.
At step $1\leq i\leq n$ in the algorithm, suppose we have a partial filling of $\dgprime(\alpha)$. Label the columns of the diagram with $0,1,2,\dots, s-1$ inductively as follows. Suppose we have already used the labels $0,1,\dots, j-1$. Scan through the coordinates $(a, b)$ with $1\leq a\leq s$ and $b\leq \alpha_a$ left to right across each row, starting with the top row. When an unfilled coordinate $(a, b)$ is reached in a column which is unlabeled, label that column with $j$. 

After all columns are labeled, let column $j$ be the unique column labeled $d_i$. If column $j$ is unfilled, label the highest unfilled cell in column $j$ with $a_i$. Otherwise, add a basement cell to column $j$ and label it with $a_i$. Let $\iota^\alpha_d(d_n,\dots, d_1)$ be the filling obtained after step $n$ of the algorithm.

See Figure~\ref{fig:DinvInsertion} for an example of the insertion algorithm for $\iota_d^\alpha$ with $n=9$, $\alpha = (3,2,0)$, $s=3$, $(d_n,\dots, d_1) = (0,0,2,1,1,0,0,1,0)$, and $\gamma = (2,2,1,2,1,1)$. Observe that the final output is a filling with diagonal inversion code $(d_n,\dots, d_1)$. Further observe that the two insertion algorithms for $\iota^\alpha$ and $\iota^\alpha_d$ output different fillings for the same sequence.

\begin{figure}
\centering
\includegraphics[scale=0.7]{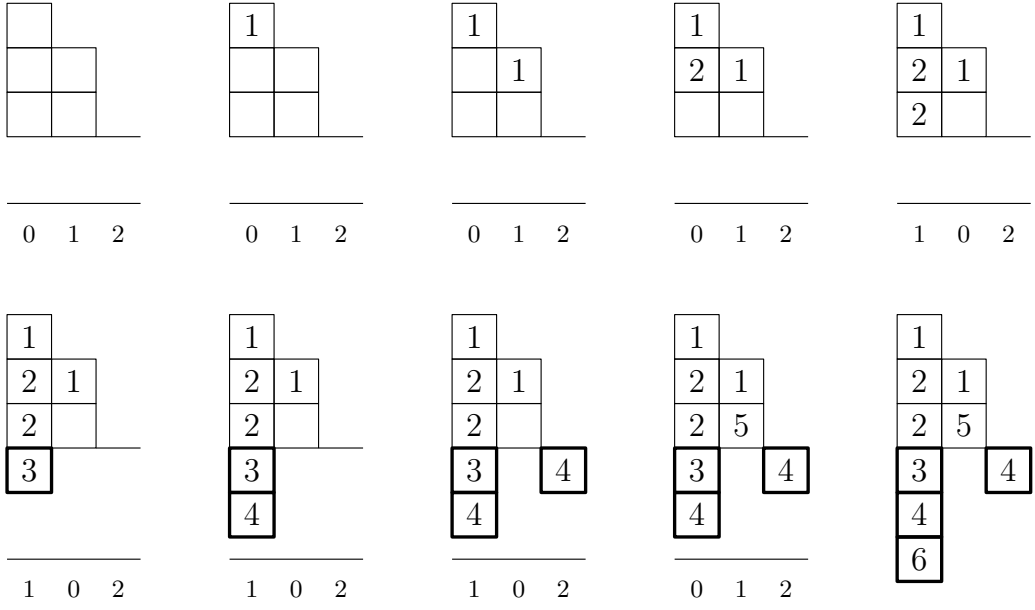}
\caption{An example of the insertion algorithm for $\iota^\alpha(c_n,\dots, c_1)$ for $n=9$, $\alpha = (3,2,0)$, $s=3$, $(c_n,\dots, c_1) = (0,0,2,1,1,0,0,1,0)$, and $\gamma=(2,2,1,2,1,1)$.\label{fig:InvInsertion}}
\end{figure}

\begin{lemma}\label{lem:IotaWellDefined}
Sending $(c_n,\dots, c_1)$ to $\iota^\alpha(c_n,\dots, c_1)$ gives a map
\begin{align}
\iota^\alpha : \cC_{n,\sort(\alpha), s}(\rev(\gamma))\to \fci_{n,\alpha, s}(\gamma).
\end{align}
Sending $(d_n,\dots, d_1)$ to $\iota^\alpha_d(d_n,\dots, d_1)$ gives a map
\begin{align}
\iota^\alpha_d : \cC_{n,\sort(\alpha), s}(\rev(\gamma))\to \fci_{n,\alpha, s}(\gamma).
\end{align}
\end{lemma}

\begin{proof}
Given $(c_n,\dots, c_1)\in \cC_{n,\sort(\alpha), s}(\rev(\gamma))$, it is immediate that the reading word of $\iota^\alpha(c_n,\dots, c_1)$ has content $\gamma$ and that the resulting filling is column-increasing, since $a_1\leq \dots\leq a_{n}$ are inserted in increasing order. Therefore, in order to prove that $\iota^\alpha(c_n,\dots, c_1)\in \fci_{n,\alpha, s}(\gamma)$, it suffices to prove that all cells of $\dgprime(\alpha)$ are filled at the end of the algorithm. We proceed by induction on $n$. In the base case $n=1$, if $\dgprime(\alpha)$ is empty then the conclusion is immediate. Otherwise, $\dgprime(\alpha)$ consists of one cell, whose column is labeled $0$ at the beginning of the algorithm. The only $(n,\sort(\alpha), s)$-staircase is $(0)$, hence the unique cell of $\dgprime(\alpha)$ is filled after step $1$.

In the inductive step, let $n>1$ and assume that $\iota^\alpha$ gives a map $\iota^\alpha : \cC_{m,\sort(\alpha), s}(\rev(\gamma))\to \fci_{m, \alpha, s}(\gamma)$ for all $m<n$ and all $\alpha$, $s$, and $\gamma$. Fix $\alpha$, $s$, and $\gamma$, and let $\la = \sort(\alpha)$. We claim that $\iota^\alpha$ gives a map $\iota^\alpha : \cC_{n,\la, s}(\rev(\gamma))\to \fci_{n,\alpha, s}(\gamma)$. Let $(c_n,\dots, c_1)\in \cC_{n,\la, s}(\rev(\gamma))$, and initialize the unfilled diagram $\dgprime(\alpha)$. 

At the first step of the algorithm, the label $a_1=1$ is inserted into the column labeled $c_1$. Suppose this column is the $j$th column from the left.
If $\alpha_{j} > 0$, then the remaining unfilled cells of $\dgprime(\alpha)$ form the conjugate diagram of the composition $\alpha^{(c_1)} \coloneqq (\alpha_1,\dots, \alpha_{j}-1,\dots, \alpha_s)$. By the way we have labeled the columns of $\dgprime(\alpha)$, we have $\sort(\alpha^{(c_1)}) = \la^{(c_1)}$. 

If $\gamma_1>1$, then by Lemma~\ref{lem:CompDecomp}, we have $(c_n,\dots, c_2)\in \cC^{(c_1)}_{n-1,\la^{(c_1)}, s}(\gamma_m,\dots, \gamma_2,\gamma_1-1)$. 
By our inductive hypothesis, we have the map
\begin{align}
\iota^{\alpha^{(c_1)}} : \cC_{n-1,\la^{(c_1)},s}(\gamma_m,\dots, \gamma_2, \gamma_1-1) \to \fci_{n-1, \alpha^{(c_1)}, s}(\gamma_1-1,\gamma_2,\dots, \gamma_m),
\end{align}
so there are no unfilled cells of $\dgprime(\alpha^{(c_1)})$ in $\iota^{\alpha^{(c_1)}}(c_n,\dots, c_2)$. Observe that by the construction of the insertion algorithm, $\iota^{\alpha^{(c_1)}}(c_n,\dots, c_2)$ is obtained from $\iota^\alpha(c_n,\dots, c_1)$ by deleting the cell labeled $a_1$. Hence, there are no unfilled cells of $\dgprime(\alpha)$ in $\iota^\alpha(c_n,\dots, c_1)$, so $\iota^\alpha(c_n,\dots, c_1)\in \fci_{n,\alpha, s}(\gamma)$.
The case when $\alpha_i > 0$ and $\gamma_1=1$, and the cases when $\alpha_i = 0$ and $\gamma_1=1$ or $\gamma_1>1$, follow by similar applications of the inductive hypothesis and Lemma~\ref{lem:CompDecomp}. Hence, $\iota^\alpha$ gives a map to $\fci_{n,\alpha, s}(\gamma)$ in all cases, and the induction is complete.

The fact that $\iota^\alpha_d$ is well-defined follows from the fact that $\iota^\alpha$ is well-defined. Indeed, it is immediate from the construction of $\iota_d^\alpha$ that the reading word of $\iota_d^\alpha(c_n,\dots, c_1)$ has content $\gamma$ and that the columns are weakly increasing. Furthermore, since unfilled columns of $\dgprime(\alpha)$ are labeled in the same way in both insertion algorithms, then the fillings $\iota^\alpha(c_n,\dots, c_1)$ and $\iota^\alpha_d(c_n,\dots, c_1)$ are the same when restricted to $\dgprime(\alpha)$. Hence, $\iota^\alpha_d(c_n,\dots, c_1)$ is also an element of $\fci_{n,\alpha, s}(\gamma)$, so $\iota_d^\alpha$ is a map to $\fci_{n,\alpha, s}(\gamma)$.
\end{proof}

\begin{lemma}\label{lem:InvcodeBij}
The map $\invcode^\alpha$ is a bijection with inverse $\iota^\alpha$, and the map $\dinvcode^\alpha$ is a bijection with inverse $\iota_d^\alpha$.
\end{lemma}

\begin{proof}
First, we prove that $\invcode^\alpha(\iota^\alpha(c_n,\dots, c_1)) = (c_n,\dots, c_1)$ for all $(c_n,\dots, c_1)\in \cC_{n,\sort(\alpha), s}(\rev(\gamma))$. At step $i$ in the algorithm constructing $\iota^\alpha(c_n,\dots, c_1)$, let $(a, b)$ be the coordinates at which the label $a_i$ is inserted.

First, suppose the column labeled $c_i$ is unfilled, so that $b\geq 1$. Since $(c_n,\dots, c_1)$ is $\rev(\gamma)$-weakly decreasing, then all other labels $a_j$ with $j > i$ and $a_j=a_i$ will be inserted in a position which is after $(a, b)$ in inversion reading order. Therefore, for every $q < a$ such that column $q$ has label less than $c_i$ then the cell $(q, b)$ is unfilled at step $i$. Therefore, the cell $(q, b)$ is filled with a label strictly larger than $a_i$ by the end of the algorithm. Furthermore, for every $q > a$ such that the column $q$ has label less than $c_i$, the cell $(q, b+1)$ is unfilled at step $i$, hence $(q, b+1)$ must be filled with a label strictly larger than $a_i$ by the end of the algorithm. Since there are $c_i$ many labels less than $c_i$, then the number of inversions of type (I1) of the form $((i,j),(a, b))$ in $\iota^\alpha(c_n,\dots, c_1)$ is $c_i$. Hence, the $(n-i+1)$th entry of $\invcode^\alpha(\iota^\alpha(c_n,\dots, c_1))$ is $c_i$, as desired.

Second, suppose column $a$ is filled, so that $a_i$ is inserted into a basement cell in column $a$. By construction of the algorithm, for each column $q>a$ which is unfilled at step $i$, the entry $(q,1)$ will be filled with a number strictly greater than $a_i$. Furthermore, each column $q < a$ has a label which is smaller than $c_i$. Hence, the number of inversions in $\iota^\alpha(c_n,\dots, c_1)$ of type (I2) of the form $((q,1),(a, b))$ plus the number of type (I3) of the form $(q,(a, b))$ is $c_i$, so the $(n-i+1)$th entry of $\invcode^\alpha(\iota^\alpha(c_n,\dots, c_1))$ is $c_i$. We conclude that $\invcode^\alpha(\iota^\alpha(c_n,\dots, c_1)) = (c_n,\dots, c_1)$, so $\invcode^\alpha\circ\iota^\alpha = \mathrm{Id}$. 

We claim that $\iota^\alpha\circ \invcode^\alpha = \mathrm{Id}$. Let $\float\in \fci_{n,\alpha, s}(\gamma)$ with $(c_n,\dots, c_1) = \invcode^\alpha(\float)$. Suppose column $a$ is the unique column labeled $c_1$ at step $1$ of the insertion algorithm for $\iota^\alpha$. By the definition of $c_1$, the first $1$ in $\float$ in inversion reading order is in the top-most cell of column $a$. Therefore, the location of the first $1$ in $\float$ is the same as the location of the label $a_1$ in $\iota^\alpha(c_n,\dots, c_1)$. Since $\iota^{\alpha^{(c_1)}}(c_n,\dots, c_2)$ is obtained from $\iota^\alpha(c_n,\dots, c_1)$ by deleting the first cell labeled $1$ (and shifting labels in the case that $\gamma_1=1$), then a straightforward induction on $n$ shows that $\iota^\alpha(c_n,\dots, c_1) = \float$, so the claim follows. 

We conclude that $\iota^\alpha$ and $\invcode^\alpha$ are mutually inverse, and hence they are bijections. The fact that $\iota^\alpha_d$ and $\dinvcode^\alpha$ are bijections which are mutually inverse follows by a similar argument using reading order in place of inversion reading order.
\end{proof}

\begin{figure}
\centering
\includegraphics[scale=0.7]{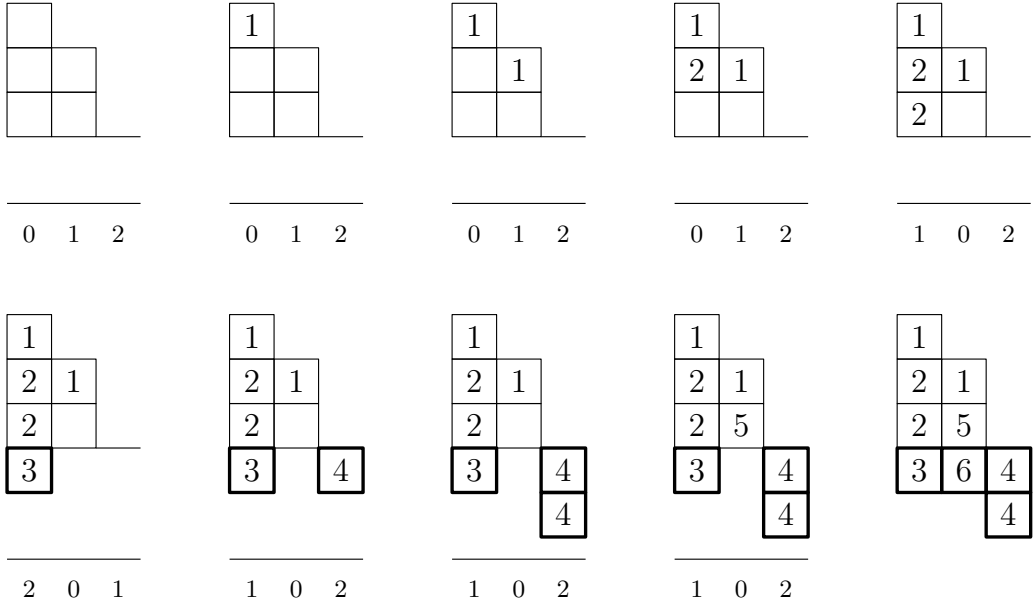}
\caption{An example of the insertion algorithm for $\iota^\alpha_d(c_n,\dots, c_1)$ when $n=9$, $\alpha = (3,2,0)$, $s=3$, and $(d_n,\dots, d_1) = (0,0,2,1,1,0,0,1,0)$.\label{fig:DinvInsertion}}
\end{figure}

\begin{proof}[Proof of Theorem~\ref{thm:Equidistrib}]
We have the bijection
\begin{align}
\iota^\la \circ \dinvcode^\la : \fci_{n, \la, s}(\gamma) \to \fci_{n,\la, s}(\gamma),
\end{align}
which maps $\float\in \fci_{n,\la, s}(\gamma)$ to $\float'\in \fci_{n,\la, s}(\gamma)$ with $\dinvcode^\la(\float) = \invcode^\la(\float')$. Therefore, we have $\fdinv(\float) = \finv(\float')$. Hence, we have
\begin{align}\label{eq:CoeffOfM}
\sum_{\float\in \fci_{n,\la, s}(\gamma)} q^{\finv(\float)} = \sum_{\float\in \fci_{n,\la, s}(\gamma)} q^{\fdinv(\float)},
\end{align}
for all $\gamma = (\gamma_1,\dots, \gamma_m)$ a composition of $n$ into positive parts. Let $\mathrm{set}(\gamma) \coloneqq \{\gamma_1,\gamma_1+\gamma_2,\dots, \gamma_1+\cdots +\gamma_{m-1}\}$. Since both sides of~\eqref{eq:Equidistrib} are quasisymmetric, we have the following expansions into monomial quasisymmetric functions,
\begin{align}
\sum_{\float\in \fci_{n,\la, s}} q^{\finv(\float)}\bx^\float &= \sum_{\gamma} \left(\sum_{\float\in \fci_{n,\la, s}(\gamma)} q^{\finv(\float)}\right) M_{n, \mathrm{set}(\gamma)}(\bx),\label{eq:M1}\\
 \sum_{\float\in \fci_{n,\la, s}} q^{\fdinv(\float)}\bx^\float &=\sum_{\gamma} \left(\sum_{\float\in \fci_{n, \la, s}(\gamma)} q^{\fdinv(\float)}\right) M_{n, \mathrm{set}(\gamma)}(\bx),\label{eq:M2}
\end{align}
where the sums are over compositions $\gamma$ of $n$ such that $\gamma_i > 0$ for all $i$. By \eqref{eq:CoeffOfM}, the right-hand sides of \eqref{eq:M1} and \eqref{eq:M2} are equal. This completes the proof.
\end{proof}

\begin{definition}\label{def:DDef}
Let $D_{n,\la, s}$ be the multivariate generating function in Theorem~\ref{thm:Equidistrib},
\begin{align}
D_{n,\la, s} \coloneqq \sum_{\float\in \fci_{n,\la, s}} q^{\finv(\float)}\bx^\float = \sum_{\float\in \fci_{n,\la, s}} q^{\fdinv(\float)} \bx^\float.\label{eq:DFormula}
\end{align}
\end{definition}

\begin{corollary}\label{cor:FBasisExpansion}
We have
\begin{align}
D_{n, \la, s} = \sum_{\float \in \FCI_{n, \la, s}} q^{\finv(\float)}F_{n,\mathrm{iDes(\rw(\float))}}(\bx) = \sum_{\float\in \FCI_{n,\la, s}} q^{\dinv(\float)} F_{n,\iDes(\rw(\float))}(\bx),
\end{align}
where for any permutation $\pi\in S_n$,  $\iDes(\pi) \coloneqq \Des(\pi^{-1})$.
\end{corollary}

\begin{proof}
Recall from~\eqref{eq:GesselAlt} that
\begin{align}
F_{n,D}(\bx) = \sum_{\substack{w\in \mathbb{N}^n,\\ \standard(w) = \pi}} \prod_{i\geq 1} x_i^{\# i\text{'s in }w},
\end{align}
where $\pi\in S_n$ is a fixed permutation such that $\iDes(\pi) = D$.
Since the statistic $\finv$ is only dependent on $\la$, the positions of the labels in $\float$, and $\rw(\float)$, then the result follows immediately by Theorem~\ref{thm:Equidistrib} and Definition~\ref{def:DDef}.
\end{proof}

\begin{theorem}\label{thm:LLTThm}
The generating function $D_{n,\la, s}$ is a symmetric function which expands as a positive sum of LLT symmetric functions, each shifted by some power of $q$.
\end{theorem}

\begin{proof}
The right-most side of \eqref{eq:DFormula} is the sum over all $\beta\in \Comp(n-k, s)$ of
\begin{align}\label{eq:RestrictedSum}
\sum_{\float} q^{\fdinv(\float)} \bx^\float,
\end{align}
where the sum ranges over column-increasing fillings $\float\in \fci_{n,\la, s}$ which have $\beta_i$ many basement cells in column $i$ for $1\leq i\leq s$. Fix $\beta$, and let $m$ be the number of diagonal inversions of type (D2) in a filling $\float\in \fci_{n,\la, s}$ which has $\beta_i$ many basement cells in column $i$ for $1\leq i\leq s$. Observe that $m$ is only dependent on $\beta$ and $s$. 

Recall the LLT polynomial $G_{\boldsymbol{\nu}}(\bx;q)$ defined in Subsection~\ref{subsec:HLFunctions}. For $1\leq i\leq \ell(\la)$, let $\nu^i$ be the single row of size $\beta_i+\la_i$, shifted so that the cells have contents 
\begin{align}
\la_i, \la_i-1,\dots, 1,0, -1,\dots, -\beta_i+1
\end{align}
 from left to right. For $\ell(\la) < i \leq s$, let $\nu^i$ be the single row of size $\beta_i$, shifted so that the cells have contents $0,-1,\dots, -\beta_i+1$. Let $\boldsymbol{\nu} = (\nu^1,\dots,\nu^s)$. 

We claim that
\begin{align}\label{eq:ShiftedLLT}
\sum_{\float} q^{\fdinv(\float)} \bx^\float = q^m G_{\boldsymbol{\nu}}(\bx;q),
\end{align} 
where the sum is over column-increasing fillings $\float\in \fci_{n,\la, s}$ which have $\beta_i$ many basement cells in column $i$ for $1\leq i\leq s$.
Indeed, to each $T = (T^{1}, \dots, T^{s})\in \SSYT(\boldsymbol{\nu})$ associate the column-increasing filling $\float\in \fci_{n,\la, s}$ whose $i$th column is filled with the same multiset of labels as $T^{i}$. It can then be checked that inversions in $T$ correspond to diagonal inversions of type (D1) in $\float$, which proves our claim. Since $D_{n,\la, s}$ is a sum over shifted LLT polynomials of the form~\eqref{eq:ShiftedLLT}, and each LLT polynomial is symmetric in $\bx$ by Theorem~\ref{thm:LLTSymmetry}, then $D_{n,\la, s}$ is symmetric in $\bx$.
\end{proof}

	\subsection{Monomial and quasisymmetric function formulas for $\Frobq(R_{n,\lambda, s})$}\label{subsec:ProofOfCoinvQSym}
		
In this subsection, we prove an expansion of $\Frobq(R_{n,\la, s})$ into Gessel's fundamental quasisymmetric functions in terms of the coinversion statistic. Our main tool to prove this result is the skewing operator $e_j(\bx)^\perp$. We obtain both the monomial expansion of $\Frobq(R_{n,\la, s})$ stated in the introduction as \eqref{eq:MonomialExpansion} and the formula for the Hilbert series of $R_{n,\la, s}$ stated as \eqref{eq:HilbEquation2}.

We first state the main theorem in this section, Theorem~\ref{thm:MonomialTheorem}, and an immediate corollary. The remainder of this subsection is dedicated to proving Theorem~\ref{thm:MonomialTheorem}.

\begin{theorem} \label{thm:MonomialTheorem}
We have
\begin{align}\label{eq:MonomialFrobFormula}
\Frobq(R_{n,\la, s}) = D_{n, \la, s} = \sum_{\float\in \fci_{n,\la, s}} q^{\finv(\float)} \bx^{\float}  = \sum_{\float\in\fci_{n,\la, s}} q^{\fdinv(\float)} \bx^{\float},
\end{align}
\end{theorem}

\begin{corollary}\label{cor:FakeInvFormulaHilb}
We have
\begin{align}
\Hilbq(R_{n,\la, s}) = \sum_{\float \in \FCI_{n,\la, s}} q^{\finv(\float)} = \sum_{\float\in \FCI_{n,\la, s}} q^{\fdinv(\float)}.
\end{align}
\end{corollary}
\begin{proof}
Apply \eqref{eq:HilbFromFrobq} to $V=R_{n,\la, s}$ and the monomial expansions in~\eqref{eq:MonomialFrobFormula}.
\end{proof}

\begin{lemma}\label{lem:SortingLemma}
Given $\alpha\in \Comp(k, s)$, we have 
\begin{align}\label{eq:SortingLemma}
\sum_{\float\in \fci_{n,\alpha, s}} q^{\inv(\float)} \bx^\float = D_{n, \sort(\alpha), s}.
\end{align}
\end{lemma}

\begin{proof}
Let $\la = \sort(\alpha)$ for convenience. By Lemma~\ref{lem:InvcodeBij}, we have the bijection
\begin{align}
\iota^\la \circ \invcode^\alpha : \fci_{n, \alpha, s}(\gamma) \to \fci_{n, \la, s}(\gamma)
\end{align}
which maps $\float \in \fci_{n, \alpha, s}(\gamma)$ to $\float'\in \fci_{n, \la, s}(\gamma)$ with $\invcode^\alpha(\float) = \invcode^\la(\float')$. Hence, we have $\inv(\float) = \inv(\float')$. Therefore,
\begin{align}
\sum_{\float\in \fci_{n,\alpha, s}(\gamma)} q^{\inv(\float)} = \sum_{\float\in \fci_{n, \la, s}(\gamma)} q^{\inv(\float)}.
\end{align}
Since both sides of \eqref{eq:SortingLemma} are quasisymmetric, the proof follows by the same reasoning as in the proof of Theorem~\ref{thm:Equidistrib}.
\end{proof}

Observe that $D_{n,\la, s}$ and $\Frobq(R_{n,\la, s})$ are homogeneous in $\bx$ of positive degree, so their constant terms are both equal to $0$. Further observe that $\Frobq(R_{1,\la, s}) = D_{1,\la, s}$ for all $\la$ and $s$.
Recall that $D_{n,\la, s}$ is a symmetric function by Theorem~\ref{thm:LLTThm}. Furthermore, by Lemma~\ref{lem:AntisymmIdentification} any symmetric function with zero constant term is uniquely determined by its images under the $e_j^\perp$ skewing operators. By Theorem~\ref{thm:AntisymmFrobqRecursion} and induction on $n$, in order to prove Theorem~\ref{thm:MonomialTheorem} it suffices to show that $D_{n,\la, s}$ satisfies the same identity under the skewing operators as $\Frobq(R_{n,\la, s})$, namely that
\begin{align}\label{eq:D-perp1}
e_j^\perp D_{n,\la, s} = \sum_{I \in \cI_\ell^j} q^{\Sigma(I)}\, D_{n-j,\la^{(I)}, s}
\end{align}
for all $j\geq 1$.

Fix a composition $\beta = (\beta_1,\dots, \beta_p)$ of $n$ into positive parts. A \emph{$\beta$-shuffle} is a shuffle of the decreasing sequences
\[
(\beta_1,\dots, 2,1), \, (\beta_1 + \beta_2,\dots, \beta_1 + 1),\dots, (n,n-1,\dots, n-\beta_p + 1),
\]
of lengths $\beta_1, \beta_2,\dots,\beta_p$, respectively. 
 Let $\FCI_{n,\la, s}^\beta$ be the set of $\float\in \FCI_{n,\la, s}$ such that $\rw(\float)$ is a $\beta$-shuffle. 
 For convenience of notation, let $j\coloneqq \beta_1$ so that  $\beta = (j,\beta_2,\dots, \beta_p)$. Observe that if $\float\in \FCI_{n,\la, s}^\beta$, then the cells labeled with $1,2, \dots, j$ are in distinct columns, and hence each of the labels $1, 2, \dots, j$ is either in the top-most cell in its column or a basement label. 

\begin{proof}[Proof of Theorem~\ref{thm:MonomialTheorem}]
Following the strategy in \cite[pp. 878]{HRS1}, $D_{n,\la, s}$ is the unique symmetric function such that for any composition $\beta = (j, \beta_2, \beta_3,\dots, \beta_p)$ of $n$ into positive parts, 
\begin{align}\label{eq:inner-prod1}
\langle D_{n,\la, s}, \, e_j(\bx)e_{\beta_2}(\mathbf{x})e_{\beta_3}(\mathbf{x})\cdots e_{\beta_p}(\mathbf{x})\rangle = \sum_{\float\in\FCI^\beta_{n, \la, s}} q^{\finv(\float)}.
\end{align}
By the definition of $e_j^\perp$, \eqref{eq:inner-prod1} is equivalent to
\begin{align}
\langle e_j^\perp D_{n,\la, s},\, e_{\beta_2}(\mathbf{x})\cdots e_{\beta_{p}}(\mathbf{x})\rangle =  \sum_{\float \in \FCI^\beta_{n, \la, s}} q^{\finv(\float)}.
\end{align}

Since elementary symmetric functions form a basis of symmetric functions, then \eqref{eq:D-perp1} is equivalent to the identity
\begin{align}
\langle e_j^\perp D_{n,\la, s}, \, e_{\beta_2}(\mathbf{x}_n)\cdots e_{\beta_{p}}(\mathbf{x}_n)\rangle = \sum_{I\in \cI_s^j} q^{\Sigma(I)} \,\langle D_{n-j,\la^{(I)}, s}, \, e_{\beta_2}(\mathbf{x}_n)\cdots e_{\beta_{p}}(\mathbf{x}_n)\rangle,
\end{align}
which can be rewritten as the identity
\begin{align}\label{eq:THEequation1}
\sum_{\float \in \FCI^\beta_{n,\la, s}} q^{\finv(\float)} 
= \sum_{I \in\cI_s^j}  \sum_{\float \in \FCI^{\beta'}_{n-j,\la^{(I)}, s}} q^{\Sigma(I) + \finv(\float)},
\end{align}
where $\beta' \coloneqq (\beta_2,\dots,\beta_{p})$.

The identity~\eqref{eq:THEequation1} has a simple bijective proof, as follows. Let $\float\in \FCI^\beta_{n,\la, s}$, and let $i_1+1, \dots, i_j+1$ be the columns of $\float$ containing $j, j-1,\dots, 1$, respectively, and let $I = (i_1,\dots, i_j)\in \cI_s^j$. Let $\alpha^I$ be the composition defined by
\begin{align}
\alpha^I_i = \begin{cases} \la_i -1  & \text{if } i+1 \in I\\ \la_i & \text{if }i+1\notin I\end{cases}.
\end{align}
Let $\float' \in \FCI_{n-j, \alpha^I, s}^{\beta'}$ be the extended column-increasing filling obtained from $\float$ by removing cells labeled $1,\dots, j$ and then standardizing the remaining labels to the set $[n-j]$. The map sending $\float$ to $(I, \float')$ is a bijection between $\FCI_{n, \la, s}^\beta$ and $\{(I, \float') \st I\in \cI_s^j,\, \float'\in \FCI_{n-j,\alpha^I, s}^{\beta'}\}$.

Recall that $1,\dots, j$ appear in $\rw(\float)$ in decreasing order, so each label $h\leq j$ is in column $i_{j-h+1}+1$. Each $h\leq j$ which is a label in $\sigma(\float)$ forms a diagonal inversion with each of the other cells in its row to its left, hence it contributes $i_{j-h+1}$ to $\inv(\float)$. For each label $h\leq j$ which is the label of a basement cell $(a, b)$ of $\float$, $h$ does not form any inversions of type (I2). Therefore, the contribution of $h$ to $\finv(\float)$ is the number of inversions of type (I3) of the form $(i,(a, b))$, which is $i_{j-h+1}$. Hence, we have $\inv(\float) = \Sigma(I) + \inv(\float')$.  

Using the bijection above, we have
\begin{align}\label{eq:AlmostToMainTheorem}
\sum_{\float\in \FCI_{n,\la, s}^\beta} q^{\finv(\float)} = \sum_{I\in \cI_s^j}\sum_{\float\in \FCI_{n-j, \alpha^I, s}^{\beta'}} q^{\Sigma(I) + \finv(\float)}.
\end{align}
For each $I\in \cI_s^j$, we have $\sort(\alpha^I) = \la^{(I)}$ by the definitions of $\alpha^I$ and $\la^{(I)}$. By Lemma~\ref{lem:SortingLemma}, 
\begin{align}\label{eq:AlllmooostThere}
\sum_{\float\in \FCI_{n-j,\alpha^I, s}} q^{\finv(\float)} \bx^\float = D_{n-j,\la^{(I)}, s}.
\end{align}
Taking the inner product of both sides of \eqref{eq:AlllmooostThere} with $e_{\beta_2}(\bx)\cdots e_{\beta_p}(\bx)$, we have
\begin{align}\label{eq:SwapBetaAlphaPrime}
\sum_{\float\in \FCI_{n-j,\alpha^I, s}^{\beta'}} q^{\finv(\float)} = \sum_{\float\in \FCI_{n-j, \la^{(I)}, s}^{\beta'}} q^{\finv(\float)}.
\end{align}
Hence, \eqref{eq:THEequation1} follows by combining \eqref{eq:AlmostToMainTheorem} with~\eqref{eq:SwapBetaAlphaPrime}, which completes the proof.
\end{proof}

\section{Applications to rank varieties}\label{sec:Geometry}

In this section, we apply our results on the rings $R_{n,\la, s}$ to the geometry of rank varieties. In particular, we show that the ring $R_{n,\la}$ defined in the Introduction is the coordinate ring of the scheme-theoretic intersection of a rank variety with diagonal matrices. We then find a monomial basis for $R_{n,\la}$. Furthermore, we compute the Hilbert series and Frobenius characteristic of $R_{n,\la}$ in terms of the inversion statistic on extended column-increasing fillings. We will show that each of these formulas for $R_{n,\la}$ is a ``limit'' as $s\to \infty$ of the corresponding formula for $R_{n,\la, s}$.

Let $\fgl_n$ be the space of $n\times n$ matrices over $\bQ$. Let $x_{i,j}$  for $1\leq i,j \leq n$ be the coordinate functions corresponding to the entries of an $n\times n$ matrix. Then the coordinate ring of $\fgl_n$ is $\bQ[\fgl_n] = \bQ[x_{i,j}]$.  

For $\la\vdash n$, let $\cO_\la\subseteq \fgl_n$ be the conjugacy class of nilpotent $n\times n$ matrices over $\bQ$ whose Jordan canonical form has block sizes recorded by $\la$. Let $\overline{\cO}_\la$ be the closure of $\cO_\la$ in $\fgl_n$ in the Zariski topology. The set of diagonal matrices $\ft$ is the variety defined by the ideal
\begin{align}
I(\ft) = \langle x_{i,j} \st i\neq j\rangle.
\end{align}
The \emph{scheme-theoretic intersection} of the varieties $\overline{\cO}_\la$ and $\ft$ is the affine scheme whose coordinate ring is defined by the sum of the defining ideals of $\overline{\cO}_\la$ and $\ft$,
\begin{align}\label{eq:CoordinateRing}
\bQ[\overline{\cO}_\la \cap \ft] \coloneqq \frac{\bQ[x_{i,j}]}{I(\overline{\cO}_\la) + I(\ft)}.
\end{align}
The symmetric group $S_n$ of permutation matrices acts by conjugation on $\ft$, which descends to an action of $S_n$ on $\bQ[\overline{\cO}_\la\cap \ft]$. Observe that the variables $x_{i,i}$ generate this coordinate ring. Reindexing the generators $x_{i,i}$ of $\bQ[\overline{\cO}_\la \cap \ft]$ by $x_i$, then $S_n$ acts by permuting the $x_i$ variables.

Motivated by work of Kostant~\cite{Kostant} on the coinvariant algebra, Kraft~\cite{Kraft} conjectured that the coordinate ring~\eqref{eq:CoordinateRing} is isomorphic to the cohomology ring of a Springer fiber. De Concini and Procesi~\cite{dCP} proved Kraft's conjecture. Tanisaki~\cite{Tanisaki} then simplified the arguments of De Concini and Procesi and further proved that these rings have the explicit presentation as the quotient ring $R_{\la} = R_{n, \la, \ell(\la)}$, as defined in Section~\ref{sec:Intro}.

Let $\Fl(n)$ be the \emph{complete flag variety} of flags $V_\bullet = (V_1\subseteq V_2\subseteq \cdots \subseteq V_n)$, where $V_i$ is an $i$-dimensional complex vector subspace of $\bC^n$ for each $i$.
Given a matrix $X\in \cO_\la$, the \emph{Springer fiber} of $X$ is
\begin{align}\label{eq:SpringerFiber}
\cF_X \coloneqq \{V_\bullet\in \Fl(n)\st XV_i\subseteq V_i, 1\leq i \leq n\}.
\end{align}
The Springer fiber gets its name from the fact that it is the fiber over $X$ of the Springer resolution of the nilpotent cone, see for example \cite{CG}. If $X, X'\in \cO_\la$, then we have an isomorphism of varieties $\cF_X \cong \cF_{X'}$. We denote by $\cF_\la$ the Springer fiber of any $X\in \cO_\la$. 

In~\cite{Springer-WeylGrpReps}, Springer proved that there is an action of $S_n$ on the cohomology ring $H^*(\cF_\la;\bQ)$, even though $S_n$ does not act directly on the space $\cF_\la$. Furthermore, Springer proved that these graded representations have the remarkable property that the top nonvanishing cohomology group $H^{2n(\la)}(\cF_\la;\bQ)$ is isomorphic to an irreducible representation of $S_n$, and that all irreducible representations of $S_n$ appear in this manner.
 It is well known that $\cF_\la$ has no nontrivial odd cohomology groups~\cite{Hotta-Springer}, so we consider $H^*(\cF_\la;\bQ)$ as a graded ring by declaring that the $i$th graded piece is $H^{2i}(\cF_\la;\bQ)$. Therefore, $H^*(\cF_\la;\bQ)$ has the structure of a graded $S_n$-module. Since Springer's construction of this symmetric group module action, much work has been done to understand the combinatorics and geometry of Springer fibers and their generalizations, the Hessenberg varieties~\cite{Fresse-Melnikov,Fung,Spaltenstein,Tymoczko,Tymoczko-LinearConditions,Vargas}.

\begin{theorem}[\cite{dCP,Tanisaki}]
We have isomorphisms of graded rings and $S_n$-modules
\begin{align}
R_\la \cong \bQ[\overline{\cO}_{\la'}\cap \ft] \cong H^*(\cF_\la;\bQ),
\end{align}
where $S_n$ acts on $H^*(\cF_\la;\bQ)$ via Springer's representation tensored with the sign representation.
\end{theorem}

In~\cite{Eisenbud-Saltman}, Eisenbud and Saltman study varieties generalizing the varieties $\overline{\cO}_{\la}$. These varieties are the main focus of this section.

\begin{definition}
Let $k\leq n$, and let $\la\vdash k$. The \emph{Eisenbud-Saltman rank variety} is the variety
\begin{align}
\overline{\cO}_{n,\la} &\coloneqq \{ X\in \fgl_n\st \dim\ker X^d \geq \la_1' + \cdots +\la_d', d=1,2,\dots, n\}\\
&= \{X\in \fgl_n \st \rk(X^d)\leq (n-k)+p_{n-d}^n(\la), d=1,2,\dots, n\}.
\end{align}
\end{definition}
The variety $\overline{\cO}_{n,\la}$ is the same as $X_r$ defined in \cite{Eisenbud-Saltman}, where $r$ is the rank function $r(d) = (n-k)+p_{n-d}^n(\la)$. Eisenbud and Saltman proved that rank varieties are Gorenstein and normal with rational singularities~\cite[Theorem 1]{Eisenbud-Saltman}. When $n=k$, we have $\overline{\cO}_{n,\la} = \overline{\cO}_\la$. When $n>k$, then $\overline{\cO}_{n,\la}$ contains matrices which are not nilpotent.  In particular, the variety $\overline{\cO}_{n,\la}$ contains all block diagonal matrices of the form
\begin{align}\label{eq:JordanForm}
X_\la \oplus A_{n-k} = \left[\begin{array}{c|c} X_{\la} & 0 \\ \hline 0 & A_{n-k}\end{array} \right]
\end{align}
where $X_\la\in \overline{\cO}_\la\subseteq \fgl_k$ and $A_{n-k}\in \fgl_{n-k}$, as well as any matrix which is conjugate to a matrix in this form.

One can compute the defining ideal of $\overline{\cO}_{n,\la}$ as the radical of the ideal generated by the ($(n-k) + p^n_{n-d}(\la)+1$)-minors of $X^d$ for $d\geq 1$.  A more explicit description of the defining ideal was conjectured by Eisenbud-Saltman~\cite{Eisenbud-Saltman} and proven by Weyman~\cite{Weyman}, which we state next.
For an integer $m$, let $\bigwedge^m (tI-X)$ be the exterior power of the matrix $tI-X$, where $I$ is the identity matrix. Recall that the entries of the exterior power are the $m\times m$ minors of $tI-X$, each of which is a polynomial in $t$ and the variables $x_{i,j}$. Let $f_d^m$ be the matrix of coefficients of $t^{m-d}$ in $\bigwedge^m(tI-X)$, and let $I(f_d^m)$ be the ideal generated by the entries of $f_d^m$. Note that $f_d^m$ is the same as $\la_d^m$ in \cite{Eisenbud-Saltman}, which is the same as $V_{n-m, d}$ in \cite{Weyman}.
\begin{theorem}[\cite{Weyman}]\label{thm:WeymanTheorem}
The ideal $I(\overline{\cO}_{n,\la'})$ is the sum of the ideals $I(f_d^m)$ for $d$ and $m$ which satisfy $d > m-p_m^n(\la)$.
\end{theorem}

Recall the ideals $I_{n,\la}$ and rings $R_{n,\la}$ introduced in Section~\ref{sec:Intro},
\begin{align}
I_{n,\la} &\coloneqq \langle e_d(S) \st S\subseteq \bx_n,\, d>|S| - p_{|S|}^n(\la)\rangle\\
R_{n,\la} &\coloneqq \bQ[\bx_n]/I_{n,\la}.
\end{align}
We have the following corollary of Theorem~\ref{thm:WeymanTheorem} 

\begin{corollary}\label{cor:WeymanCorollary}
We have an isomorphism of graded rings,
\begin{align}
R_{n,\la} \cong \bQ[\overline{\cO}_{n,\la'}\cap \ft],
\end{align}
where the right-hand side is the coordinate ring of the scheme-theoretic intersection, and the isomorphism is given by mapping $x_{i}$ to $x_{i,i}$.
\end{corollary}

\begin{proof}
By Theorem~\ref{thm:WeymanTheorem},
\begin{align}
\bQ[\overline{\cO}_{n,\la'}\cap \ft] = \frac{\bQ[x_{i,j}]}{I(f_d^m \st d>m-p_m^n(\la)) + I(\ft)}.
\end{align}
Recall that $f_d^m$ is the matrix of coefficients of $t^{m-d}$ in $\bigwedge^m(tI-X)$. Observe that any $m\times m$ minor of $tI-X$ which is not a principal minor is contained in $I(\ft)$. Furthermore, given $A\subseteq [n]$ of size $m$, the principal minor of $tI-X$ with row and column set $A$ is equivalent modulo $I(\ft)$ to 
\begin{align}
\sum_{d = 0}^m (-1)^d \,e_d(\{x_{i,i} \st i\in A\})\, t^{m-d}.
\end{align}
Hence, we have
\begin{align}
\bQ[\overline{\cO}_{n,\la'}\cap \ft] \cong \frac{\bQ[x_{i,i}]}{\langle e_d(\{x_{i,i}\st i\in A\})\st A\subseteq [n], \,d > |A|-p_{|A|}^n(\la)\rangle}.
\end{align}
Identifying $x_{i,i}$ with $x_i$ completes the proof.
\end{proof}

Recall that $\beta^j(\la) = (0,1,\dots, \la_j'-1)$ for $1\leq j\leq \la_1$. Define an \emph{$(n,\la)$-staircase} to be a shuffle of $\beta^1(\la)$, $\beta^2(\la)$, \dots, $\beta^{\la_1}(\la)$, and $(\infty^{n-k})$, which is the sequence consisting of $\infty$ repeated $n-k$ many times. Given $\alpha$ a weak composition of length $n$ and $\beta$ an $(n,\la)$-staircase, we say $\alpha$ \emph{is contained in} $\beta$  if $\alpha_i\leq \beta_i$ for all $i$. Let $\cC_{n,\la}$ be the set of weak compositions with finite integer parts
\begin{align}
\cC_{n,\la} \coloneqq \{\alpha = (\alpha_1,\dots, \alpha_n) \st \alpha \subseteq \beta \text{ for some }(n,\la)\text{-staircase }\beta\}.
\end{align}
For example if $n= 3$ and $\la = (1,1)$, then $\cC_{3,(1,1)}$ is the set of compositions of the form
\begin{align}
(0,0,a), (0,a,0), (a,0,0), (0,1,a),(0,a,1),(a,0,1),
\end{align}
ranging over all integers $0\leq a < \infty$. Let
\begin{align}
\cA_{n,\la} \coloneqq \{\bx_n^\alpha \st \alpha \in \cC_{n,\la}\}.
\end{align}

\begin{theorem}\label{thm:RankBasis}
The set $\cA_{n,\la}$ represents a basis of $R_{n,\la} \cong \bQ[\overline{\cO}_{n,\la'}\cap \ft]$.
\end{theorem}
\begin{proof}
Observe that for $d\geq 0$, the degree $d$ components of $R_{n,\la, d+1}$ and $R_{n,\la}$ coincide, since the $x_i^{d+1}$ generators of $I_{n,\la, d+1}$ are in degree $d+1$. By Theorem~\ref{thm:MonomialBasis}, the set
\begin{align}\label{eq:SubDPlusOne}
\{ \bx_n^\alpha \st \alpha\in \cC_{n,\la, d+1}, \,\alpha_1+\cdots + \alpha_n = d\}
\end{align}
represents a basis of the degree $d$ component of $R_{n,\la, d+1}$, and hence also for the degree $d$ component of $R_{n,\la}$. Observe that the set~\eqref{eq:SubDPlusOne} is equal to
\begin{align}
\{\bx_n^\alpha \st \alpha\in \cC_{n,\la},\, \alpha_1+\cdots + \alpha_n = d\},
\end{align}
which is the subset of $\cA_{n,\la}$ consisting of degree $d$ monomials. Hence, we have that $\cA_{n,\la}$ represents a basis for $R_{n,\la}$.
\end{proof}

Let 
\begin{align}
\fci_{n,\la} &\coloneqq \bigcup_{s\geq \ell(\la)} \fci_{n,\la, s},\\
\FCI_{n,\la} &\coloneqq \bigcup_{s \geq \ell(\la)} \FCI_{n,\la, s},
\end{align}
where we identify $\float\in \fci_{n,\la, s}$ with the extended column-increasing filling in $\fci_{n,\la, s+1}$ obtained by appending an empty $(s+1)$th column to $\float$. We similarly identify each element of $\FCI_{n,\la, s}$ with its counterpart in $\FCI_{n,\la, s+1}$. Observe that for each $\float \in \fci_{n,\la, s}$, the statistic $\inv(\float)$ does not depend on the parameter $s$. Hence, we may consider $\inv$ to be a statistic on elements of $\fci_{n,\la}$. Note that this is not true for the statistic $\dinv$, which is why our formulas below are only stated in terms of $\inv$.

\begin{theorem}\label{thm:RankVarFrob}
For any $k\leq n$ and $\la\vdash k$,
\begin{align}\label{eq:RankMonomialFormula}
\Frobq(R_{n,\la}) = \Frobq(\bQ[\overline{\cO}_{n,\la'} \cap \ft]) &= \sum_{\float \in \fci_{n,\la}} q^{\finv(\float)} \bx^\float.
\end{align}
\end{theorem}

\begin{proof}
The first equality in~\eqref{eq:RankMonomialFormula} follows by Corollary~\ref{cor:WeymanCorollary}, so it suffices to show the left-most side and right-most side of~\eqref{eq:RankMonomialFormula} are equal. Recall that for $d\geq 0$, the degree $d$ components of $R_{n,\la, d+1}$ and $R_{n,\la}$ coincide. Combining this with Theorem~\ref{thm:MonomialTheorem}, we have
\begin{align}\label{eq:DegDRankMonomialFormula}
[q^d] \Frobq(R_{n,\la}) = [q^d] \Frobq(R_{n,\la, d+1}) = \sum_{\substack{\float\in \fci_{n,\la, d+1},\\ \inv(\float)=d}} \bx^\float.
\end{align}

Observe that for $s > d+1$, each element $\float\in \fci_{n,\la, s}\subseteq \fci_{n,\la}$ which is not identified with an element of $\fci_{n,\la, d+1}$ must have a basement cell in some column to the right of column $d+1$. Hence, each of these fillings $\float$ has $\inv(\float) > d$. Therefore, 
\begin{align}\label{eq:OtherDegDFormula}
\sum_{\substack{\float\in \fci_{n,\la, d+1},\\ \inv(\float)=d}} \bx^\float= \sum_{\substack{\float \in \fci_{n,\la},\\ \inv(\float) = d}}  \bx^\float = [q^d]\sum_{\float \in \fci_{n,\la}} q^{\finv(\float)} \bx^\float.
\end{align}
Combining \eqref{eq:DegDRankMonomialFormula} and~\eqref{eq:OtherDegDFormula}, the $q^d$ coefficients of the left-most side and the right-most side of~\eqref{eq:RankMonomialFormula} are equal for all $d\geq 0$, hence~\eqref{eq:RankMonomialFormula} follows.
\end{proof}

We have the following two corollaries of Theorem~\ref{thm:RankVarFrob}. The first follows from Theorem~\ref{thm:RankVarFrob} by the same argument as in the proof of Corollary~\ref{cor:FBasisExpansion}, and the second follows by applying~\eqref{eq:HilbFromFrobq} to $V=R_{n,\la}$.

\begin{corollary}\label{cor:RankVarFrobCor}
We have
\begin{align}
\Frobq(R_{n,\la}) = \Frobq(\bQ[\overline{\cO}_{n,\la'}\cap \ft]) = \sum_{\float\in \FCI_{n,\la}} q^{\finv(\float)} F_{n,\iDes(\rw(\float))}(\bx).
\end{align}
\end{corollary}

\begin{corollary}\label{cor:RankHilb}
We have
\begin{align}
\Hilbq(R_{n,\la}) = \Hilbq(\bQ[\overline{\cO}_{n,\la'}\cap \ft]) = \sum_{\float \in \FCI_{n,\la}} q^{\finv(\float)}.
\end{align}
\end{corollary}

\section{Final remarks}\label{sec:FinalRemarks}

Many of our formulas for $\Frobq(R_{n,\la, s})$ are not stable under taking the limit as $s\to\infty$. This includes the formula in Theorem~\ref{thm:MonomialTheorem} in terms of the $\dinv$ statistic, our Hall-Littlewood formula~\eqref{eq:IntroGradedFrobFormula}, and the formula in Theorem~\ref{thm:RemovingZerosFrobChar}. In fact, the only formulas in this article which are stable under taking the limit $s\to\infty$ are in terms of the $\finv$ statistic on extended column-increasing fillings.

\begin{problem}
Find other formulas for $\Frobq(R_{n,\la, s})$ which yield formulas for $\Frobq(R_{n,\la})$ upon letting $s$ approach infinity.
\end{problem}

As we mentioned in the introduction, a Gr\"obner basis of the ideal $I_{n,k}$ in terms of Demazure characters was given by Haglund, Rhoades, and Shimozono~\cite{HRS1}. To our knowledge, a complete description of Gr\"obner basis of the ideals $I_\la$ is not known. 
\begin{problem}
Find a Gr\"obner basis for the ideals $I_{n,\la, s}$ in terms of a generalization of Demazure characters. Computer computations suggest that a Gr\"obner basis for $I_{n,\la, s}$ in terms of Demazure characters exists for $n\leq 6$ and $s=\ell(\la)$. However, this does not appear to hold for $n>6$.
\end{problem}

Recall that the family of rings $R_{n,\la, s}$ generalize the cohomology rings of Springer fibers $R_\la$ and the generalized coinvariant rings $R_{n,k}$ of Haglund-Rhoades-Shimozono~\cite{HRS1}. Pawlowski and Rhoades~\cite{Pawlowski-Rhoades} proved that $R_{n,k}$ is isomorphic to the rational cohomology ring of a space of spanning line arrangements in $\bC^k$.
\begin{problem}
Find a variety $X_{n,\la, s}$ whose rational cohomology ring is isomorphic to $R_{n,\la, s}$.
\end{problem}

Gillespie and Rhoades~\cite{Gillespie-Rhoades} have recently found a basis of the ring $R_{n,(n-1),s}$ that respects the decomposition of $R_{n,(n-1),s}$ into irreducible representations of $S_n$. Such a basis is called a \emph{higher Specht basis}.
\begin{problem}
Find \emph{higher Specht bases} for the rings $R_{n,\la,s}$ and $R_{n,\la}$ in general.
\end{problem}

	\section{Acknowledgements.}
The author would like to thank Sara Billey for many helpful comments which greatly improved the paper. Special thanks to Brendon Rhoades, Tianyi Yu, and Zehong Zhao for sharing notes on their coinversion statistic during the preparation of this article, which led to simpler proofs and simpler formulas for our results on rank varieties. Special thanks also to Alex Woo for interesting conversations on generalizing Springer fibers. Thanks to Jarod Alper, Sergi Elizalde, Erik Carlsson, Adriano Garsia, Maria Monks Gillespie, James Haglund, Zachary Hamaker, Monty McGovern, Brendan Pawlowski, Martha Precup, Mark Shimozono, Nathan Williams, and Andrew Wilson for helpful conversations and useful references. Finally, thanks to an anonymous referee for helpful comments.

\bibliographystyle{amsplain}
\bibliography{Springer}

\providecommand{\bysame}{\leavevmode\hbox to3em{\hrulefill}\thinspace}
\providecommand{\MR}{\relax\ifhmode\unskip\space\fi MR }
\providecommand{\MRhref}[2]{%
  \href{http://www.ams.org/mathscinet-getitem?mr=#1}{#2}
}
\providecommand{\href}[2]{#2}
\begin{thebibliography}{10}

\bibitem{Billey-Coskun}
Sara Billey and Izzet Coskun, \emph{Singularities of generalized {R}ichardson
  varieties}, Comm. Algebra \textbf{40} (2012), no.~4, 1466--1495. \MR{2912998}

\bibitem{CG}
Neil Chriss and Victor Ginzburg, \emph{Representation theory and complex
  geometry}, Birkh\"{a}user Boston, Inc., Boston, MA, 1997. \MR{1433132}

\bibitem{Church-Ellenberg-Farb}
Thomas Church, Jordan~S. Ellenberg, and Benson Farb, \emph{F{I}-modules and
  stability for representations of symmetric groups}, Duke Math. J.
  \textbf{164} (2015), no.~9, 1833--1910. \MR{3357185}

\bibitem{CoxLittleOshea}
David Cox, John Little, and Donal O'Shea, \emph{Ideals, varieties, and
  algorithms}, third ed., Undergraduate Texts in Mathematics, Springer, New
  York, 2007. \MR{2290010}

\bibitem{dCP}
Corrado De~Concini and Claudio Procesi, \emph{Symmetric functions, conjugacy
  classes and the flag variety}, Invent. Math. \textbf{64} (1981), no.~2,
  203--219. \MR{629470}

\bibitem{Eisenbud}
David Eisenbud, \emph{Commutative algebra}, Graduate Texts in Mathematics, vol.
  150, Springer-Verlag, New York, 1995. \MR{1322960}

\bibitem{Eisenbud-Saltman}
David Eisenbud and David Saltman, \emph{Rank varieties of matrices},
  Commutative Algebra (1989), 173--212.

\bibitem{Fresse-Melnikov}
Lucas Fresse and Anna Melnikov, \emph{On the singularity of the irreducible
  components of a {S}pringer fiber in {$\mathfrak{sl}_n$}}, Selecta Math.
  (N.S.) \textbf{16} (2010), no.~3, 393--418. \MR{2734337}

\bibitem{Fung}
Francis Y.~C. Fung, \emph{On the topology of components of some {S}pringer
  fibers and their relation to {K}azhdan-{L}usztig theory}, Adv. Math.
  \textbf{178} (2003), no.~2, 244--276. \MR{1994220}

\bibitem{Garsia-Procesi}
A.~M. Garsia and C.~Procesi, \emph{On certain graded {$S_n$}-modules and the
  {$q$}-{K}ostka polynomials}, Adv. Math. \textbf{94} (1992), no.~1, 82--138.
  \MR{1168926}

\bibitem{GesselQSym}
Ira~M. Gessel, \emph{Multipartite {$P$}-partitions and inner products of skew
  {S}chur functions}, Combinatorics and algebra ({B}oulder, {C}olo., 1983),
  Contemp. Math., vol.~34, Amer. Math. Soc., Providence, RI, 1984,
  pp.~289--317. \MR{777705}

\bibitem{Gillespie-Rhoades}
M.~Gillespie and B.~Rhoades, \emph{Higher {S}pecht bases for generalizations of
  the coinvariant ring}, Ann. Comb. \textbf{25} (2021), no.~1, 51--77.
  \MR{4233510}

\bibitem{Griffin-GPMod-FPSAC}
Sean~T. Griffin, \emph{Ordered set partitions, {T}anisaki ideals, and rank
  varieties}, S\'{e}m. Lothar. Combin. \textbf{84B} (2020), Art. 90, 12.

\bibitem{HHL}
J.~Haglund, M.~Haiman, and N.~Loehr, \emph{A combinatorial formula for
  {M}acdonald polynomials}, J. Amer. Math. Soc. \textbf{18} (2005), no.~3,
  735--761. \MR{2138143}

\bibitem{HHLplus}
J.~Haglund, M.~Haiman, N.~Loehr, J.~B. Remmel, and A.~Ulyanov, \emph{A
  combinatorial formula for the character of the diagonal coinvariants}, Duke
  Math. J. \textbf{126} (2005), no.~2, 195--232. \MR{2115257}

\bibitem{HRW}
J.~Haglund, J.~B. Remmel, and A.~T. Wilson, \emph{The {D}elta {C}onjecture},
  Trans. Amer. Math. Soc. \textbf{370} (2018), no.~6, 4029--4057. \MR{3811519}

\bibitem{HRS1}
James Haglund, Brendon Rhoades, and Mark Shimozono, \emph{Ordered set
  partitions, generalized coinvariant algebras, and the {D}elta {C}onjecture},
  Adv. Math. \textbf{329} (2018), 851--915. \MR{3783430}

\bibitem{Hotta-Springer}
R.~Hotta and T.~A. Springer, \emph{A specialization theorem for certain {W}eyl
  group representations and an application to the {G}reen polynomials of
  unitary groups}, Invent. Math. \textbf{41} (1977), no.~2, 113--127.
  \MR{486164}

\bibitem{Kostant}
Bertram Kostant, \emph{Lie group representations on polynomial rings}, Amer. J.
  Math. \textbf{85} (1963), 327--404. \MR{158024}

\bibitem{Kraft}
Hanspeter Kraft, \emph{Conjugacy classes and {W}eyl group representations},
  Young tableaux and {S}chur functors in algebra and geometry ({T}oru\'{n},
  1980), Ast\'{e}risque, vol.~87, Soc. Math. France, Paris, 1981, pp.~191--205.
  \MR{646820}

\bibitem{LLT}
Alain Lascoux, Bernard Leclerc, and Jean-Yves Thibon, \emph{Ribbon tableaux,
  {H}all-{L}ittlewood functions, quantum affine algebras, and unipotent
  varieties}, J. Math. Phys. \textbf{38} (1997), no.~2, 1041--1068.
  \MR{1434225}

\bibitem{Macdonald}
I.~G. Macdonald, \emph{Symmetric functions and {H}all polynomials}, second ed.,
  Oxford Mathematical Monographs, The Clarendon Press, Oxford University Press,
  New York, 1995. \MR{1354144}

\bibitem{GillespieThesis}
Maria Monks~Gillespie, \emph{A combinatorial approach to the {$q,t$}-symmetry
  relation in {M}acdonald polynomials}, Electron. J. Combin. \textbf{23}
  (2016), no.~2, P2.38, 64. \MR{3512660}

\bibitem{Pawlowski-Rhoades}
Brendan Pawlowski and Brendon Rhoades, \emph{A flag variety for the {D}elta
  {C}onjecture}, Trans. Amer. Math. Soc. \textbf{372} (2019), no.~11,
  8195--8248. \MR{4029695}

\bibitem{Rhoades-Wilson-Stirling}
Brendon Rhoades and Andrew~Timothy Wilson, \emph{Line configurations and
  {$r$}-{S}tirling partitions}, J. Comb. \textbf{10} (2019), no.~3, 411--431.
  \MR{3960509}

\bibitem{Rhoades-Yu-Zhao}
Brendon Rhoades, Tianyi Yu, and Zehong Zhao, \emph{Harmonic bases for
  generalized coinvariant algebras}, Electron. J. Combin. \textbf{27} (2020),
  no.~4, Paper No. 4.16, 23. \MR{4245191}

\bibitem{Spaltenstein}
N.~Spaltenstein, \emph{The fixed point set of a unipotent transformation on the
  flag manifold}, Nederl. Akad. Wetensch. Proc. Ser. A, 79 Indag. Math.
  \textbf{38} (1976), no.~5, 452--456. \MR{0485901}

\bibitem{Springer-TrigSum}
T.~A. Springer, \emph{Trigonometric sums, {G}reen functions of finite groups
  and representations of {W}eyl groups}, Invent. Math. \textbf{36} (1976),
  173--207. \MR{0442103}

\bibitem{Springer-WeylGrpReps}
\bysame, \emph{A construction of representations of {W}eyl groups}, Invent.
  Math. \textbf{44} (1978), no.~3, 279--293. \MR{0491988}

\bibitem{Tanisaki}
Toshiyuki Tanisaki, \emph{Defining ideals of the closures of the conjugacy
  classes and representations of the {W}eyl groups}, T\^{o}hoku Math. J. (2)
  \textbf{34} (1982), no.~4, 575--585. \MR{685425}

\bibitem{Tymoczko-LinearConditions}
Julianna Tymoczko, \emph{Linear conditions imposed on flag varieties}, Amer. J.
  Math. \textbf{128} (2006), no.~6, 1587--1604. \MR{2275912}

\bibitem{Tymoczko}
\bysame, \emph{The geometry and combinatorics of {S}pringer fibers}, Around
  {L}anglands correspondences, Contemp. Math., vol. 691, Amer. Math. Soc.,
  Providence, RI, 2017, pp.~359--376. \MR{3666060}

\bibitem{Vargas}
J.~A. Vargas, \emph{Fixed points under the action of unipotent elements of
  {${\rm SL}_{n}$} in the flag variety}, Bol. Soc. Mat. Mexicana (2)
  \textbf{24} (1979), no.~1, 1--14. \MR{579665}

\bibitem{Weyman}
J.~Weyman, \emph{The equations of conjugacy classes of nilpotent matrices},
  Invent. Math. \textbf{98} (1989), no.~2, 229--245.

\bibitem{Wilson}
Andrew~Timothy Wilson, \emph{An extension of {M}ac{M}ahon's equidistribution
  theorem to ordered multiset partitions}, Electron. J. Combin. \textbf{23}
  (2016), no.~1, P1.5, 21. \MR{3484710}

\end{thebibliography}

\end{document}